\documentclass[aap]{imsart}

\RequirePackage{amsthm,amsmath,amsfonts,amssymb}
\RequirePackage[numbers,sort&compress]{natbib}
\RequirePackage[colorlinks,citecolor=blue,urlcolor=blue]{hyperref}
\RequirePackage{graphicx}

\usepackage{mathrsfs,color,tikz,extarrows,bbm,pdfpages}

\startlocaldefs
\numberwithin{equation}{section}
\theoremstyle{plain}

\newtheorem{theorem}{Theorem}[section]
\newtheorem{lemma}[theorem]{Lemma}

\newtheorem{assump}{Assumption}[section]
\newtheorem{remark}[theorem]{Remark}
\newtheorem{corollary}[theorem]{Corollary}
\newtheorem{proposition}[theorem]{Proposition}
\theoremstyle{remark}
\newtheorem{definition}[theorem]{Definition}

\endlocaldefs

\newcommand{\bbE}{{\ensuremath{\mathbbm E}} }

\newcommand{\bbN}{{\ensuremath{\mathbbm N}} }

\newcommand{\bbP}{{\ensuremath{\mathbbm P}} }

\newcommand{\bbR}{{\ensuremath{\mathbbm R}} }

\newcommand{\bbT}{{\ensuremath{\mathbbm T}} }

\newcommand{\bbZ}{{\ensuremath{\mathbbm Z}} }

\newcommand{\cA}{{\ensuremath{\mathcal A}} }
\newcommand{\cB}{{\ensuremath{\mathcal B}} }
\newcommand{\cC}{{\ensuremath{\mathcal C}} }
\newcommand{\cD}{{\ensuremath{\mathcal D}} }
\newcommand{\cE}{{\ensuremath{\mathcal E}} }

\newcommand{\cI}{{\ensuremath{\mathcal I}} }
\newcommand{\cJ}{{\ensuremath{\mathcal J}} }

\newcommand{\cM}{{\ensuremath{\mathcal M}} }
\newcommand{\cN}{{\ensuremath{\mathcal N}} }

\newcommand{\cP}{{\ensuremath{\mathcal P}} }
\newcommand{\cQ}{{\ensuremath{\mathcal Q}} }
\newcommand{\cR}{{\ensuremath{\mathcal R}} }
\newcommand{\cS}{{\ensuremath{\mathcal S}} }
\newcommand{\cT}{{\ensuremath{\mathcal T}} }
\newcommand{\cU}{{\ensuremath{\mathcal U}} }
\newcommand{\cV}{{\ensuremath{\mathcal V}} }

\newcommand{\cY}{{\ensuremath{\mathcal Y}} }
\newcommand{\cZ}{{\ensuremath{\mathcal Z}} }

\newcommand{\sE}{\mathscr{E}}

\newcommand{\sL}{\mathscr{L}}

\newcommand{\sR}{\mathscr{R}}

\newcommand{\bE}{{\ensuremath{\mathbf E}} }

\newcommand{\bI}{{\ensuremath{\mathbf I}} }

\newcommand{\bM}{{\ensuremath{\mathbf M}} }

\newcommand{\bP}{{\ensuremath{\mathbf P}} }

\newcommand{\bU}{{\ensuremath{\mathbf U}} }
\newcommand{\bV}{{\ensuremath{\mathbf V}} }
\newcommand{\bW}{{\ensuremath{\mathbf W}} }

\newcommand{\ga}{\alpha}
\newcommand{\gb}{\beta}

\newcommand{\go}{\omega}

\renewcommand{\epsilon}{\varepsilon}
\newcommand{\widebar}{\overline}

\newcommand{\bbmone}{\mathbbm{1}}

\newcommand{\R}{\mathbb{R}}
\newcommand{\Z}{\mathbb{Z}}
\newcommand{\N}{\mathbb{N}}

\renewcommand{\tilde}{\widetilde}

\newcommand{\ind}{\mathbbm{1}}
\newcommand{\dd}{{\ensuremath{\mathrm d}} }
\newcommand{\var}{{\rm Var}}

\newcommand{\rmi}{\mathrm i}
\newcommand{\rmj}{\mathrm j}
\newcommand{\rma}{\mathrm a}
\newcommand{\rmb}{\mathrm b}
\newcommand{\rmn}{\mathbf n}

\newcommand{\Anotriple}{\cA_\epsilon^{\text{\rm (no triple)}}}
\newcommand{\vAnotriple}{\vec{\mathbf{\cA}}_\epsilon^{\text{\rm (no triple)}}}
\newcommand{\Znotriple}{\cZ_{N,\epsilon}^{\text{\rm (no triple)}}}
\newcommand{\Lcg}{\sL_\epsilon^{\text{\rm (cg)}}}
\newcommand{\Zcg}{\cZ_{N,\epsilon}^{\text{\rm (cg)}}}

\definecolor{rosso}{RGB}{206,43,55}
\definecolor{blu}{RGB}{140,204,171}
\definecolor{verde}{RGB}{0,146,70}
\definecolor{arancione}{RGB}{255,102,51}
\definecolor{viola}{RGB}{255,0,255}

\renewcommand{\hat}{\widehat}

\newcommand{\dis}{\displaystyle}
\newcommand{\ba}{\begin{array}}
	\newcommand{\ea}{\end{array}}

\newcommand{\eps}{\epsilon}
\newcommand{\ov}{\overline}

\begin{document}

\begin{frontmatter}
\title{The critical disordered pinning measure}
\runtitle{The critical disordered pinning measure}

\begin{aug}
\author[A]{\fnms{Ran}~\snm{Wei}\ead[label=e1]{ran.wei@xjtlu.edu.cn}},
\and\author[B]{\fnms{Jinjiong}~\snm{Yu}\ead[label=e2]{jjyu@sfs.ecnu.edu.cn}}

\address[A]{Department of Financial and Actuarial Mathematics, Xi'an Jiaotong-Liverpool University, 111 Ren'ai Road, Suzhou, Jiangsu Province, China, 215123 \printead[presep={,\ }]{e1}}

\address[B]{KLATASDS-MOE, School of Statistics, East China Normal University, 3663 North Zhongshan Road, Shanghai 200062, China \printead[presep={,\ }]{e2}}
\end{aug}

\begin{abstract}
In this paper, we study a disordered pinning model induced by a random walk whose increments have a finite $(2+\kappa)$-th moment for some $\kappa>0$. It is known that this model is marginally relevant, and moreover, it undergoes a phase transition in an intermediate disorder regime. We show that, in the critical window, the point-to-point partition functions converge to a unique limiting random measure, which we call the critical disordered pinning measure. We also obtain an analogous result for a continuous counterpart to the pinning model, which is closely related to two other models: one is a critical stochastic Volterra equation that gives rise to a rough volatility model, and the other is a critical stochastic heat equation with multiplicative noise that is white in time and delta in space.
\end{abstract}



\end{frontmatter}

\tableofcontents

\section{Introduction}\label{S1}

\subsection{Background}
In this paper, we study the \textit{disordered pinning model} (abbreviated as the \textit{pinning model} hereafter), which is also known as a model of \textit{pinning on a defect line} (see, e.g., the monographs of Giacomin \cite{G07,G11} and den Hollander \cite{dH09} for a comprehensive introduction). This model holds significance in statistical mechanics, serving as a disordered system that is both interesting enough to have a phase transition and simple enough for extensive mathematical analysis.

The model is comprised of a renewal process and a random field on $\Z$, representing the polymer chain and disorders on the membrane, respectively. Let $\tau=\{\tau_0,\tau_1,\tau_2,\cdots\}$, where $\tau_0=0$ and $(\tau_n-\tau_{n-1})_{n\geq1}$ are i.i.d.\ $\bbN$-valued random variables, be the renewal process with probability law~$\bP$. 
Let the disorder $\omega:=(\omega_x)_{x\in\bbZ}$ be a family of i.i.d.\ random variables with probability law~$\bbP$. Expectations with respect to $\bP$ and $\bbP$ are denoted by $\bE$ and $\bbE$, respectively. 
We assume independence between the renewal process $\tau$ and the disorder $\omega$, with $\omega$ 
satisfying
\begin{equation}\label{assump:omega}
\bbE[\omega_0]=0,~~\var(\omega_0)=1,~~\text{and}~~ \lambda(\beta):=\log\bbE[\exp(\beta\omega_0)]<+\infty, ~~\forall\beta\in(-\beta_0,\beta_0),
\end{equation}
where $\beta_0$ is some positive constant.
The model is then defined via a Gibbs transform
\begin{equation}\label{def:renewalpin}
\frac{\dd\bP_{N,\beta}^{\omega,h}}{\dd\bP}(\tau):=\frac{1}{Z_{N,\beta}^{\omega,h}}\exp\Big(\sum_{n=1}^{N}(\beta\omega_n+h-\lambda(\beta))\ind_{\{n\in\tau\}}\Big),
\end{equation}
where $\beta>0$ is the inverse temperature, $h\in\bbR$ denotes an external field, $n\in\tau$ means that $n$ is a renewal time, and
\begin{equation}\label{def:renewpartition}
Z_{N,\beta}^{\omega,h}:=\bE\Big[\exp\Big(\sum_{n=1}^{N}(\beta\omega_n+h-\lambda(\beta))\ind_{\{n\in\tau\}}\Big)\Big]
\end{equation}
is the partition function which makes $\bP_{N,\beta}^{\omega,h}$ a (random) probability measure.
When $\beta=0$, signifying the absence of disorder, the resulting model $\bP_{N}^{h}:=\bP_{N,\beta=0}^{\omega,h}$ is referred to as the \textit{pure pinning model}.

A particularly interesting setup is to assume the renewal $\tau$ is recurrent, i.e., $\bP(\tau_1<\infty)=1$, and has a distribution
\begin{equation}\label{renew_kernel}
K(n):=\bP(\tau_1=n)=\frac{L(n)}{n^{1+\alpha}}
\end{equation}
for some $\ga\geq0$ and some slowly varying function $L(\cdot)$. Examples of slowly varying functions include $L(x)\equiv C$ and $L(x)=\log x$ (see \cite{BGT89} for details). For different values of $\alpha$, determining whether the disordered measure $\bP_{N,\beta}^{\omega,h}$ significantly differs from the pure measure $\bP_{N}^{h}$ for arbitrarily small $\beta>0$ becomes a subtle problem, which we will elucidate below.

We introduce the {\em free energy} of the system, defined as
\begin{equation}\label{FE}
F(\beta,h):=\lim\limits_{N\to\infty}\frac1N\log Z_{N,\beta}^{\omega,h}\xlongequal{\bbP\text{-a.s.}}\lim\limits_{N\to\infty}\frac1N\bbE\big[\log Z_{N,\beta}^{\omega,h}\big].
\end{equation}
The existence and non-randomness of the limit are proven in \cite{G11}. Moreover, it shows that for any $\beta\geq0$ fixed, $F(\beta,h)$ is non-negative, convex, and non-decreasing as a function of $h$. Thus, the system undergoes a phase transition in $h$ at
\begin{equation}\label{h_crit}
h_c(\beta):=\inf\{h: F(\beta,h)>0\}.
\end{equation}
The phase $\{(\beta,h):F(\beta,h)>0\}$ is called the \textit{localized regime}, wherein, as $N\to\infty$, the number of renewal times under $\bP_{N,\beta}^{\omega,h}$ is proportional to $N$. This signifies that the energy gain dominates the entropy cost for the renewal process to visit the origin with atypical high frequency. On the other hand, the phase $\{(\beta,h):F(\beta,h)=0\}$ is called the \textit{delocalized regime}, where the number of renewal times is $o(N)$, as the energy gain at renewal returns is less attractive to the renewal process.

It can be seen by Jensen's inequality that 
\begin{equation}\label{Jensen}
\bbE\big[\log Z_{N,\beta}^{\omega,h}\big]\leq\log\bbE\big[Z_{N,\beta}^{\omega,h}\big]\quad\text{with}\quad\bbE\big[Z_{N,\beta}^{\omega,h}\big]=\bE\Big[\exp\Big({\sum_{n=0}^N h\ind_{\{n\in\tau\}}}\Big)\Big]= Z_N^h,
\end{equation}
where $Z_N^h$ is the partition function of the pure pinning model ($\beta=0$). Note that the pure free energy
$F(h)=F(0,h)=\lim_{N\to\infty}\frac1N\log Z_N^h$
is deterministic and has an explicit form. It is well established (see \cite{G11}) that 
$F(h)=0$ if and only if $h\leq0$, i.e., $h_c:=\inf\{h\in\bbR: F(h)>0\}=0$, and moreover,
\begin{equation}\label{eq:homo_F}
\lim\limits_{h\downarrow h_c}\frac{\log F(h)}{\log(h-h_c)}=1\vee\ga^{-1}.
\end{equation}
From \eqref{Jensen}, it is evident that $F(\beta,h)\leq F(h)$. Hence, we must have $h_c(\beta)\geq h_c(0)=0$. 

The key question for the pinning model is: does arbitrarily small disorder ($\beta>0$) qualitatively change the critical behavior of the system? In particular, one may ask whether for arbitrarily small $\beta>0$, the {\em critical exponent} for $F(\beta,h)$ differs from $1\vee\ga^{-1}$ that appears in \eqref{eq:homo_F} for the pure model.
It turns out that this can be equivalently characterised in terms of {\em critical point shift}, namely whether $h_c(\beta)>h_c(0)=0$ for arbitrarily small $\beta>0$, or there exists some $\beta^\ast>0$ such that $h_c(\beta)=0$ for $\beta\in[0,\beta^\ast]$.
The answer hinges on the exponent $\alpha$ in \eqref{renew_kernel}. It is predicted by the {\em Harris criterion} \cite{H74} that the disorder is relevant (resp.\ irrelevant) if $\alpha>1/2$ (resp.\ $\alpha<1/2$). Indeed, it is now known that

$\bullet$ For $\ga<\frac12$, by \cite{A08,T08,L10} the critical exponents for both the disordered model and the pure model are the same, and $h_c(\beta)\equiv0=h_c$ for $\beta>0$ small enough. This regime is known as the \textit{disorder irrelevant regime}.

$\bullet$ For $\ga>\frac12$, it was shown by Giacomin and Toninelli \cite{GT06} that the critical exponent for the disordered model is at least 2, which is strictly larger than $1\vee\ga^{-1}$ in \eqref{eq:homo_F}. Furthermore, \cite{AZ09,DGLT09} established that $h_c(\beta)>0$ for any $\beta>0$, and later \cite{BCPSZ14,CTT17} obtained the sharp asymptotics for $h_c(\beta)$ as $\beta\to0$. This is called the \textit{disorder relevant regime}.

$\bullet$ For $\ga=\frac12$, the situation becomes intricate, and it is out of the scope of the Harris criterion. A crucial quantity is $R_N:=\sum_{n=1}^{N}\bP(n\in\tau)^2$, the expected number of overlaps for two i.i.d.\ copies of the renewal process $\tau$. It is shown in \cite{BL18} that $h_c(\beta)>0$ for any $\beta>0$ (and hence the disorder is relevant) if and only if $R_N\sim\sum_{n=1}^N 1/(nL(n)^2)\to+\infty$. In this scenario, it is called the \textit{marginally relevant regime}, and \cite{BL18} provides the sharp asymptotics for $h_c(\beta)$ as $\beta\to0$.

In a nutshell, for $\alpha\geq\frac12$ and when the disorder is relevant, we observe significantly different behavior of the pinning model near the critical curve $h=h_c(\beta)$ at $(\beta,h_c(\beta))=(0,0)$. The disordered system and the pure system belong to two \textit{universality classes} ({\em cf.} \cite[Section~1.3]{CSZ16}): it is believed that there are different universal scaling limits for the two systems, where neither the disorder distribution nor the details of the renewal process are relevant (as long as some mild conditions hold, e.g., \eqref{assump:omega} holds and the exponent $\alpha$ is fixed).
%
%

It is then interesting to study the interpolation of the two universalities.
In view of the partition function, as per \eqref{FE}, in the localized regime, $Z_{N,\beta}^{\omega,h}$ exhibits exponential growth, determined by the exponent $N F(\beta,h)$. This exponent increases as $N$ tends to infinity and decreases as both $\beta$ and $h$ approach 0. Conversely, in the delocalized regime, $Z_{N,\beta}^{\omega,h}$ demonstrates sub-exponential growth, with $Z_{N,0}^{\omega,0} \equiv 1$. 
A natural question about interpolating is: if we systematically tune down $\beta := \beta_N \downarrow 0$ and $h := h_N \downarrow 0$ at specific rates as $N \to \infty$, can we ascertain a non-trivial limit for $Z_{N,\beta_N}^{\omega,h_N}$? The affirmative answer implies the existence of a limit, termed the {\em weak coupling limit}, situating the model within the {\em  intermediate disorder regime}. The inception of exploring the intermediate disorder regime traces back to \cite{AKQ14}, focusing on the $1+1$ directed polymer model. Subsequent investigations into weak coupling limits include disorder relevant models such as the 2d random field Ising model, the pinning model with $\alpha>1/2$ in \cite{CSZ16}, and marginally relevant models such as the 2d stochastic heat equation, the $1+2$ directed polymer model, and the pinning model with $\alpha=1/2$ in \cite{CSZ17b}.

For the pinning model with $\alpha>\frac12$, a non-trivial weak coupling limit emerges through the choice of $\beta_N=\hat{\beta}\cP_1^\ga(N)$ and $h_N=\hat{h}\cP_2^\ga(N)$ where $\hat\beta>0,\hat h\in\bbR$, and $\cP_1^\ga(\cdot)$ and $\cP_2^\ga(\cdot)$, both vanishing at infinity, are two regularly varying functions whose exponents depending on $\ga$ (see \cite{CSZ16} for details). However, the scenario becomes more sophisticated in the marginal case $\alpha=\frac{1}{2}$, as illustrated in \cite{CSZ17b}. The appropriate choice is $h=0$ and $\beta_N=\hat{\beta}/\sqrt{R_N}$, where $R_N=\sum_{n=1}^N\bP(n\in\tau)^2\to+\infty$ as above and is slowly varying at infinity. Furthermore, the constant $\hat{\beta}$ plays a pivotal role. If $\hat{\beta}\in(0,1)$, then the weak coupling limit for $Z_{N,\beta_N}^{\omega,0}$ is log-normal, characterizing a subcritical phase. For the critical phase $\hat\beta=1$ and the supercritical phase $\hat{\beta}>1$, $Z_{N,\beta_N}^{\omega,0}\to0$ in probability as $N\to\infty$.

While the behavior of the partition function in the supercritical phase remains mysterious, a recent breakthrough by Caravenna, Sun, and Zygouras in \cite{CSZ21} concerning the marginally relevant directed polymer in two dimension sheds light on the critical phase. They viewed the random field of point-to-point partition functions as a process of random measures on $\bbR^2\times\bbR^2$, and established the existence of a universal scaling limit known as the {\em critical 2d stochastic heat flow}. Moreover, the limit can be interpreted as the solution of the stochastic heat equation in the critical dimension $d=2$ and in the critical window, which lies beyond existing solution theories via regularity structures by Hairer \cite{H14} and paracontrolled distributions by Gubinelli et.\ al.\ \cite{GIP15} for subcritical singular SPDEs. Motivated by their work, our study in this paper is centered on the critical phase of the marginally relevant pinning model (Subsection~\ref{S:pin}), as well as a rough volatility model characterised by the critical stochastic Volterra equation (Subsection~\ref{S:sve}).

\subsection{The pinning model and weak coupling limit}\label{S:pin}
Instead of the most general setting \eqref{def:renewalpin}, we focus on a specific class of pinning models associated to a random walk $S:=(S_n)_{n\geq0}$ on $\Z$ due to technical limitation. To be specific, our proof requires that the renewal process is associated to a Markov process with sharp enough local limit theorems (see Remark \ref{rmk:rw_assump}). Noting that the times $S$ visiting $0$ give rise to a renewal process, it naturally induces a pinning model. Denote the probability law of $S$ by $\bP^x$ if $\bP^x(S_0=x)=1$, indicating that the starting point of $S$ is $x$. The expectation with respect to $\bP^x$ is denoted by $\bE^x$. When $x=0$, we omit the superscripts in $\bP$ and $\bE$ to lighten the notation. The distribution of the induced renewal process is then given by
\begin{equation}\label{eq:RW_renew}
K(n)=\bP(\tau_1=n)=\bP(S_n=0, S_k\neq0, 1\leq k\leq n-1)\quad\text{and}\quad\bP(n\in\tau)=\bP(S_n=0).
\end{equation}
To avoid complexities related to periodicity, $S$ is assumed to be irreducible and aperiodic throughout the paper. Additionally, we assume the following for $S$:
\begin{equation}\label{assump:RW}
\bE[S_1]=0,\quad\var(S_1)=1,\quad\text{and}\quad\bE[|S_1|^{2+\kappa}]<+\infty~\text{for some}~\kappa>0,
\end{equation}
where $\kappa$ could be arbitrarily small. The assumption on the finite $(2+\kappa)$-th moment is due to technical reasons (see Remark \ref{rmk:rw_assump}).

It is known (see \cite{U11}) that for a random walk on $\bbZ$ with mean $0$ and variance $1$, the induced renewal process $\tau$ satisfies $K(n)\sim1/(\sqrt{2\pi}\hspace{2pt}n^{3/2})$, which falls within the framework of the marginally relevant pinning model, since $\alpha=\frac12$ in \eqref{renew_kernel}. Furthermore, $\bP(S_n=0)\sim1/\sqrt{2\pi n}$ by the local limit theorem, and
\begin{equation}\label{def:overlap}
R_N=\sum_{n=1}^{N}\bP(n\in\tau)^2=\sum_{n=1}^{N}\bP(S_n=0)^2\sim\frac{\log N}{2\pi}\to\infty, \quad\text{as~} N\to\infty.
\end{equation}

The induced pinning model is defined by
\begin{equation}\label{def:pin}
\frac{\dd\bP_{N,\beta_N}^\omega}{\dd\bP}(S):=\frac{1}{Z_{N,\beta_N}^\omega}\exp\Big(\sum_{n=1}^{N}(\beta_N\omega_n-\lambda(\beta_N))\ind_{\{S_n=0\}}\Big),
\end{equation}
where the partition function is
\begin{equation}\label{def:partition}
Z_{N,\beta_N}^\omega:=\bE\Big[\exp\Big(\sum_{n=1}^{N}(\beta_N\omega_n-\lambda(\beta_N))\ind_{\{S_n=0\}}\Big)\Big].
\end{equation}
Note that as in \cite{CSZ17b}, here we set the external field $h\equiv0$ to study the intermediate disorder regime, which is for simplicity and not essential. We are satisfied with this setting and will briefly discuss nonzero choices of $h$ later in Remark \ref{rmk:h}.

Our goal is to study the weak coupling limit of $Z_{N,\beta_N}^\omega$ at the \textit{critical temperature} $\beta_{N,c}:=1/\sqrt{R_N}\sim\sqrt{2\pi}/\sqrt{\log N}$. Recall that $Z_{N,\beta_{N,c}}^\omega$ converges in probability to the trivial limit 0 as $N\to\infty$.
The right way to study the weak coupling limit, in the same spirit as \cite{CSZ21}, is to introduce a random measure related to $Z_{N,\beta_{N,c}}^\omega$ and consider its scaling limit. To this end, for given starting and end times $M\leq K$, we introduce a point-to-point pinning partition function
\begin{equation}\label{def:partition2}
Z_{M,K}^{\omega,\beta_N}:=\bE\Big[\exp\Big(\sum\limits_{n=M}^{K}(\beta_N\omega_n-\lambda(\beta_N))\bbmone_{\{S_n=0\}}\Big)\bbmone_{\{S_K=0\}}\Big|S_M=0\Big].
\end{equation}
Note that by translation invariance of the random walk, the partition function is actually well-defined for all integers $M\leq K$.  In contrast to the normal pinning partition function counting disorders from time $M+1$, in (\ref{def:partition2}) we add the disorder at the starting time $M$ for computational convenience later on (see \eqref{def:dpintegral4}-\eqref{def:pinint}). Note that such modification does not have substantial impact on the results.

We can then define a locally finite random measure on the half plane $\bbR^2_{\leq}:=\{(s,t):-\infty<s\leq t<\infty\}$ through (\ref{def:partition2}) as
\begin{equation}\label{def:pinmeasure-}
\cZ_{N}^{\beta_N}(\dd s,\dd t):=\sqrt{N}Z_{\lfloor sN\rfloor,\lfloor tN\rfloor}^{\go,\beta_N}\dd s\dd t,
\end{equation}
where the choice of the scaling factor $\sqrt{N}$ can be seen from the local limit theorem
\begin{equation}\label{eq:scaling}
\sqrt{N}\bbE\big[Z_{\lfloor sN\rfloor,\lfloor tN\rfloor}^{\go,\beta_N}\big]=\sqrt{N}\bP\big(S_{\lfloor tN\rfloor-\lfloor sN\rfloor}=0\big)\sim\frac{1}{\sqrt{2\pi(t-s)}}.
\end{equation}
We equip the space of locally finite measures on $\bbR^2_{\leq}$ with the topology of vague convergence:
\begin{equation}\label{MR2<}
\mu_n\to\mu \quad\Longleftrightarrow\quad \int \phi(s,t)\mu_n(\dd s,\dd t) \overset{\rm d}{\longrightarrow} \int \phi(s,t)\mu(\dd s,\dd t),\quad\forall \phi\in C_c(\bbR^2_\leq).
\end{equation}

Before we state our main result, let us stress that beyond the exact critical value $\beta_{N,c}=1/\sqrt{R_N}$, there is actually a critical window for $\beta_N$ that we can study, as given in (\ref{def:betaN}) below. First, we introduce a rescaled and centered disorder 
\begin{equation}\label{def:zeta}
\zeta_n=\zeta_n^{(N)}:=e^{\beta_N\omega_n-\lambda(\beta_N)}-1.
\end{equation}
Note that $(\zeta_n)_{n\in\bbZ}$ are also i.i.d.\ and
\begin{equation}\label{zeta_mean_var}
\bbE[\zeta_n]=0,\quad\sigma_N^2=\var(\zeta_n):=e^{\lambda(2\beta_N)-2\lambda(\beta_N)}-1.
\end{equation}
We will frequently use $(\zeta_n)_{n\in\bbZ}$ in the paper. Then, recall $R_N$ from \eqref{def:overlap} and we choose $\beta_N$ such that
\begin{equation}\label{def:critwin}
\sigma_N^2=e^{\lambda(2\beta_N)-2\lambda(\beta_N)}-1=\frac{1}{R_N}\Big(1+\frac{\vartheta+o(1)}{\log N}\Big),\quad\text{for some fixed }\vartheta\in\bbR.
\end{equation}
Using $\lambda(\beta)=\sum_{i=2}^4\frac{\kappa_i}{i!}\beta^i+O(\beta^5)$ as $\beta\to0$, where $\kappa_i$ is the $i$-th cumulant of $\omega$, a direct computation (see \cite[Appendix~A.4]{CSZ19a} and \cite[equation (3.12)]{CSZ21}) yields that
\begin{equation}\label{def:betaN}
\beta_N^2=R_N^{-1}+C_1(\omega)R_N^{-3/2}+C_2(\omega,\vartheta)R_N^{-2}+o(R_N^{-2}),
\end{equation}
where $C_1(\omega)$ and $C_2(\omega,\vartheta)$ are constants depending only on $\kappa_3$, $\kappa_4$ and $\vartheta$. As the parameter $\vartheta$ only appears in $C_2(\omega,\vartheta)$, contributing to the second order asymptotics, the choice ($\ref{def:betaN}$) is thus a \textit{critical window} for $\beta_N$. We can now state the main result. 

\begin{theorem}\label{thm:1}
	Let $S$ be an irreducible and aperiodic random walk satisfying \eqref{assump:RW}. Let $\beta_N$ be in the critical window of temperature with parameter $\vartheta\in\bbR$, satisfying \eqref{def:critwin}-\eqref{def:betaN}. Then the random measure $\cZ_{N}^{\beta_N}(\dd s,\dd t)$ on $\bbR^2_{\leq}$ defined in \eqref{def:pinmeasure-} converges in distribution to a unique limit $\sL^\vartheta_{}(\dd s,\dd t)$. The limit is universal, that is, it does not depend on the law of $\omega$ except for the assumption \eqref{assump:omega}.
\end{theorem}
\begin{remark}
 We call the universal limit $\sL^\vartheta$ the {\em critical disordered pinning measure (CDPM)}.
\end{remark}
\begin{remark}
		It is straightforward to see that the CDPM $\sL^\vartheta$ is translation invariant in law, 
		\begin{equation}\label{trans}
		\sL^\vartheta(\dd (s+{\rm a}),\dd (t+{\rm a})) \overset{{\rm dist}}{=} \sL^\vartheta(\dd s,\dd t),\quad\forall\hspace{1pt} {\rm a}\in\R.
		\end{equation}
		It is a corollary of Theorem~\ref{thm:a} and Theorem~\ref{thm:2} that $\sL^\vartheta$ satisfies a scaling relation
		\begin{equation}\label{scale}
		\sL^\vartheta(\dd ({\rm a}s),\dd ({\rm a}t)) \overset{{\rm dist}}{=} {\rm a} \hspace{1pt} \sL^{\vartheta+\log{\rm a}}(\dd s,\dd t),\quad\forall\hspace{1pt} {\rm a}>0.
		\end{equation}
		The above translation and scaling relations are analogous to those satisfied by the critical 2d stochastic heat flow \cite[Theorem~9.2]{CSZ21}, and their proofs are exactly the same, which we leave to the readers.
\end{remark}
\begin{remark}\label{rmk:rw_assump}

		While the first and second moments assumptions on $S$ in (\ref{assump:RW}) are standard, the finite $(2+\kappa)$-th moment assumption is for technical reason.
		In Sections \ref{S4} and \ref{S6} $($in particular, see \eqref{eq:apply_local_limit} and \eqref{eq:apply_local_limit2}$)$ we demonstrate that this condition is nearly optimal within our framework. Specifically, to approximate the random walk kernel by the heat kernel, we require more than an $o(1)$ error term in the local limit theorem, namely we need some finer control over the error. Achieving such control necessitates a stronger assumption than the existence of a finite second moment $($see Theorem \ref{thm:local_1}$)$.
		
		A general renewal process of the form \eqref{renew_kernel} can be generated by Bessel random walks, which are a family of one-dimensional nearest-neighbor random walks with inhomogeneous increments (see \cite{A11} for the choice of spatially dependent transition probabilities). For the critical case $\alpha=\frac12$ with non-constant slowly varying function $L(\cdot)$, the transition probabilities are slightly inhomogeneous, where $$\bP(S_{n+1}-S_n=\pm 1|S_n=x)=\frac{1}{2}+o\Big(\frac{1}{x}\Big).$$
		Theorem~\ref{thm:1} is also expected to hold for such random walks, as long as some necessary inputs, such as relevant local limit theorems, are established.
		
		It is also worth noting that, as demonstrated in Tsai \cite{Ts24} very recently for the stochastic heat flow, one may establish a set of conditions that fully characterizes the CDPM. This provides an alternative approach to show the convergence for general underlying renewal processes, which avoids using detailed structures such as the Markov property. It requires only verifying that the discrete model satisfies these conditions in the limit. We plan to pursue this direction in future work.
\end{remark}
\begin{remark}
		For two random walks $S$ and $\tilde S$ both satisfying (\ref{assump:RW}), it is not hard to see that the difference of their expected number of the overlap $R_N-\tilde R_N$ has a finite limit as $N$ tends to $\infty$. Denote the limit ${\rm a}({S,\tilde S}):=\lim_{N\to\infty}(R_N-\tilde R_N)$. Suppose $\sL^\vartheta(\dd s,\dd t)$ is the limit w.r.t.\ $S$ and a choice of $\beta_N$ satisfying (\ref{def:critwin}). Then, for $\tilde S$ and the same $\beta_N$,
		\begin{equation}
		\tilde{\sigma}_N^2\!=\!\frac{1}{\tilde R_N}\Big(1+\frac{\tilde R_N-R_N}{\log N/(2\pi)}\hspace{1pt} \frac{\log N/(2\pi)}{R_N}\Big)\! \Big(1+\frac{\vartheta+o(1)}{\log N}\Big)\!=\!\frac{1}{\tilde R_N}\Big(1+\frac{\vartheta-2\pi{\rm a}({S,\tilde S})+o(1)}{\log N}\Big),
		\end{equation}
		and therefore the limit w.r.t.\ $\tilde S$ and $\beta_N$ is $\sL^{\vartheta-2\pi{\rm a}({S,\tilde S})}(\dd s,\dd t)$.
\end{remark}

To show the convergence of $\cZ_{N}^{\beta_N}(\dd s,\dd t)$ to the CDPM $\sL^\vartheta(\dd s,\dd t)$ on $\R^2_\leq$, by the translation and scaling relations (\ref{trans})-(\ref{scale}), it suffices to show the convergence in a restricted domain $[0,1]^2_{\leq}:=\{(s,t):0\leq s\leq t\leq1\}$. For the latter convergence, we prove it in two steps, as detailed below.

In the first step, our goal is to establish Theorem~\ref{thm:a} below, which ensures that the convergence of random measures defined through the pinning model can be reduced to the convergence of random measures defined through a variant of the {\em directed polymer model} (for more information about the directed polymer model, see \cite{C17}). Specifically, we recall the partition function $Z_{M,K}^{\omega,\beta_N}$ from (\ref{def:partition2}) and the random measure $\cZ_{N}^{\beta_N}(\dd s,\dd t)$ from (\ref{def:pinmeasure-}). We introduce the companion point-to-point directed polymer partition function without boundary disorders
\begin{equation}\label{def:dp}
\widetilde Z_{M,K}^{\omega,\beta_N}(x,y):= \bE\Big[\exp\Big(\sum\limits_{n=M+1}^{K-1}(\beta_N\omega_n-\lambda(\beta_N))\bbmone_{\{S_n=0\}}\Big)\bbmone_{\{S_K=y\}}\Big|S_M=x\Big]
\end{equation}
and its related random measure
\begin{equation}\label{def:dpmeasure}
\widetilde\cZ_{\lfloor sN\rfloor, \lfloor tN\rfloor}^{\beta_N}(\dd x,\dd y):=\sqrt{N}\tilde Z_{\lfloor sN\rfloor, \lfloor tN\rfloor}^{\omega,\beta_N}\Big(\lfloor x\sqrt{N}\rfloor,\lfloor y\sqrt{N}\rfloor\Big)\dd x\dd y,
\end{equation}
where the disorder of the directed polymer is only located on the line $\bbZ\times\{0\}$.
Similar to (\ref{MR2<}), we regard $\widetilde\cZ_{\lfloor sN\rfloor, \lfloor tN\rfloor}^{\beta_N}(\dd x,\dd y)$ as a random variable taking values in the space of locally finite measures on $\bbR\times\bbR$ equipped with the topology of vague convergence.
\begin{theorem}\label{thm:a}
	Let $S$ and $\beta_N$ be as specified in Theorem \ref{thm:1}. The following are equivalent. 
	\begin{itemize}
		\item [{\rm (i)}] The disordered pinning measure $\cZ_N^{\beta_N}(\dd s,\dd t)$ on $[0,1]^2_\leq$ defined in \eqref{def:pinmeasure-} converges weakly to the unique CDPM $\sL^\vartheta(\dd s,\dd t)$ on $[0,1]^2_\leq$.
		\item [{\rm (ii)}]  The directed polymer measure $\widetilde\cZ_{0, N}^{\beta_N}(\dd x,\dd y)$ on $\bbR\times\bbR$ defined in \eqref{def:dpmeasure} converges weakly to a unique limiting random measure $\widetilde\sL^\vartheta_{0,1}(\dd x,\dd y)$ on $\bbR\times\bbR$. 
	\end{itemize}
	Moreover, $\sL^\vartheta(\dd s,\dd t)$ and $\widetilde\sL^\vartheta_{0,1}(\dd x,\dd y)$ satisfies the following relationship:
	\begin{equation}\label{ztoz}
	\widetilde\sL^\vartheta_{0,1}(\dd x,\dd y)= \ov Q(x,y)\dd x\dd y+ \iint_{0<s<t<1} Q(x,s) \sL^\vartheta(\dd s,\dd t) Q(-y,1-t)\dd x\dd y,
	\end{equation}
	where $\ov Q(x,y)$ is the probability density at time-space $(1,y)$ of a Brownian motion which starts at time-space $(0,x)$ and does not hit $0$ in the time interval $(0,1)$,
	and $Q(x,s)=\frac{|x|}{\sqrt{2\pi}s^{3/2}}e^{-\frac{x^2}{2s}}$ is the probability density that a Brownian motion, which starts at time-space $(0,x)$, first hits $0$ at time $s$.
\end{theorem}

With Theorem \ref{thm:a} in hand, to conclude Theorem \ref{thm:1}, it suffices to show the theorem below.
\begin{theorem}\label{thm:2}
	Let $S$ and $\beta_N$ be as specified in Theorem \ref{thm:1}. The random measure $\widetilde\cZ_{0,N}^{\beta_N}(\dd x,\dd y)$ defined in \eqref{def:dpmeasure} converges in distribution to a unique limit $\widetilde\sL_{0,1}^{\vartheta}(\dd x, \dd y)$. The limit is universal, that is, it dose not depend on the law of $\omega$ except for the assumption \eqref{assump:omega}.
\end{theorem}
\begin{remark}\label{r:fdd}
As shown in \cite[Theorem~1.1]{CSZ21}, the above convergence can be extended to convergence in finite dimensional distributions of $(\widetilde\cZ_{\lfloor sN\rfloor, \lfloor tN\rfloor}^{\beta_N})_{0\leq s<t<\infty}$ to $(\widetilde\sL^{\vartheta}_{s,t})_{0\leq s<t<\infty}$. 
\end{remark}
\begin{remark}\label{r:ptop}
 The limiting measures $\sL(\dd s,\dd t)$ and $\widetilde\sL_{s,t}(\dd x,\dd y)$ corresponds to the {\em point-to-point} partition function. The original pinning model \eqref{def:pin} and the directed polymer model \eqref{def:dp}, as well as the solutions \eqref{eq:FK} of stochastic Volterra equations and stochastic heat equations to be discussed in the next subsection, are given in terms of {\em point-to-line} partition function. It is not hard to see that the scaling limits of these point-to-line (and similarly, line-to-point) partition functions can be obtained by integrating out one relevant coordinate in $\sL(\dd s,\dd t)$ or $\widetilde\sL_{s,t}(\dd x,\dd y)$.
\end{remark}

The advantage of working with the directed polymer measure $\widetilde\cZ_{0,N}^{\beta_N}(\dd x,\dd y)$ in Theorem~\ref{thm:2}, rather than the disordered pinning measure $\cZ_{0,N}^{\beta_N}(\dd s,\dd t)$ in Theorem~\ref{thm:1}, lies in the introduction of spatial information, resulting in a Markovian structure. Our proof of random measure convergence involves high moments computations, achieved by coupling multiple polymer chains in the same random environment and then integrating out the environment. Since the random environment is concentrated on the spatial point 0, it is convenient to decompose the polymer chains at each renewal time that anyone of the polymers hits 0. Keeping track of the spatial positions of all polymers, unlike the renewal processes in the pinning model, gives rise to a Markov chain.

\subsection{Stochastic Volterra equation and stochastic heat equation}\label{S:sve}
The stochastic Volterra equation has recently attracted vast attention from mathematical finance community. It serves as a {\em rough volatility model}, capturing certain key features of short time behavior for asset pricing, see Bayer~et.~al.\ \cite{BFGMS14} and the references therein. For a smooth function $f$, a parameter $H\in[0,\frac12)$, a Brownian motion $W_{\cdot}$ and a random initial condition $\eta$ independent of $W_{\cdot}$, the Volterra equation is defined by
\begin{equation}\label{sve0}
Y_t= \eta+\int_0^t \frac{f(Y_s)}{(t-s)^{\frac{1}{2}-H}}\dd W_s,
\end{equation}
which is deemed rough due to the presence of the singular kernel $(t-s)^{H-\frac{1}{2}}$. In practice, there is evidence \cite{NR18} suggesting that the volatility aligns well with \eqref{sve0} for $H$ close to 0.

For constant function $f(x)\equiv c$ and $H>0$, \eqref{sve0} is the additive Volterra equation, and its solution $Y^H$ is the Riemann-Liouville fractional Brownian motion with Hurst index $H$. Letting $H\downarrow0$, \cite{NR18} showed that a sequence of suitably rescaled $Y^H$ converges to a log-correlated Gaussian measure $Y^0$, and thus $Y^0$ can be viewed as the solution of the additive Volterra equation with zero Hurst index.

For $f(x)=cx$, we obtain the \textit{stochastic Volterra equation} (SVE) in a multiplicative form
\begin{equation}\label{sve}
Y_t= \eta+\frac{\beta}{\sqrt{2\pi}}\int_0^t \frac{Y_s}{(t-s)^{\frac{1}{2}-H}}\dd W_s,
\end{equation}
where we replaced the constant $c$ by $\beta/\sqrt{2\pi}$ for later use. For $H\in[0,\frac12)$, the SVE \eqref{sve} has no classic solution, as can be seen by a calculation of regularity of the terms in the integral \cite{BFGMS14}. In particular, for $H\in(0,\frac12)$, the equation falls within the so-called {\em subcritical} domain, extensively explored in previous literature. In this regime, the solution has been defined in \cite{BFGMS14} using regularity structures. Parallelly, the subcritical SVE has also been studied in \cite{C19} via rough paths for $H\in(\frac14,\frac12)$, and in \cite{PT21} based on paracontrolled distributions for $H\in(\frac13,\frac12)$. However, for the {\em critical} SVE where $H=0$, a rigorous interpretation of the equation (\ref{sve}) remains open. Below, we will focus on the critical SVE and attempt to address this open question.

For $H=0$, the SVE is closely related to the {\em stochastic heat equation} (SHE), which we now illustrate. Consider the following (formal) SHE with multiplicative noise that is white in time and delta in space,
\begin{equation}\label{she}
\partial_t u(x,t) = \frac12 \Delta u(x,t) +\beta u(x,t)\cdot\delta_0(x)\dot{W}_t, \quad u(x,0)=\eta(x),\quad (x,t)\in\bbR\times\bbR^+,
\end{equation}
where $\delta_0$ is the delta function in the space coordinate, $\dot{W}_t$ is the one-dimensional white noise indexed by time (i.e., $W_t$ is a one-dimensional standard Brownian motion), and the initial condition $\eta(x)$ is a random function independent of $\dot{W}_t$. Utilizing the heat kernel 
\begin{equation}\label{heatkernel}
g_t(x):=\frac{1}{\sqrt{2\pi t}}e^{-\frac{x^2}{2t}}\quad\text{with}\quad g_0(0):=1,
\end{equation} 
we can rewrite \eqref{she} in the following integral form
\begin{equation} \label{shesol}
u(t,x)=\int_\bbR g_t(x-y)\eta(x)\dd y+\beta\int_0^t\int_\bbR g_{t-s}(x-y)  u(s,y)\delta_0(y)\dot W_s\dd y\dd s.
\end{equation}
We can then obtain \eqref{sve} by the following. Setting $\eta(x)\equiv\eta$, the first integral simplifies to $\eta$. Integrating over $y$ in the second integral, we see that $Y_t=u(t,0)$ solves the SVE \eqref{sve} with $H=0$.

Nevertheless, with the random noise term in SHE, the solution $u$ should be a distribution, and thus the equations \eqref{she} and \eqref{shesol} are ill-defined due to the product of two distributions $u(t,x)$ and $\delta_0(x)\dot W_t$. To make sense of such critical SHE, we introduce the following mollifying procedure. For $\delta>0$, let 
\begin{equation}\label{eq:mollify}
\rho^\delta(x):=\frac{1}{\delta}\rho\Big(\frac{x}{\delta}\Big),\quad x\in\bbR,
\end{equation}
with $\rho\in C_c^\infty(\bbR)$ an even probability density. The mollified equation
\begin{equation}\label{eq:mollify_she}
\begin{split}
\partial_t u^\delta(t,x)=\frac{1}{2}\partial_{xx}u^\delta(t,x)+\beta u^\delta(t,x)\rho^\delta(x)\dot{W}_t,
\quad u^\delta(0,x)= \eta(x)
\end{split}
\end{equation}
is well-defined. Using the time-reversed noise $\dot{\widetilde{W}}_t:=\dot W_{1-t}$, the solution $(u^\delta(t,x))_{t\in[0,1],x\in\bbR}$ admits a generalised Feynman-Kac representation (see \cite{BC95,CSZ17b}) with
\begin{equation}\label{eq:FK}
u^{\delta}(t,x)=\bE^{1-t,x}\Big[\eta(B_1)\exp\Big(\beta\int_{1-t}^{1}\rho^\delta(B_{s})\dd\widetilde W_s-\frac12\beta^2\int_{1-t}^1(\rho^\delta(B_{s}))^2\dd s\Big)\Big],
\end{equation}
where $\bE^{1-t,x}$ denotes the expectation for a Brownian motion $B_\cdot$ independent of $\tilde{W}_\cdot$, starting at time-space $(1-t,x)$. The right-hand side of \eqref{eq:FK} is a natural continuous \textit{point-to-plane} analogue of the directed polymer \eqref{def:dp}. Thereby as $\delta\downarrow0$, a result similar to Theorem~\ref{thm:2} is expected to hold.

The precise result is stated as follows. We only treat the marginal distribution at $t=1$ (and the case for general $t$ follows from Remark~\ref{r:fdd}). With a slight abuse of notation, denote $\bE^{0,x}$ and the associated probability $\bP^{0,x}$ by $\bE^x$ and $\bP^x$, respectively. By a change of variables $(s,y)=(\delta^2\tilde{s}, \delta\tilde{y})$ and by the invariance of diffusive scaling for Brownian motion, $u^\delta(1,x)$ is equal in distribution to the following integral,
\begin{equation}\label{eq:feynman-kac}
\begin{split}
&\bE^{\frac{x}{\delta}}\!\Big[\eta(\delta B_{\delta^{-2}})\exp\Big(\beta\int_0^{\delta^{-2}}\rho(B_{\tilde{s}})\dd\tilde{W}_{\tilde{s}}-\frac12\beta^2\int_0^{\delta^{-2}}(\rho(B_{\tilde{s}}))^2\dd\tilde{s}\Big)\Big]= \\
\int_{y\in\bbR}\!\!\eta(y) &\bE^{\frac{x}{\delta}}\!\Big[\exp\Big(\beta\int_0^{\delta^{-2}}\!\!\!\!\rho(B_{\tilde{s}})\dd\tilde{W}_{\tilde{s}}-\frac12\beta^2\int_0^{\delta^{-2}}\!\!\!\!(\rho(B_{\tilde{s}}))^2\dd\tilde{s}\Big)\Big| B_{\delta^{-2}}=\frac{y}{\delta} \Big] \bP^{\frac{x}{\delta}}\!\Big(B_{\delta^{-2}}\in\frac{\dd y}{\delta}\Big).
\end{split}
\end{equation}
In the following, we write $s=\tilde{s}$ and $W=\tilde{W}$ to lighten the notations. If we pin the Brownian motion at $y/\delta$ at the terminal time~$\delta^{-2}$, then we can define a rescaled continuous analogue of $\tilde{\cZ}_{0,N}^{\beta_N}(\dd x,\dd y)$ in \eqref{def:dpmeasure} by
\begin{equation}\label{def:heatm}
\begin{split}
&\tilde u_{0,1}^\delta(\dd x,\dd y):=\\
\bE^{\frac{x}{\delta}}\Big[\exp\Big(\beta_\delta&\int_0^{\delta^{-2}} \hspace{-8pt} \rho(B_s)\dd W_s -\frac12\beta_\delta^2\int_0^{\delta^{-2}}\hspace{-8pt}(\rho(B_s))^2\dd s\Big)\Big|B_{\delta^{-2}}=\frac{y}{\delta}\Big]\bP^{x}(B_1\in \dd y )\dd x,
\end{split}
\end{equation}
where we also set $\beta_\delta\downarrow0$ as $\delta\to0$ such that
\begin{equation}\label{def:beta_delta}
\beta_\delta^2:=\frac{2\pi}{\log\delta^{-2}}+\frac{\theta+o(1)}{(\log\delta^{-2})^{2}},\quad\text{for some fixed}~\theta\in\bbR.
\end{equation}
We have the following result.
\begin{theorem}\label{thm:2+}
	The random measure $\tilde u_{0,1}^\delta(\dd x,\dd y)$, defined as in \eqref{def:heatm}, converges in distribution to the unique limit $\widetilde\sL_{0,1}^\vartheta(\dd x, \dd y)$ defined in Theorem \ref{thm:a}, where with the mollifier $\rho(\cdot)$ given in \eqref{eq:mollify},
	\begin{equation}\label{contin_vartheta}
	\vartheta=\log2-\gamma+\idotsint_{\bbR^4}\rho(x')\rho(y')\log\Big[\frac{1}{(x-x')^2+(y-y')^2}\Big]\rho(x)\rho(y)\dd x'\dd y'\dd x\dd y+\frac{\theta}{2\pi} \hspace{1pt}.
	\end{equation}
\end{theorem}
\begin{remark}\label{rmk:she}
		In the proof of Theorem  \ref{thm:2+}, we will actually show that for $\varphi\in C_c(\bbR)$ and $\psi\in C_b(\bbR)$, $\iint_{\bbR^2}\!\varphi(x)\tilde{u}_{0,1}^\delta\!(\dd x,\dd y)\psi(y)$ converges weakly to $\iint_{\bbR^2}\!\varphi(x)\widetilde\sL_{0,1}^{\vartheta}\!(\dd x,\dd y)\psi(y)$ as $\delta\to0$.
\end{remark}

As in \eqref{def:heatm}, for $0\leq s<t\leq 1$, we can also define
\begin{equation}
\begin{split}
&\tilde u_{s,t}^\delta(\dd x,\dd y):=\\ \bE^{\frac{s}{\delta^{2}},\frac{x}{\delta}}\Big[\exp\Big(\beta_\delta&\int_{s\delta^{-2}}^{t\delta^{-2}} \hspace{-8pt} \rho(B_s)\dd W_s -\frac12\beta_\delta^2\int_{s\delta^{-2}}^{t\delta^{-2}}\hspace{-8pt}(\rho(B_s))^2\dd s\Big)\Big|B_{t\delta^{-2}}=\frac{y}{\delta}\Big] \hspace{1pt}\bP^{s,x}(B_t\in \dd y )\dd x.
\end{split}
\end{equation}
By Remark~\ref{r:fdd}, we have that $(\tilde{u}^\delta_{s,t})_{0\leq s<t\leq1}$ converges in finite dimensional distributions to $(\widetilde\sL_{s,t}^\vartheta)_{0\leq s<t\leq1}$. Note that by \eqref{eq:FK}-\eqref{eq:feynman-kac}, the solution $u^\delta(t,x)$ of the mollified SHE \eqref{eq:mollify_she} satisfies the relation
$u^\delta(t,x)\dd x=\tilde u^\delta_{1-t,1}(\dd x, \eta):=\int_{y\in\bbR}\tilde u^\delta_{1-t,1}(\dd x, \dd y)\eta(y)$. Therefore, in view of Theorem~\ref{thm:2+} and Remark~\ref{rmk:she}, if the initial random function $\eta(x)$ is bounded and continuous, then the critical singular SHE \eqref{she} has a natural candidate solution $\cU(t,\dd x)$ given by
\begin{equation}\label{solshe}
\cU(t,\dd x)= \widetilde{\sL}_{1-t,1}^\vartheta(\dd x,\eta):=\int_{y\in\bbR}\widetilde{\sL}_{1-t,1}^\vartheta(\dd x,\dd y)\eta(y).
\end{equation}

We then deal with the critical singular SVE \eqref{sve} with the help of \eqref{eq:FK}-\eqref{def:heatm}. Since $\eta$ is independent of $B$ and $W$, we assume without loss of generality that $\eta\equiv1$. Decomposing $\tilde{u}_{0,1}^\delta(\dd x,\dd y)$ according to the first time in $(0,1)$ that the Brownian motion $B$ visits $0$, we have that
\begin{equation}\label{eq:density}
\tilde{u}_{0,1}^\delta(\dd x,\dd y)=\ov Q(x,y)\dd x\dd y+\int_{0<t<1}\Big(Q(x,t)\tilde{u}_{t,1}^\delta(0,\dd y)\dd x\Big)\dd t,
\end{equation}
where the probability densities $\ov Q$ and $Q$ are defined in \eqref{ztoz} and note that the measure $\tilde{u}_{t,1}^\delta(\dd x,\dd y)$ is absolutely continuous with respect to the Lebesgue measure, and hence $\tilde{u}_{t,1}^\delta(x,\dd y):=\tilde{u}_{t,1}^\delta(\dd x,\dd y)/\dd x$ is a well-defined function of $x$.

Integrating $y$ over $\bbR$ in \eqref{eq:density}, we obtain
\begin{equation}\label{eq:density_int}
\int_{y\in\bbR}\!\tilde{u}_{0,1}^\delta(\dd x,\dd y)\!=\!\int_{y\in\bbR}\!\ov Q(x,y)\dd x\dd y+\int_{0<t<1}\!\Big[Q(x,t)\Big(\int_{y\in\bbR}\tilde{u}_{t,1}^\delta(0,\dd y)\Big)\dd x\Big]\dd t.
\end{equation}
Recall \eqref{shesol} and note that $Y^\delta_t:=u^{\delta}(t,0)=\int_{y\in\bbR}\tilde{u}_{t,1}^\delta(0,\dd y)$ can be viewed as the solution of mollified SVE. Therefore, by Theorem \ref{thm:2+}, as $\delta\to0$, \eqref{eq:density_int} converges in distribution to
\begin{equation}\label{eq:density_lim}
\tilde{\sL}_{0,1}^{\vartheta}(\dd x,\ind_\bbR)=\int_{y\in\bbR}\ov Q(x,y)\dd x\dd y+\int_{0<t<1}Q(x,t)\cY(\dd t)\dd x,
\end{equation}
where the random measure $\cY(\dd t)$ is the weak limit of $\cY^{\delta}(\dd t):=Y_t^{\delta}\dd t=u^\delta(t,0)\dd t$. We can thus conclude that a natural candidate solution of the critical singular SVE \eqref{sve} is $\cY(\dd t)$. Moreover, by comparing \eqref{eq:density_lim} with \eqref{ztoz}, we have that
\begin{equation}\label{def:cY}
\cY(\dd t)=\int_{t<s<1}\int_{y\in\R}Q(-y,1-s)\dd y \sL^\vartheta(\dd t,\dd s) =\sqrt{\frac{2}{\pi}}\int_{t<s<1} s^{-1/2}\sL^\vartheta(\dd t,\dd s).
\end{equation}

We summarize the above argument as the following corollary.
\begin{corollary}
Let $\beta_\delta$ be specified as in \eqref{def:beta_delta} and $u^\delta(t,x)$ be the solution of mollified SHE \eqref{eq:mollify_she}. Then $u^\delta(t,x)\dd x$ converges in distribution to the random measure $\cU(t,\dd x)$, defined in \eqref{solshe}, which can be viewed as a solution of the SHE \eqref{she}. Meanwhile, $u^\delta(t,0)\dd t$ converges in distribution to a random measure $\cY(\dd t)$, given in \eqref{def:cY}, which can be viewed as a solution of the SVE \eqref{sve} with critical Hurst index $H=0$.
\end{corollary}

\subsection{Contributions and discussions}
An effectively studied marginally relevant disordered model at criticality is the directed polymer as presented in the seminal work \cite{CSZ21}. It proves that for the two-dimensional directed polymer with temperature in the critical window, the family of polymer measures (similar to (\ref{def:dpmeasure})), indexed by starting and end times, converges in finite dimensional distributions to the critical 2d stochastic heat flow. The same convergence is also believed to hold for the critical 2d SHE, the continuous analogue of the directed polymer, although \cite{CSZ21} does not provide proof details. This outcome offers an interpretation of the solution to a critical singular SPDE for the first time. The proof therein is robust enough, enabling us to adapt it for Theorem~\ref{thm:2}. Using the same framework, we provide a complete proof of Theorem~\ref{thm:2+} as well, which, in turn, complements the case of SHE. 
Besides the technical challenges needed to adapt the mathematical framework of \cite{CSZ21} to the current setting, the interest of our paper lies in the establish of the CDPM for the disordered pinning model, and in the applications to key problems in SPDEs and mathematical finance, which we highlight as follows:

\begin{itemize}
	\item Firstly, the limiting CDPM of the marginally relevant pinning model at the critical temperature is identified for the first time. For the pinning model (\ref{def:renewalpin}), when the disorder is relevant, a scaling limit of the pinning partition function was obtained in \cite{CSZ16}; when the disorder is marginally relevant and the temperature is in the so-called subcritical regime, a scaling limit of the pinning partition function was established in \cite{CSZ17b}. However, to the best of our knowledge, for marginally relevant pinning model at criticality, how to obtain a scaling limit was unknown in both physics and mathematics literature.
	
	\item Secondly, we give, for the first time, a meaning to the solution of the multiplicative SVE (\ref{sve}) with critical index $H=0$, in contrast to the rough volatility modelled by additive SVE with $H=0$ as studied in \cite{NR18}. By connecting with SHE, we define a solution to such singular models of SDEs, thereby extending their scope from the subcritical phase to the critical phase.
	
	\item Thirdly, in Theorem \ref{thm:a} we establish the equivalence of convergence between random measures on time coordinate and on spatial coordinate, which allows us to use the techniques developed in \cite{CSZ21} to deduce Theorem \ref{thm:1} by proving Theorem \ref{thm:2}. In particular, as in \eqref{def:dpintegral4}-\eqref{def:pinint}, we show that the measures on time coordinate and on spatial coordinate can be expressed as weighted averages of each other, where the weights are probability densities for the first hitting times of a Brownian motion, which might be of independent interests. Although the equivalence seems intuitively natural, it turns out that the proof is highly non-trivial, crucially relying on the expressions of transition kernel related to the Brownian motion, see, e.g., (\ref{hkint}).
	
	\item Fourthly, the paper can be read as a complement to \cite{CSZ21}. Some proofs that were sketched in \cite{CSZ21} are now provided in more details, such as the moment computations of the coarse-grained model w.r.t.\ disorders $W_{s,t,u}^{\vec{\rmi}_1,\vec{\rmi}_2}$ in Section~\ref{S9} (see Lemma \ref{lem:bound_W}). We also provide enough details in proving Theorem \ref{thm:2+}, the SHE case, demonstrating the robustness of techniques in \cite{CSZ21} under continuous settings.
\end{itemize}

In the next section, we will illustrate our proof strategy for Theorem~\ref{thm:a}.
For Theorems~\ref{thm:2} and \ref{thm:2+}, we will first review the strategies in \cite{CSZ21}, and then discuss the adaption to our case. Along the way, we will also introduce some preliminary ingredients that will be moreover used throughout the paper. At the end of that section, we will state the organization of the rest of the paper.

\section{Proof Outlines, Notations and Tools}\label{S:outline}
Note that Theorem~\ref{thm:1} is a direct consequence of Theorem~\ref{thm:a} and Theorem~\ref{thm:2}. Theorem~\ref{thm:2} and Theorem \ref{thm:2+} are proved in the same framework as developed in \cite{CSZ21}. We thereby only focus on explaining the proofs for Theorem~\ref{thm:a} and Theorem~\ref{thm:2}. Before illustrating the proof of Theorem~\ref{thm:2}, we will first review the framework in \cite{CSZ21}, so that it is more straightforward to explain the modifications needed when adapting. We will also introduce some necessary notations and tools, including polynomial chaos expansion, the Dickman subordinator, coarse-grained disorders and a coarse-grained model.

\subsection{Proof outline for Theorem~\ref{thm:a}}\label{outline1}
We need to show that the random pinning measure $\cZ_N^{\beta_N}(\dd s,\dd t)$ (resp.\ the random polymer measure $\widetilde\cZ_{0,N}^{\beta_N}(\dd x,\dd y)$) has a unique limit $\sL^\vartheta(\dd s,\dd t)$ (resp.\ $\widetilde\sL^\vartheta_{0,1}(\dd x,\dd y)$). Thanks to \cite[Corollary~4.14]{K17}, it suffices to apply test functions $h\in C_c([0,1]^2_\leq)$ (resp.\ $\phi\in C_c(\bbR^2)$) w.r.t.\ the random measures $\cZ_N^{\beta_N}(\dd s,\dd t)$ (resp.\ $\tilde{\cZ}_{0,N}^{\beta_N}(\dd x,\dd y)$), and show that the resulting sequence of random variables has a weak limit. In particular, to prove Theorem~\ref{thm:a}, we only need to consider test functions of the form $f(s)g(t)$ with $f,g\in C_c([0,1])$ (resp. $\varphi(x)\psi(y)$ with $\varphi,\psi\in C_c(\bbR)$) by a density argument, and show that the following two statements imply each other:
\begin{align}
\label{eq1}&\cZ_N^{\beta_N}(f,g):= \iint _{[0,1]^2_\leq} f(s)\cZ_N^{\beta_N}(\dd s,\dd t)g(t)~~\text{converges weakly for all}~f,g\in C_c([0,1]);\\
\label{eq2}&\widetilde\cZ_{0,N}^{\beta_N}(\varphi,\psi):=\iint _{\bbR^2} \varphi(x)\widetilde\cZ_{0,N}^{\beta_N}(\dd x,\dd y)\psi(y)~~\text{converges weakly for all}~\varphi,\psi\in C_c(\bbR).
\end{align}
A key step is to compare the decompositions of $\cZ_N^{\beta_N}(f,g)$ and $\widetilde\cZ_{0,N}^{\beta_N}(\varphi,\psi)$.

We first decompose $\widetilde\cZ_{0,N}^{\beta_N}(\varphi,\psi)$ according to the first and last hitting times to $0$ for the random walk in the time interval $[1,N-1]$. Let the first and last hitting times be
\begin{align}
\label{def:renew1}\tau_1&:=\inf\{n\geq1: S_n=0~\text{and}~S_k\neq0~\text{for}~k=1,\cdots,n-1\},\\
\label{def:renew2}\tau_2&:=\tau_2^{(N)}=\sup\{n\leq N-1: S_n=0~\text{and}~S_k\neq0~\text{for}~k=n+1,\cdots,N-1\}.
\end{align}
Recall $\widetilde Z_{0,N}^{\go,\gb_N}(x,y)$ from (\ref{def:dp}), and we have that
\begin{equation}\label{dp:decomp}
\widetilde Z_{0,N}^{\go,\gb_N}(x,y)=\widetilde Z_{0,N}^{\omega,\beta}(x,y)\left[\ind_{\{\tau_1\geq N\}}\right] +\sum\limits_{m=1}^{N-1}\sum\limits_{n=m}^{N-1}\widetilde Z_{0,N}^{\go,\gb_N}(x,y)\left[\ind_{\{\tau_1=m,\tau_2=n\}}\right],
\end{equation}
where
\begin{equation*}
\widetilde Z_{0,N}^{\go,\gb_N}(x,y)\left[\ind_A\right]:=\bE^x\Big[\exp\Big(\sum\limits_{n=1}^{N-1}(\beta\omega_n-\lambda(\beta))\bbmone_{\{S_n=0\}}\Big)\bbmone_{\{S_N=y\}}\ind_A\Big].
\end{equation*}
It is straightforward that
\begin{equation}\label{dp:poly1}
\widetilde Z_{0,N}^{\omega,\beta}(x,y)\left[\ind_{\{\tau_1\geq N\}}\right]=\bP^x(S_N=y, \tau_1\geq N),
\end{equation}
and due to the absence of the disorders at times $0$ and $N$,
\begin{equation}\label{dp:poly3}
\widetilde Z_{0,N}^{\omega,\beta}(x,y)\left[\ind_{\{\tau_1=m,\tau_2=n\}}\right]= \bP^x(\tau_1=m)\times Z^{\omega,\beta_N}_{m,n}\times\bP^x(\tau_2=n,S_N=y|S_n=0),
\end{equation}
where $Z^{\omega,\beta_N}_{m,n}$ is the pinning partition function defined in (\ref{def:partition2}).

We introduce
\begin{equation}\label{eq:hitkernel}
q_x(n):=\bP^x(\tau_1=n)
\end{equation}
to denote the probability of the first hitting time to $0$ for a random walk starting at $x$. Note that by the Markov property and a time reversal, we have that
\begin{equation*}
\bP^x(S_N=y,\tau_2=n|S_n=0)=\bP^{-y}(\tau_1=N-n)=q_{-y}(N-n).
\end{equation*}

Plugging \eqref{dp:poly1} and \eqref{dp:poly3} into \eqref{dp:decomp}, and recalling $\widetilde\cZ_{0,N}^{\beta_N}(\dd x,\dd y)$ from (\ref{def:dpmeasure}), we obtain 
\begin{equation}\label{def:dpintegral4}
\begin{split}
&\widetilde\cZ_{0,N}^{\beta_N}(\varphi,\psi)=\iint_{\bbR\times\bbR} \varphi(x) \sqrt{N} \hspace{1pt}\bP^{x_N}(S_N=y_N,\tau_1\geq N) \psi(y) \dd x\dd y \\
+\sum\limits_{m=1}^{N-1}\sum\limits_{n=m}^{N-1} & \Big(\int_{\bbR}\varphi(x)\sqrt{N}q_{x_N}(m)\dd x\Big)\Big(\sqrt{N} Z^{\omega,\beta_N}_{m,n}\Big) \Big(\int_{\bbR}\psi(x)\sqrt{N}q_{-y_N}(N-n)\dd y\Big),
\end{split}
\end{equation}
where $x_N:=\lfloor x\sqrt{N}\rfloor$ and $y_N:=\lfloor y\sqrt{N}\rfloor$.

Similarly, we have that for integral with respect to the pinning measure,
\begin{equation}\label{def:pinint}
\cZ_N^{\beta_N}(f,g)=\sum\limits_{m=0}^{N-1}\sum\limits_{n=m}^{N-1}\Big(\int_{\frac{m}{N}}^{\frac{m+1}{N}}f(s)\dd s\Big)\Big(\sqrt{N}Z_{m,n}^{\omega,\beta_N}\Big)\Big(\int_{\frac{n}{N}}^{\frac{n+1}{N}}g(t)\dd t\Big).
\end{equation}

Note that the first term at the right-hand side of \eqref{def:dpintegral4} is deterministic. To show (i)$\Longleftrightarrow$(ii) in Theorem \ref{thm:2}, it suffices to show that (a) the deterministic term converges; (b) the term $m=0$ in \eqref{def:pinint} is negligible; (c) the two double-sums with $1\leq m\leq n\leq N-1$ in \eqref{def:pinint} and \eqref{def:dpintegral4} can be made arbitrarily close by choosing suitable $f,g,\varphi,\psi$. While (a) and (b) are somewhat direct to check, (c) is verified through a crucial observation (\ref{hkint}) regarding the heat kernel, that allows approximations between test functions $f,g$ and $\varphi,\psi$.

\subsection{Proof strategy in \cite{CSZ21}}\label{subsec:strategy} In \cite{CSZ21}, the authors treated the classic directed polymer model, whose point-to-point partition function is defined by
\begin{equation*}
\widebar Z_{M,K}^{\omega,\beta_N}(x,y):= \bE\Big[\exp\Big(\sum\limits_{n=M+1}^{K-1}(\beta_N\omega_{n,S_n}-\lambda(\beta_N))\Big)\bbmone_{\{S_K=y\}}\Big|S_M=x\Big],
\end{equation*}
where the disorders $\omega_{n,x}$ are located on the whole time-space lattice $\bbN\times\bbZ^2$. The related random measure is defined similarly to \eqref{def:dpmeasure}, by
\begin{equation*}
\widebar\cZ_{\lfloor sN\rfloor, \lfloor tN\rfloor}^{\beta_N}(\dd x,\dd y):=\frac{N}{4}\widebar Z_{\lfloor sN\rfloor, \lfloor tN\rfloor}^{\omega,\beta_N}\Big(\lfloor x\sqrt{N}\rfloor,\lfloor y\sqrt{N}\rfloor\Big)\dd x\dd y.
\end{equation*}

Given test functions $\varphi,\psi\in C_c(\bbR^2)$, their integral w.r.t.\ $\widebar\cZ_{0, N}^{\beta_N}(\dd x,\dd y)$ is denoted by $\widebar\cZ_N=\widebar\cZ_N^{\beta_N}(\varphi,\psi):=\iint_{\bbR^4}\varphi(x)\widebar\cZ_{0,N}^{\beta_N}(\dd x,\dd y)\psi(y)$. To show the weak convergence of $\widebar\cZ_N$, it suffices to show that for any $f\in C_b(\bbR)$, $\bbE[f(\widebar{\cZ}_N)]$ is a Cauchy sequence. The difficulty here is that, for different $N,M$, $\widebar\cZ_N$ and $\widebar\cZ_M$ have different $\beta_N$ and $\beta_M$ and different polymer lengths $N$ and $M$, which makes the comparison between $\bbE[f(\widebar\cZ_N)]$ and $\bbE[f(\widebar\cZ_M)]$ very intractable.

To proceed, the first step is to perform a \textit{polynomial chaos expansion} (see Subsection \ref{sub:poly}) on $\widebar\cZ_N$, so that it can be expressed by a multilinear function of rescaled and centered disorders. Each term in the multilinear function is a product of some disorders and the random walk transition kernels connecting the time-space indices of the disorders. If we partition the time-space by boxes and observe the multilinear function in ``box-level'', that is, we only concern the boxes instead of the concrete time-space coordinates visited by the random walk, then $\cZ_N$ exhibits some self-similarity.

Inspired by this, for fixed $\epsilon>0$, we can introduce \textit{mesoscopic boxes} by
\begin{equation}\label{def:dp_meso}
\begin{split}
\cB_{N,\epsilon}(\rmi,\rma)&:=\big((\rmi-1)\epsilon N,\rmi\epsilon N\big]\times\big((\rma-(1,1))\sqrt{\epsilon N},\rma\sqrt{\epsilon N}\big]\\
&:=\cT_{N,\epsilon}(\rmi)\times\cS_{N,\epsilon}(\rma),\quad \rmi\in\bbN, \rma\in\bbZ^2,
\end{split}
\end{equation}
where $\cT_{N,\epsilon}(\rmi)$ is a mesoscopic time interval and $\cS_{N,\epsilon}(\rma)$ is a mesoscopic spatial box. The advantage of introducing a mesoscopic scale is that, the mesoscopic length of the polymer is now fixed by $\frac1\epsilon$ for all $N$ and many estimates only depend on $\epsilon$ uniformly in $N$ large enough.

Furthermore, it can be shown that the paths with consecutive short visits in time or large deviations in space have negligible contributions to $\widebar\cZ_N$. This significantly reduces the number of mesoscopic boxes that we need to analyze, which brings much convenience. We can then (1) group disorders on mesoscopic boxes together to have \textit{coarse-grained disorders} (similar to \eqref{def:X} and \eqref{cgvariable} below), whose structures are similar to that of $\widebar\cZ_N$; (2) replace the random walk transition kernels connecting coarse-grained disorders by heat kernels labeled by mesoscopic indices, where the costs can be controlled.

After all the treatments above, we end up with a \textit{coarse-grained model} (similar to \eqref{cgmodel} below), which can be made arbitrarily close to $\widebar\cZ_N$ for large enough $N$ and small enough $\epsilon$. The convenience of working on the coarse-grained model is that, when taking $N\to\infty$, we only need to analyze the coarse-grained disorders, without changing the structure of the model, unlike $\widebar\cZ_N$.

Finally, to show that $\bbE[f(\widebar\cZ_N)]$ is a Cauchy sequence, we can first replace $\widebar\cZ_N$ there by the coarse-grained model to obtain a new expectation sequence, and then show that the new expectation sequence is a Cauchy sequence as $\epsilon\to0$, thanks to an enhanced Lindeberg principle (see Appendix \ref{A1}). It is worth mentioning that in this last step, we need to estimate the higher moments of the coarse-grained disorders. This is done by estimating the higher moments of $\widebar\cZ_N$, which is a byproduct of \cite{CSZ21}.

\subsection{Proof outlines for Theorem \ref{thm:2} and Theorem \ref{thm:2+}}\label{outline2}
The strategy for proving Theorem \ref{thm:2} and Theorem \ref{thm:2+} is an adaption of Subsection~\ref{subsec:strategy}. Hence, we only highlight some differences below.

We mainly focus on Theorem \ref{thm:2}.
\begin{itemize}
	\item In our variant of directed polymer model \eqref{def:dp}, the disorders are only located on the time-space line $\bbZ\times\{0\}$. Hence, except when visiting the origin, the random walk distribution is not tilted by the polymer measure. Therefore, we only need to introduce the mesoscopic time interval $\cT_{N,\epsilon}(\rmi)$ in \eqref{def:dp_meso}, instead of the mesoscopic time-space boxes $\cB_{N,\epsilon}(\rmi,\rma)$, since we only concern the return times to $0$.
	
	\item As a consequence, we need less estimates on random walk transition probability than those in \cite{CSZ21}. To be specific, unlike in \cite{CSZ21}, large deviation estimates are not necessary in our case. We only require a local limit theorem whose error satisfies some summable conditions. This allows us to relax the constraint on the random walk to the assumption \eqref{assump:RW} (e.g.\ see \eqref{Q_tran2}-\eqref{Q_tran3} and \eqref{eq:apply_local_limit2}).
\end{itemize}

To prove Theorem \ref{thm:2+}, we adapt the techniques for continuous setting in \cite{CSZ19b} and show that $\iint_{\bbR^2}\varphi(x)u_{0,1}^\delta(\dd x,\dd y)\psi(y)$ (see Remark \ref{rmk:she}) can also be approximated by the coarse-grained model (see \ref{cgmodel}) introduced when proving Theorem \ref{thm:2}.

\subsection{Polynomial chaos expansion}\label{sub:poly}
The polynomial chaos expansion is a celebrated tool in treating random polymer models when the random environment is independent in time, see, e.g., \cite{CSZ16, CSZ17b, CSZ21}. We illustrate it by taking $Z_{N,\beta_N}^\omega$ defined in \eqref{def:partition} as an example. 

Recall $\zeta_n$ from \eqref{def:zeta} and \eqref{zeta_mean_var}. Note that
\begin{equation*}
e^{(\beta_N\omega_n-\lambda(\beta_N))\ind_{\{S_n=0\}}}=1+\Big(e^{\beta_N\omega_n-\lambda(\beta_N)}-1\Big)\ind_{\{S_n=0\}}=1+\zeta_n\ind_{\{S_n=0\}}.
\end{equation*}
Then with the convention $n_0:=0$,
\begin{equation}\label{def:pce}
\begin{split}
Z_{N,\beta_N}^\omega=&\bE\Big[\prod\limits_{n=1}^N\Big(1+\zeta_n\ind_{\{S_n=0\}}\Big)\Big]=1+\sum\limits_{k=1}^N\sum\limits_{1\leq n_1<\cdots<n_k\leq N}\prod\limits_{i=1}^k p_{n_{i-1},n_i}(0)\zeta_{n_i},
\end{split}
\end{equation}
where
\begin{equation}\label{RWkernel}
p_{m,n}(x,y):=\bP(S_n=y|S_m=x),\quad\text{for}~m<n,
\end{equation}
is the notation for the random walk transition kernel with the conventions $p_{m,n}(y)=p_{m,n}(0,y)$, $p_n(y)=p_{0,n}(y)$, and $p_0(0)=1$.

Consequently, the integral of test functions $\varphi, \psi$ with respect to the measure $\tilde\cZ_{0,N}^{\beta_N}(\dd x,\dd y)$ defined in \eqref{def:dpmeasure} can be expressed by the polynomial chaos expansion
\begin{equation}\label{polynomial}
\begin{split}
&\tilde\cZ_{0,N}^{\beta_N}(\varphi,\psi):=\iint_{\bbR^2}\varphi(x)\widetilde\cZ_{0,N}^{\beta_N}(x,y)\psi(y)\\
=&\sqrt{N}\sum\limits_{u,v\in\frac{1}{\sqrt{N}}\bbZ}\int_u^{u+\frac{1}{\sqrt{N}}}\varphi(x)\dd x\Big(\tilde Z_{0, N}^{\omega,\beta_N}\big(u\sqrt{N}, v\sqrt{N}\big)\Big)\int_v^{v+\frac{1}{\sqrt{N}}}\psi(y)\dd y\\
=&q_{0,N}^N(\varphi,\psi)+\frac{1}{\sqrt{N}}\sum\limits_{r=1}^\infty\sum_{0<n_1<\cdots<n_r<N}q_{0,n_1}^N(\varphi,0)\zeta_{n_1}\Big(\prod\limits_{j=2}^r p_{n_{j-1},n_j}(0)\zeta_{n_j}\Big)q_{n_r,N}^N(0,\psi),
\end{split}
\end{equation}
where for $x,y\in\bbZ$,
\begin{align}
\label{qphi0}q_{0,n}^N(\varphi,x)&:=\sum\limits_{u\in\bbZ}\varphi_N(u)p_n(x-u),\\
\label{qpsi0}q_{m,N}^N(y,\psi)&:=\sum\limits_{v\in\bbZ}p_{m,N}(v-y)\psi_N(v),\\
\label{qphipsi}q_{0,N}^N(\varphi,\psi)&:=\frac{1}{\sqrt{N}}\sum\limits_{u,v\in\bbZ}\varphi_N(u)p_N(v-u)\psi_N(v),
\end{align}
with
\begin{equation}\label{def:phi_psi2}
\varphi_N(x):=\int_{x}^{x+1}\varphi\Big(\frac{t}{\sqrt{N}}\Big)\dd t,\quad\psi_N(y):=\int_{y}^{y+1}\psi\Big(\frac{s}{\sqrt{N}}\Big)\dd s.
\end{equation}

As $(\zeta_n)_{n\in\N}$ are i.i.d.\ with mean 0, we have
\begin{equation}
\begin{split}\label{varN}
&\bbE\left[(\tilde\cZ_{0,N}^{\beta_N}(\varphi,\psi))^2\right]=\left(q_{0,N}^N(\varphi,\psi)\right)^2+\\
\frac{1}{N}\sum\limits_{r=1}^\infty&(\sigma_N^2)^r\sum\limits_{0<n_1<\cdots<n_r<N}\left(q_{0,n_1}^N(\varphi,0)\right)^2\Big(\prod\limits_{j=2}^r u(n_j-n_{j-1})\Big)\left(q_{n_r,N}^N(0,\psi)\right)^2,
\end{split}
\end{equation}
where $\sigma_N^2=\var(\zeta_n)$, and for $n\in\bbZ$,
\begin{equation}\label{def:u}
u(n):=p_n^2(0)\sim\frac{1}{2\pi n}.
\end{equation}

A powerful tool for computing the second moment \eqref{varN} is the so-called Dickman subordinator, which we now introduce in the following subsection.

\subsection{Dickman subordinator and second moment computations}\label{sub:dickman}
We first introduce a renewal process $\iota^{(N)}=\{\iota^{(N)}_1,\iota^{(N)}_2,\cdots\}$ on $\bbN\cup\{0\}$ with one-step distribution
\begin{equation}\label{new_renew}
\bP\big(\iota^{(N)}_1=n\big)=\frac{u(n)}{R_N}\ind_{\{1,\cdots,N\}}(n),
\end{equation}
where $R_N$ is the expected overlap in \eqref{def:overlap}.

In view of \eqref{varN} above, if the first and last return times to $0$ for the random walk are denoted by $m$ and $n$ respectively with $m<n$, then we have that by the renewal structures of the random walk and \eqref{new_renew},
\begin{equation}\label{connect_to_renew}
\begin{split}
&(\sigma_N^2)^2u(n-m)+\sum\limits_{k=1}^\infty(\sigma_N^2)^{k+2}\sum\limits_{n_0:=m<n_1<\cdots<n_k<n=:n_{k+1}}\prod\limits_{j=1}^{k+1}u(n_j-n_{j-1})
\\=&\sigma_N^2\sum\limits_{k=1}^{\infty}(\lambda_N)^k\bP(\iota^{(N)}_k=n-m),\quad\text{where}~\lambda_N:=\sigma_N^2R_N=1+\frac{\vartheta+o(1)}{\log N}.
\end{split}
\end{equation}
Hence, to analyze \eqref{varN}, in particular, \eqref{connect_to_renew}, it turns out that we can investigate the renewal process $\iota^{(N)}$, which has been intensively studied in \cite{CSZ19a}. We cite the following result.

\begin{proposition}[\cite{CSZ21}, Lemma 3.3]\label{prop:dickman}
	Let $\iota^{(N)}$ be defined as above, then as $N\to\infty$
	\begin{equation*}
	\bigg(\frac{\iota_{\lfloor s\log N\rfloor}^{(N)}}{N}\bigg)_{s\geq 0}\overset{({\rm d})}{\longrightarrow}(Y_s)_{s\geq 0},
	\end{equation*}
	where $Y_s$ is a pure jump L\'{e}vy process with L\'{e}vy measure $\frac{1}{t}\ind_{(0,1)}(t)\dd t$ called the {\rm Dickman subordinator}, and the convergence in distribution holds on the space of c\`{a}dl\`{a}g paths equipped with the Skorohod topology. Moreover, the density of $Y_s$ is known ({\rm cf.} \cite[Theorem 1.1]{CSZ19a}), which is
	\begin{equation}\label{density_Dickman}
	f_s(t)=\begin{cases}
	\displaystyle
	\frac{st^{s-1}e^{-\gamma s}}{\Gamma(s+1)},\quad&\text{for}~t\in(0,1],\\
	\displaystyle
	\frac{st^{s-1}e^{-\gamma s}}{\Gamma(s+1)}-st^{s-1}\int_0^{t-1}\frac{f_s(a)}{(1+a)^s}\dd s,\quad&\text{for}~t\in(1,+\infty),
	\end{cases}
	\end{equation}
	for all $s\in(0,+\infty)$, where $\Gamma(\cdot)$ is the Gamma function and $\gamma=-\int_0^\infty\log ue^{-u}\dd u$ is the Euler-Mascheroni constant.
\end{proposition}

To include the case that there is only one return to $0$ (i.e., $m=n$ in \eqref{connect_to_renew}), we introduce the following quantity, which plays a central role throughout the paper. Let
\begin{equation}\label{def:U_bar}
\overline{U}_N(n):=\begin{cases}
\sigma_N^2,&~~\text{if}~n=0,\\
\displaystyle
(\sigma_N^2)^2 u(n)+\sum\limits_{k=1}^\infty(\sigma_N^2)^{k+2}\sum\limits_{0<n_1<\cdots<n_k<n}\prod\limits_{j=1}^{k+1} u(n_j-n_{j-1}),&~~\text{if}~n\geq1,
\end{cases}
\end{equation}
where we use the convention $n_0:=0$ and $n_{k+1}:=n$. We then collect some useful properties of $\overline{U}_N(n)$ from \cite[Section 1-3]{CSZ19a}.

Given $\vartheta$ in \eqref{def:critwin}, let
\begin{equation}\label{Gtheta}
G_\vartheta(t):=\int_0^\infty e^{\vartheta s}f_s(t)\dd s,
\end{equation}
where $f_s(\cdot)$ is the density of the Dickman subordinator $Y_s$ as in Proposition \ref{prop:dickman}.

\begin{itemize}
	\item By \cite[Theorem 1.4]{CSZ19a}, given $T>0$, there exists some constant $C>0$, such that
	\begin{equation}\label{overlineU}
	\overline{U}_N(n)\leq\frac{C}{N}G_\vartheta\left(\frac{n}{N}\right),\quad\text{uniformly for all}~1\leq n\leq TN.
	\end{equation}
	Furthermore, if $\theta\in(0,1)$ is also given, then
	\begin{equation}\label{overlineUasym}
	\overline{U}_N(n)=\frac{2\pi}{N}\left(G_\vartheta\left(\frac{n}{N}\right)+o_N(1)\right),\quad\text{uniformly for all}~\theta N\leq n\leq TN.
	\end{equation}
	
	\item By \cite[Proposition~1.6]{CSZ19a}, the asymptotics of $G_\vartheta(t)$ is known, that is,
	\begin{align}\label{asym:Gtheta}
	G_{\vartheta}(t)\sim\frac{1}{t}\Big(\log\frac{1}{t}\Big)^{-2}\quad\text{and}\quad\int_0^t G_\vartheta(s)\dd s\sim\Big(\log\frac{1}{t}\Big)^{-1},\quad\text{as}\quad t\to0.
	\end{align}
\end{itemize}

\begin{remark}\label{rmk:dickman}
		It is not surprising that the Dickman subordinator and the function $\overline{U}(\cdot)$ appear both in \cite{CSZ21} and our paper. If we perform a polynomial chaos expansion to $\widebar\cZ_N$ in Subsection \ref{subsec:strategy} and compute its second moment, then we have
		\begin{equation*}
		\begin{split}
		\bbE\left[(\widebar\cZ_N)^2\right]&=\left(q_{0,N}^N(\varphi,\psi)\right)^2+\\
		\frac{1}{N}\sum\limits_{r=1}^\infty&(\sigma_N^2)^r\sum\limits_{0<n_1<\cdots<n_r<N\atop x_1,\cdots,x_r\in\bbZ^2}\left(q_{0,n_1}^N(\varphi,x_1)\right)^2\Big(\prod\limits_{j=2}^r p_{n_{j-1},n_j}(x_{j-1},x_j)^2\Big)\left(q_{n_r,N}^N(x_r,\psi)\right)^2,
		\end{split}
		\end{equation*}
		where we have slightly abused the notation $p_{m,n}(x,y)$ for simple random walk on $\bbZ^2$.
		
		Next, we introduce a pair of a renewal process and a random walk $(\widebar\iota^{(N)}_k,\widebar S^{(N)}_k)_{k\geq0}$ on $(\bbN\cup\{0\})\times\bbZ^2$ starting at $(0,0)$ with one-step transition probability
		\begin{equation*}
		\bP\Big(\widebar\iota^{(N)}_1=n,\widebar S^{(N)}_1=x\Big):=\frac{p_n(x)^2}{\widebar R_N}\ind_{\{1,\cdots,N\}}(n)\quad\text{where}~\widebar R_N:=\sum\limits_{n=1}^N\bP(S_n=S'_n), 
		\end{equation*}
		and $S_n$ and $S'_n$ are two independent simple random walks on $\bbZ^2$. Note that
		\begin{equation*}
		\sum\limits_{x\in\bbZ^2}\bP\Big(\widebar\iota^{(N)}_1=n,\widebar S^{(N)}_1=x\Big)=\bP\Big(\widebar\iota^{(N)}_1=n\Big)=\frac{\bP(S_n=S'_n)}{\widebar R_N}\ind_{\{1,\cdots,N\}}(n)
		\end{equation*}
		and $\bP(S_n=S'_n)=\bP(S_{2n}=0)\sim\frac{1}{\pi n}$, and consequently $\widebar R_N\sim\log N/\pi$. Hence, the distributions of $\widebar\iota^{(N)}_1$ and $\iota^{(N)}_1$ in \eqref{new_renew} have the same asymptotics.
		
		Finally, similar to \eqref{connect_to_renew}, a straightforward computations yields that
		\begin{equation*}
		\begin{split}
		&\sum\limits_{x\in\bbZ}\bigg((\sigma_N^2)^2p_n(x)+\sum\limits_{k=1}^\infty(\sigma_N^2)^{k+2}\sum\limits_{n_0:=0<n_1<\cdots<n_k<n=:n_{k+1}\atop x_0:=0, x_1,\cdots,x_k, x_{k+1}:=x}\prod\limits_{j=1}^{k+1}p_{n_{j-1},n_j}(x_{j-1},x_j)^2\bigg)\\	
		=&\sigma_N^2\sum\limits_{x\in\bbZ^2}\sum\limits_{k=1}^\infty(\lambda_N)^k\bP\Big(\widebar\iota^{(N)}_k=n,\widebar S^{(N)}_k=x\Big)=\sigma_N^2\sum\limits_{k=1}^\infty(\lambda_N)^k\bP\Big(\widebar\iota^{(N)}_k=n\Big),
		\end{split}
		\end{equation*}
		which has the same asymptotics of $\overline{U}_N(n)$. Here, the key observation is that both $\bP(S_{2n}=0)$ for simple random walk on $\bbZ^2$ and $u(n)$ in \eqref{def:u} are asymptotically $cn^{-1}$.
	
\end{remark}

As an easy application, we can immediately compute the limits of the first two moments of $\tilde\cZ_{0,N}^{\beta_N}(\varphi,\psi)$. Recalling the heat kernel $g_t(x)$ from (\ref{heatkernel}), define
\begin{align}
\label{gphi}g_t(\varphi, a)&:=\int_\bbR\varphi(x)g_t(a-x)\dd x,\\
\label{gpsi}g_t(b,\psi)&:=\int_\bbR g_t(y-b)\psi(y)\dd y,\\
\label{gphipsi}g_t(\varphi,\psi)&:=\int_{\bbR\times\bbR}\varphi(x)g_t(y-x)\varphi(y)\dd x\dd y.
\end{align}
We omit the subscript $t$ when $t=1$.

\begin{proposition}\label{prop:mean_var}
	For any $\varphi\in C_c(\bbR)$ and $\psi\in C_b(\bbR)$, we have that
	\begin{align}
	\label{conv:mean}\lim\limits_{N\to\infty}\bbE\Big[\tilde\cZ_{0,N}^{\beta_N}(\varphi,\psi)\Big]&=g(\varphi,\psi),\\
	\label{conv:var}\lim\limits_{N\to\infty}\bbE\Big[\Big(\tilde\cZ_{0,N}^{\beta_N}(\varphi,\psi)\Big)^2\Big]&=g(\varphi,\psi)^2+\cV_1^\vartheta(\varphi,\psi)
	\end{align}
	where
	\begin{equation}\label{def:cV}
	\begin{split}
	\cV_1^\vartheta(\varphi,\psi)&=\iiiint_{\bbR^{\otimes4}}\varphi(x)\varphi(x')K_1^\vartheta(x,x';y,y')\psi(y)\psi(y')\dd x\dd x'\dd y\dd y'\\
	&\leq2\pi\|\varphi\|_{\infty}^2\|\psi\|^2_\infty\int_0^1 G_\vartheta(u)\dd u
	\end{split}
	\end{equation}
	with
	\begin{align}
	\label{def:K}K_t^\vartheta(x,x';y,y')&:=2\pi\iint_{0<r<s<t}g_r(x)g_r(x')G_\vartheta(s-r)g_{t-s}(y)g_{t-s}(y')\dd r\dd s.
	\end{align}
\end{proposition}

\begin{proof}
	The limit \eqref{conv:mean} directly follows by the local limit theorem and a Riemann sum approximation (recall \eqref{polynomial} and $\bbE[\zeta_n]=0$).
	
	For \eqref{conv:var}, we only need to treat the second line of \eqref{varN}. We have that
	\begin{equation}\label{eq:kernel_approx}
	\begin{split}
	q_{0,n}^N(\varphi,0)&=\sum\limits_{u\in\bbZ}\sqrt{N}\int_{\frac{u}{\sqrt{N}}}^{\frac{u+1}{\sqrt{N}}}\varphi(x)\bP(S_n=u)\dd x=\int_\bbR\varphi(x)g_{\frac{n}{N}}(x)\dd x+o_N(1)\\
	q_{n,N}^N(0,\psi)&=\sum\limits_{v\in\bbZ}\sqrt{N}\int_{\frac{v}{\sqrt{N}}}^{\frac{v+1}{\sqrt{N}}}\psi(y)\bP(S_{N-n}=v)\dd y=\int_\bbR\varphi(y)g_{1-\frac{n}{N}}(y)\dd y+o_N(1).
	\end{split}
	\end{equation}
	Then recall the definition and properties of $\overline{U}(n)$ from \eqref{def:U_bar}-\eqref{overlineUasym}. By a Riemann sum approximation, we have that for any $\epsilon>0$ and $\epsilon N\leq n\leq (1-\epsilon)N$,
	
	\begin{equation*}
	\begin{split}
	\lim\limits_{N\to\infty}&\sum\limits_{s,t\in\frac{1}{N}\bbN\cap[\epsilon,1-\epsilon]}\frac{1}{N^2}g_s(x)g_s(x')\overline{U}_N(t-s)g_{1-t}(y)g_{1-t}(y')\\
	=&2\pi\int_{\epsilon<s<t<1-\epsilon}g_s(x)g_s(x')G_\vartheta(t-s)g_{1-t}(y)g_{1-t}(y')\dd s\dd t.
	\end{split}
	\end{equation*}
	By \eqref{overlineU}, the contribution from $n\leq\epsilon N$ and $n\geq(1-\epsilon)N$ is bounded by a constant $C_{\varphi,\psi}$ timing $\int_{[0,\epsilon]\cup[1-\epsilon,1]}G_\vartheta(t)\dd t$, which converges to $0$ as $\epsilon\to0$. The proof is completed.
\end{proof}

\begin{remark}\label{rmk:h}
	{\rm
		It can be seen from above that if we introduce an external field $h_N$ in \eqref{def:dp}, then the proper choice should be $h_N=\hat{h}/\sqrt{N}$ with $\hat{h}\in\bbR$. In this case, $\lim_{N\to\infty}\bbE\big[\tilde{\cZ}_{0,N}^{\beta_N}(\varphi,\psi)\big]$ is shifted by
		\begin{equation*}
		\sum\limits_{r=1}^\infty(\hat{h})^r\idotsint\limits_{0<t_1<\cdots<t_r<1}g_{t_1}(\varphi,0)\Big(\prod\limits_{i=2}^{r}\frac{1}{\sqrt{t_{i}-t_{i-1}}}\Big)g_{1-t_r}(0,\psi)\dd t_1,\dots\dd t_r,
		\end{equation*}
		which is convergent by \cite[Lemma B.2]{CSZ16}. Note that as long as $h_N=o((\log N)^{-1})$, the asymptotics of $\sigma_N^2$ in \eqref{def:critwin} remains unchanged. Since $h_N$ only affects the mean, we can set $h_N\equiv0$ without loss of generality. 
	}
\end{remark}

\subsection{Coarse graining procedure}\label{sec:cg}
We now introduce the definitions and notations related to the \textit{mesoscopic time intervals}, the \textit{coarse-grained disorders} and the \textit{coarse-grained model} introduced in Subsection \ref{subsec:strategy}-\ref{outline2}.

Given $\epsilon\in(0,1)$ and $N\in\bbN$, we partition the time period $[1,N]$ into mesoscopic intervals
\begin{equation}\label{meso}
\cT_{N,\epsilon}(\rmi):=((\rmi-1)\epsilon N,\rmi\epsilon N],\quad \rmi=1,\cdots,\left\lfloor\frac{1}{\epsilon}\right\rfloor.
\end{equation}
For $\kappa>0$ in \eqref{assump:RW}, fix a threshold
\begin{equation}\label{Keps}
K_\epsilon:=\Big(\log\frac{1}{\epsilon}\Big)^{\frac{10}{\kappa\wedge1}}
\end{equation}
and we introduce a collection of indices for the mesoscopic intervals, which is
\begin{equation}\label{def:notriple}
\begin{split}
\Anotriple:=\bigcup\limits_{k\in\bbN}\Big\{&(\rmi_1,\cdots,\rmi_k)\in\bbN^k: K_\epsilon\leq\rmi_1<\cdots<\rmi_k\leq\left\lfloor\frac{1}{\epsilon}\right\rfloor-K_\epsilon,\\
&\text{such that if}~\rmi_{j+1}-\rmi_j<K_\epsilon, \text{then}~\rmi_{j+2}-\rmi_{j+1}\geq K_\epsilon\Big\}.
\end{split}
\end{equation}
This set is labeled by ``no triple'' since there is no consecutive three indices $\rmi_j, \rmi_{j+1},\rmi_{j+2}$ such that $\rmi_{j+1}-\rmi_j<K_\epsilon$ and $\rmi_{j+2}-\rmi_{j+1}<K_\epsilon$.

To proceed, we introduce
\begin{equation}\label{def:X}
X_{d,f}:=\begin{cases}
\zeta_d,&\quad\text{if}~d=f,\\
\zeta_d \tilde Z_{d,f}^{\omega, \beta_N}(0,0)\zeta_f,&\quad\text{if}~f\geq d+1,
\end{cases}
\end{equation}
where $\zeta_n$ is defined in \eqref{def:zeta} and $\tilde Z_{d,f}^{\omega, \beta_N}(0,0)$ is defined in \eqref{def:dp}. 

Now we can define a modified version of $\tilde\cZ_{0,N}^{\beta_N}(\varphi, \psi)$, which is
\begin{equation}\label{Znotriple}
\begin{split}
\Znotriple(\varphi,\psi):=q_{0,N}^N(\varphi,\psi)+&\frac{1}{\sqrt{N}}\sum\limits_{k=1}^{\left(\log\frac{1}{\epsilon}\right)^2}\sum\limits_{(\rmi_1,\cdots,\rmi_k)\in\Anotriple}\sum\limits_{d_1\leq f_1\in\cT_{N,\epsilon}(\rmi_1),\cdots,d_k\leq f_k\in\cT_{N,\epsilon}(\rmi_k)}\\
&q_{0,d_1}^N(\varphi,0)X_{d_1,f_1}\bigg\{\prod\limits_{j=2}^k p_{d_j-f_{j-1}}(0)X_{d_j,f_j}\bigg\}q_{f_k,N}^N(0,\psi),
\end{split}
\end{equation}
where we rule out the consecutive visits to mesoscopic intervals $\cT_{N,\epsilon}(\rmi_j), \cT_{N,\epsilon}(\rmi_{j+1})$ and $\cT_{N,\epsilon}(\rmi_{j+2})$ with $\rmi_{j+1}-\rmi_j<K_\epsilon$ and $\rmi_{j+2}-\rmi_{j+1}<K_\epsilon$. We will approximate $\tilde\cZ_{0,N}^{\beta_N}(\varphi,\psi)$ by  $\Znotriple(\varphi,\psi)$ in $L_2$.

Then we introduce the width and the ``distance'' for mesoscopic time intervals.
\begin{definition}\label{def:block}
	We call any pair $\vec{\rmi}=(\rmi,\rmi')\in\bbN\times\bbN$ with $\rmi\leq\rmi'$ a {\rm time block}. Its width is defined by
	\begin{equation*}
	|\vec{\rmi}|:=\rmi'-\rmi+1.
	\end{equation*}
	The (non-symmetric) ``distance'' between two time blocks $\vec{\rmi}=(\rmi,\rmi')$ and $\vec{\rmj}=(\rmj,\rmj')$ is defined by
	\begin{equation*}
	{\rm dist}(\vec{\rmi},\vec{\rmj})=\rmj-\rmi'.
	\end{equation*}
	We write $\vec{\rmi}<\vec{\rmj}$ if and only if $\rmi'<\rmj$. For some constant $K$, we write $K\leq\vec{\rmi}$ if and only if $K\leq\rmi$ and $\vec{\rmi}\leq K$ if and only if $\rmi'\leq K$.
\end{definition}

With the definitions above, we can pair $\rmi_j$ and $\rmi_{j+1}$ for $\rmi_{j+1}-\rmi_j<K_\epsilon$ as $\vec{\rmi}_j$ and by noting the ``no triple'' condition, we can rewrite $\cA_\epsilon^{\text{(no triple)}}$ by
\begin{equation}\label{def:notriple2}
\begin{split}
\vAnotriple:=\bigcup\limits_{r\in\bbN}\Big\{&(\vec{\rmi}_1,\cdots,\vec{\rmi}_r)\in\bbN^r: K_\epsilon\leq\vec{\rmi}_1<\cdots<\vec{\rmi}_k\leq\Big\lfloor\frac{1}{\epsilon}\Big\rfloor-K_\epsilon, ~\text{such that}\\
&|\vec{\rmi}_j|<K_\epsilon, \forall~j=1,\cdots,r, ~{\rm dist}(\vec{\rmi}_{j-1},\vec{\rmi}_j)\geq K_\epsilon, \forall~j=2,\cdots,r\Big\}.
\end{split}
\end{equation}

We then bind $X_{d,f}$ and $X_{d',f'}$ together if their mesoscopic distance is smaller than $K_\epsilon$, which gives the definition of the coarse-grained variable. To be precise, let
\begin{equation}\label{cgvariable}
\Theta(\vec{\rmi})=\Theta_{N,\epsilon}(\vec{\rmi}):=\begin{cases}
\frac{1}{\sqrt{\epsilon N}}\sum\limits_{d\leq f\in\cT_{N,\epsilon}(\rmi)}X_{d,f},&\text{if}~|\vec{\rmi}|=1,\\
\frac{1}{\sqrt{\epsilon N}}\sum\limits_{d\leq f\in\cT_{N,\epsilon}(\rmi)}\sum\limits_{d'\leq f'\in\cT_{N,\epsilon}(\rmi')}\big(X_{d,f}\times p_{d'-f}(0)\times X_{d',f'}\big),~~&\text{if}~|\vec{\rmi}|>1.
\end{cases}
\end{equation}

Now we introduce a second modified version of $\tilde\cZ_{0,N}^{\beta_N}(\varphi,\psi)$, which is
\begin{equation}\label{cg_Z}
\begin{split}
\Zcg(\varphi,&\psi):=g_1(\varphi,\psi)+\sqrt{\epsilon}\sum\limits_{r=1}^{(\log\frac{1}{\epsilon})^2}\sum\limits_{(\vec{\rmi}_1,\cdots,\vec{\rmi}_r)\in\vAnotriple}\sum\limits_{\rma,\rmb\in\bbZ}\\
&\varphi_{N,\epsilon}(\rma)g_{\rmi_1}(\rma)\Theta(\vec{\rmi}_1)\times\bigg\{\prod_{j=2}^r g_{(\rmi_j-\rmi'_{j-1})}(0)\Theta(\vec{\rmi}_j)\bigg\}g_{(\lfloor\frac{1}{\epsilon}\rfloor-\rmi'_r)}(\rmb)\psi_{N,\epsilon}(\rmb),
\end{split}
\end{equation}
where $g_t(x)$ is the heat kernel defined in \eqref{heatkernel}, and

\begin{equation}\label{Nepstestfunc}
\begin{split}
\varphi_{N,\epsilon}(\rma)&:=\frac{1}{\sqrt{\epsilon N}}\sum\limits_{u\in(\rma\sqrt{\epsilon N},(\rma+1)\sqrt{\epsilon N}]\cap\bbZ}\varphi_N(u),\\
\quad\psi_{N,\epsilon}(\rmb)&:=\frac{1}{\sqrt{\epsilon N}}\sum\limits_{v\in(\rmb\sqrt{\epsilon N},(\rmb+1)\sqrt{\epsilon N}]\cap\bbZ}\varphi_N(v).
\end{split}
\end{equation}
We will approximate $\Znotriple(\varphi,\psi)$ by $\Zcg(\varphi,\psi)$ in $L_2$.

Finally, we introduce the coarse-grained model $\Lcg(\varphi,\psi|\Theta)$. We define
\begin{equation}\label{cgmodel}
\begin{split}
\Lcg(\varphi,\psi|&\Theta)=\Lcg(\varphi,\psi|\Theta_{N,\epsilon}):=g_1(\varphi,\psi)+\sqrt{\epsilon}\sum\limits_{r=1}^{(\log\frac{1}{\epsilon})^2}\sum\limits_{(\vec{\rmi}_1,\cdots,\vec{\rmi}_r)\in\vAnotriple}\sum\limits_{\rma,\rmb\in\bbZ}\\
&\varphi_\epsilon(\rma)g_{\rmi_1}(\rma)\Theta_{N,\epsilon}(\vec{\rmi}_1)\times\bigg\{\prod_{j=2}^r g_{(\rmi_j-\rmi'_{j-1})}(0)\Theta_{N,\epsilon}(\vec{\rmi}_j)\bigg\}g_{(\lfloor\frac{1}{\epsilon}\rfloor-\rmi'_r)}(\rmb)\psi_\epsilon(\rmb),
\end{split}
\end{equation}
where
\begin{equation}\label{epstestfunc}
\varphi_\epsilon(\rma):=\int_{\rma}^{\rma+1}\varphi(\sqrt{\epsilon}x)\dd x,\quad\psi_\epsilon(\rmb):=\int_{\rmb}^{\rmb+1}\psi(\sqrt{\epsilon}y)\dd y,\quad\text{for}~\rma,\rmb\in\bbZ.
\end{equation}
We will approximate $\Zcg(\varphi,\psi)$ by $\Lcg(\varphi,\psi|\Theta)$ in $L_2$. For $\iint_{\bbR^2}\!\varphi(x)\tilde{u}_{0,1}^\delta\!(\dd x,\dd y)\psi(y)$, we will also approximate it by
$\Lcg(\varphi,\psi|\Theta_{\delta,\epsilon})$, where $\Theta_{\delta,\epsilon}$ is a continuum version of coarse-grained variables (see \eqref{def:she_theta}).

\subsection{Organizations of the paper}
The rest of the paper is organized as follows. In Section \ref{S2}, we complete the proof of Theorem \ref{thm:a}. In Section \ref{S4}, we estimate the higher moments of $\tilde\cZ_{0,N}^{\beta_N}(\varphi,\psi)$. In Section \ref{S6}, we approximate $\tilde\cZ_{0,N}^{\beta_N}(\varphi,\psi)$ by the coarse-grained model. In Section \ref{S8}, we bound the fourth moment of the coarse-grained model, where a first step is to bound the moments of the coarse-grained disorders. In Section \ref{S9}, we prove Theorem~\ref{thm:2} by showing the convergence of the coarse-grained model, where an enhanced Lindeberg principle is applied. In Section \ref{S:sde}, we prove Theorem \ref{thm:2+}.

Since Theorem \ref{thm:a} proved in Section \ref{S2} is novel, and the continuous setting treated in Section \ref{S:sde} was not addressed in \cite{CSZ21}, we provide enough details for both sections. In contrast, the proofs in Sections~\ref{S4}-\ref{S9} are adapted from \cite{CSZ21}. Section~\ref{S9} is written with full details as the corresponding part in \cite{CSZ21} is only outlined. For the remaining three sections, we explain the strategies in detail but only sketch some computations, so that the paper is convincing while maintaining in a reasonable length. We also point out the major differences and modifications between \cite{CSZ21} and our paper. To make the paper self-contained, we include the enhanced Lindeberg principle in Appendix~\ref{A1} and some local limit estimates in Appendix~\ref{A2}.

\subsection{Notations} Here we collect some important notations in our paper:

$\bullet$ $\zeta_n$: The disorder in polynomial chaos expansion, see \eqref{def:zeta}.

$\bullet$ $\sigma_N^2$: The variance of $\zeta_n$, see \eqref{zeta_mean_var}.

$\bullet$ $\cZ_N^{\beta_N}(\dd s,\dd t)$: The pinning measure, see \eqref{def:pinmeasure-}.

$\bullet$ $\tilde\cZ_{0,N}^{\beta_N}(\dd x,\dd y)$: The directed polymer measure, see \eqref{def:dpmeasure}.

$\bullet$ $\ov U_N(n)$: A renewal structure related to the Dickman subordinator, see \eqref{def:U_bar}.

$\bullet$ $G_\vartheta(t)$: The asymptotics of $\ov U_N(n)$, see \eqref{Gtheta}.

$\bullet$ $X_{d,f}$: The directed polymer partition function $\tilde Z_{d,f}^{\omega,\beta_N}(0,0)$ equipped with boundary disorders $\zeta_d$ and $\zeta_f$, see \eqref{def:X}.



$\bullet$ $\Theta_{N,\epsilon}(\vec{\rmi})$: The coarse-grained variable, see \eqref{cgvariable}.

$\bullet$ $\Lcg(\varphi,\psi|\Theta)$: The coarse-grained model, see \eqref{cgmodel}.

\section{Proof of Theorem \ref{thm:a}}\label{S2}
We are going to show the equivalence between the two convergences (i) and (ii) in Theorem \ref{thm:a}. From Subsection \ref{outline1}, it suffices to work on $\widetilde\cZ_{0,N}^{\beta_N}(\varphi,\psi)$ and $\cZ_N^{\beta_N}(f,h)$ (see \eqref{def:dpintegral4} and \eqref{def:pinint}), which can be expressed by

\begin{equation}\ba{l}\label{def:dpintegral4+}
	\dis\qquad\widetilde\cZ_{0,N}^{\beta_N}(\varphi,\psi)= \iint_{\bbR\times\bbR} \varphi(x) \sqrt{N} \hspace{1pt}\bP^{x_N}(S_N=y_N,\tau_1\geq N) \psi(y) \dd x\dd y \\ [5pt]
	\dis \quad+\sum\limits_{m=1}^{N-1}\sum\limits_{n=m}^{N-1} \Big(\int_{\bbR}\varphi(x)\sqrt{N}q_{x_N}(m)\dd x\Big)\Big(\sqrt{N} Z^{\omega,\beta_N}_{m,n}\Big) \Big(\int_{\bbR}\psi(x)\sqrt{N}q_{-y_N}(N-n)\dd y\Big),
\ea\end{equation}
\vspace{-8pt}\begin{equation}
	\label{def:pinint+}
	\dis\cZ_N^{\beta_N}(f,h)=\sum\limits_{m=0}^{N-1}\sum\limits_{n=m}^{N-1}\Big(\int_{\frac{m}{N}}^{\frac{m+1}{N}}f(s)\dd s\Big)\Big(\sqrt{N}Z_{m,n}^{\omega,\beta_N}\Big)\Big(\int_{\frac{n}{N}}^{\frac{n+1}{N}}h(t)\dd t\Big).
\end{equation}
We need to show that the convergence of (\ref{def:pinint+}) for all $f,h\in C([0,1])$) implies the convergence of (\ref{def:dpintegral4+}) for all $\varphi, \psi\in C_c(\bbR)$, and vice versa. To this end, we need the following ingredients.

%

\textbf{Part (a):} the first (deterministic) term in the right-hand side of \eqref{def:dpintegral4+} converges as $N\to\infty$;

\textbf{Part (b):} the term
\begin{equation}\label{def:error}
\mathscr{R}_N:=\mathscr{R}_N(f,h)=\sqrt{N}\int_{0}^{\frac{1}{N}}f(s)\dd s\sum_{n=0}^{N-1}Z_{0,n}^{\omega,\beta_N}\int_{\frac{n}{N}}^{\frac{n+1}{N}}h(t)\dd t,
\end{equation}
which refers to the sum $m\leq n\leq N-1$ and $m=0$ in (\ref{def:pinint+}), converges to $0$ in $L_2$ as $N\to\infty$;

\textbf{Part (c):} Given $f,h\in C([0,1])$ (Given $\varphi, \psi\in C_c(\bbR)$ resp.), for any $\epsilon>0$, we can find $\varphi, \psi\in C_c(\bbR)$ (we can find $f, h\in C([0,1])$ resp.), such that
\begin{align}
\label{int_approx1}\limsup\limits_{N\to\infty}\max\limits_{1\leq m\leq N}&N\bigg|\int_{\bbR}\varphi(x)\sqrt{N}q_{x_N}(m)\dd x-\int_{\frac{m}{N}}^{\frac{m+1}{N}}f(s)\dd s\bigg|\leq\epsilon,\\
\label{int_approx2}\limsup\limits_{N\to\infty}\max\limits_{1\leq m\leq N}&N\bigg|\int_{\bbR}\psi(x)\sqrt{N}q_{-y_N}(N-n)\dd y-\int_{\frac{n}{N}}^{\frac{n+1}{N}}h(t)\dd t\bigg|\leq\epsilon.
\end{align}

To be specific, applying {\bf Part (a)} to $\widetilde\cZ_{0,N}^{\beta_N}(\varphi,\psi)$ and {\bf Part (b)} to $\cZ_N^{\beta_N}(f,h)$, showing (i)$\Longleftrightarrow$(ii) is then reduced to showing the equivalence of convergences of the two double sums for $1\leq m\leq n\leq N-1$ in (\ref{def:dpintegral4+}) and (\ref{def:pinint+}). This, by comparing the two expressions, can be deduced from the key approximation {\bf Part (c)}. We first show the three statements above.


\vspace{5pt}
We need the following refined estimate \cite[Corollary 1.1]{U11} for the first hitting times  of random walks with a finite second moment: as $N\to\infty$, uniformly in $x\in\frac{1}{\sqrt{N}}\bbZ$ and $s\in\frac{1}{N}\bbN$,
\begin{equation}\label{LLT:renew}
q_{x_N}(sN)=\frac{|x|}{\sqrt{2\pi}s^{3/2}N}e^{-\frac{x^2}{2s}}+o\Big(\frac{|x|\sqrt{N}\!\vee\!1}{(sN)^{3/2}}\!\wedge\!\frac{1}{x^2 N\!\vee\!1}\Big)=:\!\frac{1}{N}Q(x,s)\!+\!\frac{1}{N}R_N(x,s).
\end{equation}

\noindent\textbf{Proof for Part (a).} We have already shown in Proposition \ref{prop:mean_var} that
\begin{equation*}
\lim\limits_{N\to\infty}\bbE\big[\widetilde\cZ_N^{\beta_N}(\varphi,\psi)\big]=g(\varphi,\psi).
\end{equation*}
Hence, to prove \textbf{Part (a)}, it suffices to prove that the expectation of the double summation in \eqref{def:dpintegral4+}, that is,
\begin{equation}\label{eq:2ndline}
\sum\limits_{m=1}^{N-1}\sum\limits_{n=m}^{N-1}\Big(\int_{\bbR}\varphi(x)\sqrt{N}q_{x_N}(m)\dd x\Big)\Big(\sqrt{N}p_{m,n}(0)\Big) \Big(\int_{\bbR}\psi(x)\sqrt{N}q_{-y_N}(N-n)\dd y\Big),
\end{equation}
also converges as $N\to\infty$.

By writing $m=sN$ and $n=tN$ and by \cite[Corollary 1.1]{U11}, the above quantity equals to
\begin{equation}\label{approx}
\begin{split}
\sum\limits_{s<t\in\{\frac{1}{N},\cdots,\frac{N-1}{N}\}}&\sum\limits_{u,v\in\frac{1}{\sqrt{N}}\bbZ}\frac{1}{\sqrt{N}}\varphi(u)\frac{1}{N}\big(Q(|u|,s)+R_N(|u|,s)\big)\\
&\times\frac{1+o(1)}{\sqrt{2\pi(t-s)}}\frac{1}{N}\Big(Q(|v|,1-t)+R_N(|v|,1-t)\Big)\frac{1}{\sqrt{N}}\psi(v).
\end{split}
\end{equation}

For the main term (involving $Q$ only), by a Riemann sum approximation, it converges to
\begin{equation}\label{def:sE}
\begin{split}
&\sE(\varphi,\psi):=\\
\frac{1}{(2\pi)^{\frac{3}{2}}}\!\iint_{0<s<t<1}&\iint_{\bbR\times\bbR}\!\varphi(u)\frac{|uv|}{s\sqrt{s(t-s)(1-t)}(1-t)}e^{-\frac{1}{2}\left(\frac{u^2}{s}+\frac{v^2}{1-t}\right)}\psi(v)\dd u\dd v\dd t\dd s.
\end{split}
\end{equation}
Note that by a change of variables $x=u/\sqrt{s}$ and $y=v/\sqrt{1-t}$, we have that

\begin{equation*}
|\sE(\varphi,\psi)|\leq\frac{\|\varphi\|_{\infty}\|\psi\|_{\infty}}{(2\pi)^{\frac32}}\iint_{0<s<t<1}\frac{\dd s\dd t}{\sqrt{s(t-s)(1-t)}}\iint_{\bbR^2}|xy|e^{-\frac{x^2+y^2}{2}}\dd x\dd y<+\infty,
\end{equation*}
where the finiteness follows from \cite[Lemma B.2]{CSZ16}. Hence, $\sE(\varphi,\psi)$ is well defined.

For the error terms (which involves $R_N$), we only treat the term with $u,s$. The other term with $v,t$ is exactly the same. Note that if $|u|\leq\sqrt{s}$, then $(sN)^{\frac{3}{2}}\geq((u^2N)^{\frac{3}{2}}\vee1)$ since $s\geq\frac{1}{N}$. The error term is bounded by
\begin{equation}\label{error2}
\begin{split}
&\sum\limits_{u\in\frac{1}{\sqrt{N}}\bbZ\atop|u|\leq\sqrt{s}}\frac{1}{\sqrt{N}}\varphi(u)|u|\frac{o(1)}{s^{\frac32}N}+\sum\limits_{u\in\frac{1}{\sqrt{N}}\bbZ\atop|u|>\sqrt{s}}\frac{1}{\sqrt{N}}\frac{\varphi(u)}{|u|^2}\frac{o(1)}{N}\\
\leq&\frac{o(1)}{s^{\frac32}N}\int_{|u|\leq\sqrt{s}}|u|\dd u+\frac{o(1)}{N}\int_{|u|\geq\sqrt{s}}\frac{1}{|u|^2}\dd u=\frac{o(1)}{\sqrt{s}N},
\end{split}
\end{equation}
which is integrable with respect to $s$ at $0$. Hence, the error term is negligible.

By combining \eqref{def:sE} and \eqref{error2}, it follows that
\begin{equation*}
\begin{split}
\lim\limits_{N\to\infty}\iint_{\bbR\times\bbR} \varphi(x) \sqrt{N} \hspace{1pt}\bP^{x_N}(S_N=y_N,\tau_1\geq N) \psi(y) \dd x\dd y=g(\varphi,\psi)-\sE(\varphi,\psi),
\end{split}
\end{equation*}
which concludes \textbf{Part (a)}. 

\vspace{0.2cm}
\noindent\textbf{Proof for Part (b).} It suffices to show that for $\sR_N$ in \eqref{def:error}, $\lim_{N\to\infty}\bbE[\mathscr{R}_N]=0$ and $\lim_{N\to\infty}\var(\mathscr{R}_N)=0$.

Since $\bbE[\zeta_n]=0$, the local limit theorem and a Riemann sum approximation yields that
\begin{equation}\label{eq:mean_error}
\begin{split}
\bbE[\sR_N]&=\sqrt{N}\int_0^{\frac{1}{N}}f(s)\dd s\sum\limits_{n=0}^{N-1}p_n(0)\int_{\frac{n}{N}}^{\frac{n+1}{N}}h(t)\dd t\sim\frac{1}{N}f(0)\int_{0}^{1}\frac{h(t)}{\sqrt{t}}\dd t=\frac{O(1)}{N}.
\end{split}
\end{equation}

For the variance, perform a polynomial chaos expansion to $Z_{0,n}^{\omega,\beta_N}$ (see Subsection \ref{sub:poly}) , and note that $\bbE[\sR_N]$ is actually the deterministic term in the polynomial chaos expansion of $\sR_N$. We have that
\begin{equation*}
\begin{split}
&\sR_N-\bbE[\sR_N]=\\
\sqrt{N}\int_{0}^{\frac{1}{N}}f(s)\dd s&\sum\limits_{n=0}^{N-1}\sum\limits_{k=1}^{\infty}\sum\limits_{n_0:=0\leq n_1<\cdots<n_k:=n}\zeta_0^{\ind_{\{n\neq0\}}}\Big(\prod\limits_{j=1}^k p_{n_{j-1},n_j}(0)\zeta_{n_j}\Big)\int_{\frac{n}{N}}^{\frac{n+1}{N}}h(t)\dd t.
\end{split}
\end{equation*}
Taking square and expectation, and by the independence of $(\zeta_n)_{n\in\bbZ}$, we have that
\begin{equation}\label{eq:varRN}
\begin{split}
&\var(\sR_N)=\\
N\Big(\int_0^{\frac{1}{N}}\!\!f(s)\dd s\Big)^2&\sum\limits_{n=0}^{N-1}\sum\limits_{k=1}^\infty\big(\sigma_N^2\big)^{k+1}\!\!\!\!\sum\limits_{n_0:=0\leq n_1<\cdots<n_k:=n}\!\!\Big(\prod_{j=1}^k u(n_j-n_{j-1})\Big)\!\Big(\int_{\frac{n}{N}}^{\frac{n+1}{N}}\!\!h(t)\dd t\Big)^2,
\end{split}
\end{equation}
where $\sigma_N^2=\var(\zeta_n)$ (see \eqref{zeta_mean_var}) and $u(n)$ is defined in \eqref{def:u}. 

Recall the definition and properties of $\overline{U}_N(n)$ from \eqref{def:U_bar}-\eqref{asym:Gtheta}. Then \eqref{eq:varRN} is bounded from above by
\begin{equation}\label{bound:varRN}
\begin{split}
\frac{1}{N^3}\|f\|_\infty^2\|h\|_\infty^2\sum\limits_{n=0}^{N-1}\overline{U}(n)\leq\frac{C_{f,h}}{N^3}\sum\limits_{n=0}^{N-1}\frac{1}{N}G_\vartheta\Big(\frac{n}{N}\Big)\leq\frac{C}{N^3}\int_0^1 G_\vartheta(t)\dd t=\frac{O(1)}{N^3}.
\end{split}
\end{equation}

By Combining \eqref{eq:mean_error} and \eqref{bound:varRN}, we have shown that $\sR_N$ converges to $0$ in $L_2$.

\vspace{0.2cm}
\noindent\textbf{Proof for Part (c)}. We only treat the case \eqref{int_approx1}. The proof for the case \eqref{int_approx2} is totally the same. Recall \eqref{LLT:renew} and note that we have proved in \textbf{Part (a)} that $R_N$ has a smaller order than $Q$ (see \eqref{error2}). Hence, it suffice to treat
\begin{equation}\label{int_approx1:a}
\sum\limits_{u\in\frac{1}{\sqrt{N}}\bbZ}\Big(\int_u^{u+\frac{1}{\sqrt{N}}}\varphi(x)\dd x\Big)\frac{1}{\sqrt{N}}Q\Big(u,\frac{m}{N}\Big)=\sum\limits_{u\in\frac{1}{\sqrt{N}}\bbZ}\Big(\int_u^{u+\frac{1}{\sqrt{N}}}\varphi(x)\dd x\Big)\frac{|u|}{\sqrt{2\pi}m^{\frac32}}e^{-\frac{(\sqrt{N}u)^2}{2m}}.
\end{equation}

Given $\varphi\in C_c(\bbR)$, by its uniform continuity, \eqref{int_approx1:a} can be written by
\begin{equation*}
\frac{1}{\sqrt{N}}\sum\limits_{u\in\frac{1}{\sqrt{N}}\bbZ}(\varphi(u)+o(1))\frac{|u|}{\sqrt{2\pi}m^{\frac32}}e^{-\frac{(\sqrt{N}u)^2}{2m}}=\frac{1}{N}\sum\limits_{x\in\bbZ}\Big(\varphi\Big(\frac{x}{\sqrt{N}}\Big)+o(1)\Big)\frac{|x|}{\sqrt{2\pi}m^{\frac32}}e^{-\frac{x^2}{2m}}.
\end{equation*}
Hence, by $\lim_{s\to0}\frac{|x|}{s^{3/2}}\exp(-x^2/s)=0$ for all $x$, we can take
\begin{equation*}
f(s)=\sum\limits_{x\in\bbZ}\varphi(0)\frac{|x|}{\sqrt{2\pi}s^{\frac32}}e^{-\frac{x^2}{2s}}.
\end{equation*}

Conversely, fix $f\in C([0,1])$. Note that for $\varphi(x)=|x|^{2k+1}$,
\begin{equation}\label{hkint}
\int_{\bbR} \varphi(x)Q(x,s)\dd x=\int_\bbR\frac{x^{2k+2}}{\sqrt{2\pi}s^{\frac32}}e^{-\frac{x^2}{2s}}\dd x =(2k+1)!!s^k,
\end{equation}
which is a multiple of a power of $s$. By Weierstrass approximation theorem, for any $f\in C([0,1])$ and $\eps>0$, we can
choose $\tilde\varphi_\eps(x)$ of the form $\sum_{k=1}^{A_\eps}C_k|x|^{2k+1}$ such that
\begin{equation*}
\sup_{s\in[0,1]}\Big|\int_{\bbR} \tilde\varphi_\eps(x)Q(x,s)\dd x-f(x)\Big|<\eps,
\end{equation*}
where $A_\epsilon$ is some positive integer. Since $Q(x,s)$ has a double exponential decay in $x$, we can truncate and mollify $\tilde{\varphi}_\epsilon$ to obtain some $\varphi_\epsilon\in C_c(\bbR)$, such that
\begin{equation}
\sup\limits_{s\in[0,1]}\Big|\int_{\bbR} \tilde\varphi_\eps(x)Q(x,s)\dd x-\int_{\bbR} \varphi_\eps(x)Q(x,s)\dd x\Big|<\epsilon,
\end{equation}
and we have found the desired $\varphi_\epsilon$.

\vspace{0.2cm}
We are ready to prove Theorem \ref{thm:a}. In particular, we show (i)$\Longleftrightarrow$(ii) therein.
	
We start by (i)$\Longrightarrow$(ii). Recall $Q(x,s)$ and $R_N(x,s)$ from \eqref{LLT:renew}. Let
\begin{align*}
Q(\varphi, s)&:=\int_\bbR\varphi(x)Q(x,s)\dd x,\\
\tilde{Q}(\psi, s)&:=\int_\bbR\psi(-y)Q(-y,s)\dd y,\\
R_N(\varphi, s)&:=\int_\bbR\varphi(x)R_N(x,s)\dd x.
\end{align*}
It is not hard to see (for example, in the proof for \textbf{Part (c)}) that $Q(\varphi,s)$ and $\tilde{Q}(\psi,s)$ are continuous in $s$.
To show (ii), namely the directed polymer partition function $\tilde{\cZ}_{0,N}^{\beta_N}(\varphi,\psi)$ in (\ref{def:dpintegral4+}) converges, we note that the first term in the right-hand side converges due to {\bf Part (a)}. It remains to consider the second term
\begin{equation}
\begin{split}
&\sum\limits_{m=1}^{N-1}\sum\limits_{n=m}^{N-1}\Big(\int_{\bbR}\varphi(x)\sqrt{N}q_{x_N}(m)\dd x\Big)\Big(\sqrt{N} Z^{\omega,\beta_N}_{m,n}\Big) \Big(\int_{\bbR}\psi(x)\sqrt{N}q_{-y_N}(N-n)\dd y\Big)\\	
=&\frac{1}{N^2}\sum_{s\leq t \in\frac{1}{N}\bbZ\cap(0,1)}\Big(Q+R_N\Big)(\varphi,s)\Big(\sqrt{N}Z^{\omega,\beta_N}_{sN,tN} \Big)\Big(\tilde{Q}+R_N\Big)(\psi,1-t).\\
\end{split}
\end{equation}

Assuming (i), the convergence of $\dis\cZ_N^{\beta_N}$ in (\ref{def:pinint+}) with $f(s)=Q(\varphi,s)$ and $h(t)=\tilde Q(\psi, 1-t)$  holds.
Removing the term referring to $m=0$ thanks to {\bf Part (b)}, we have
\begin{equation}\label{eq:ztoz}
\begin{split}
&\frac{1}{N^2}\sum_{s\leq t \in\frac{1}{N}\bbZ\cap(0,1)}Q(\varphi,s) \Big(\sqrt{N}Z^{\omega,\beta_N}_{sN,tN} \Big)\tilde{Q}(-\psi,1-t)\\
&\xlongrightarrow[N\to\infty]{(\dd)}\iint\limits_{0<s\leq t<1}Q(\varphi,s)\sL^\vartheta(\dd s,\dd t)\tilde{Q}(\psi,1-t).
\end{split}
\end{equation}
Similar to \eqref{error2}, we can show that all error terms converge to $0$. 
We thus conclude (i)$\Longrightarrow$(ii).

\vspace{0.2cm}
Then we show (ii)$\Longrightarrow$(i). Similarly, we only need to consider the two double sums. But now for given $f,h\in C([0,1])$, it is not straightforward to find $\varphi,\psi$ and use the convergence of $\widetilde\cZ_{0, N}^{\beta_N}(\varphi,\psi)$.
Instead, we exploit {\bf Part (c)} that $Q(\varphi, s)$ and $\tilde{Q}(\psi, s)$ are dense in $C([0,1])$.
Precisely, we write the difference of the two double sums
\begin{equation}\label{eq:Z_4func}
	\ba{l}
	\dis \cZ_N^{\beta_N}(f,h,\varphi,\psi):=
	\sum\limits_{m=1}^{N-1}\sum\limits_{n=m}^{N-1}\Big(\int_{\frac{m}{N}}^{\frac{m+1}{N}}f(s)\dd s\Big)\Big(\sqrt{N}Z_{m,n}^{\omega,\beta_N}\Big)\Big(\int_{\frac{n}{N}}^{\frac{n+1}{N}}h(t)\dd t\Big) \\ [5pt]
	\dis\quad -\sum\limits_{m=1}^{N-1}\sum\limits_{n=m}^{N-1} \Big(\int_{\bbR}\varphi(x)\sqrt{N}q_{x_N}(m)\dd x\Big)\Big(\sqrt{N} Z^{\omega,\beta_N}_{m,n}\Big) \Big(\int_{\bbR}\psi(x)\sqrt{N}q_{-y_N}(N-n)\dd y\Big),
\ea\end{equation}
We then show that given $\epsilon>0$, for any $f,h\in C([0,1])$, we can find $\varphi,\psi\in C_c(\bbR)$, such that
\begin{equation}\label{eq:L2_error}
\limsup\limits_{N\to\infty}\bbE\Big[\big(\cZ_N^{\beta_N}(f,h,\varphi,\psi)\big)^2\Big]\leq\epsilon.
\end{equation}
The weak convergence of $\cZ_N^{\beta_N}(f,h)$ then follows from the weak convergence of $\tilde{\cZ}_{0,N}^{\beta_N}(\varphi,\psi)$, \eqref{eq:L2_error}, and {\bf Part (a)-(b)}.

	To prove \eqref{eq:L2_error}, we write $I_m^{(N)}(f):=I_m(f)=\int_{\frac{m}{N}}^{\frac{m+1}{N}}f(s)\dd s$
	and we have that
	\begin{equation}\label{keybound}
	\begin{split}
	&\quad\quad\quad\cZ_N^{\beta_N}(f,h,\varphi,\psi)=\\
	&\sqrt{N}\sum\limits_{m=1}^{N-1}\sum\limits_{n=m}^{N-1}I_m(f)\zeta_m^{\ind_{\{n\neq m\}}}\bigg[p_{m,n}(0)\zeta_n+\sum\limits_{k=1}^{\infty}\sum\limits_{n_0:=m\leq n_1<\cdots<n_k=n}\Big(\prod_{j=1}^k p_{n_{j-1},n_j}(0)\zeta_{n_j}\Big)\bigg]\\
	&\times\bigg(I_n(h)-\tilde{Q}_N\Big(\psi,\frac{N-n}{N}\Big)\bigg)+\sqrt{N}\sum\limits_{m=1}^{N-1}\sum\limits_{n=m}^{N-1}\bigg(I_m(f)-Q_N\Big(\varphi,\frac{m}{N}\Big)\bigg)\zeta_m^{\ind_{\{n\neq m\}}}\\
	&\times\bigg[p_{m,n}(0)\zeta_n+\sum\limits_{k=1}^{\infty}\sum\limits_{n_0:=m\leq n_1<\cdots<n_k=n}\Big(\prod_{j=1}^k p_{n_{j-1},n_j}(0)\zeta_{n_j}\Big)I_n(h)\bigg].
	\end{split}
	\end{equation}
	
	For any $\epsilon>0$ and large enough $N$, we have shown that by \eqref{int_approx1} and \eqref{int_approx2}, we can find $\varphi, \psi$, such that
	\begin{equation}\label{eq:fg_pp_error}
	\max_{1\leq m\leq n\leq N-1}\bigg\{~\bigg|I_m(f)-Q_N\Big(\varphi,\frac{m}{N}\Big)\bigg|,\quad\bigg|I_n(h)-\tilde{Q}_N\Big(\psi,\frac{N-n}{N}\Big)\bigg|~\bigg\}\leq\frac{\epsilon}{N}.
	\end{equation}
Then by Proposition \ref{prop:mean_var}, the second moment of $\cZ_N^{\beta_N}(f,h,\varphi,\psi)$ is bounded above by
	\begin{equation*}
	C\epsilon\big(\|f\|_\infty^2+\|h\|_\infty^2\big)\int_0^1 G_\vartheta(t)\dd t,
	\end{equation*}
	which concludes the proof.

\section{Moment estimates for $\cZ_{0,N}^{\beta_N}(\varphi,\psi)$}\label{S4}
In this subsection, we bound the higher central moments of $\tilde\cZ_{0,N}^{\beta_N}(\varphi,\psi)$. From now on, we work on $\widetilde\cZ_{0,N}^{\beta_N}(\varphi,\psi)$ in \eqref{def:dpintegral4}. When there is no ambiguity, we omit the tilde in notations, and simply write $\cZ_{0,N}^{\beta_N}(\varphi,\psi)$ by $\cZ_N^{\beta_N}(\varphi,\psi)$ or $\cZ_N$. We have the following Theorem \ref{T:hmom}. A technical novelty is that, when proving Proposition \ref{prop:Opbounds} below, we relax the assumption on the random walk to \eqref{assump:RW}, which requires a more careful analysis to utilize the local limit theorem.
\begin{theorem}\label{T:hmom}
	For $N\leq\tilde{N}\in\bbN$, consider $\cZ_N^{\beta_{\tilde{N}}}(\varphi, \psi)$ as \eqref{polynomial} with $\beta_N$ replaced by $\beta_{\tilde{N}}=\beta_{\tilde{N}}(\vartheta)$ satisfying \eqref{def:critwin}. Fix $p,q\in(1+\infty)$ with $\frac1p+\frac1q=1$ and $h\geq3$. Then there exist constants $C_1, C_2\in(0,+\infty)$ depending on $p, q, h$ and $\omega$, such that for $\varphi, \psi\in C_c(\bbR)$, and uniformly in $N\leq\tilde{N}$,
	\begin{equation}\ba{rl}
	\dis\Big|\bbE\Big[\Big( \cZ^{\beta_{\tilde{N}}}_N(\varphi,\psi) - \bbE\big[\cZ^{\beta_{\tilde{N}}}_N(\varphi,\psi)\big] \Big)^h \Big] \Big|
	&\dis\leq \frac{C_1}{\log(1+\frac{\tilde N}{N})} \frac{1}{N^{\frac{h}{2}}}\lVert \varphi_N\rVert^h_{p}\lVert\psi_N\rVert^h_{\infty}\lVert\ind_{B_N}\rVert^h_{q} \\ [5pt]
	&\dis \leq \frac{C_2}{\log(1+\frac{\tilde N}{N})}\lVert\varphi\rVert^h_{p} \lVert\psi\rVert^h_{\infty}\lVert\ind_B\rVert^h_{q},
	\ea\end{equation}
	where $\varphi_N, \psi_N: \bbZ\to\bbR$ are defined in \eqref{def:phi_psi2}, $B$ is a bounded ball that contains the support of $\psi$ and $B_N:=B\sqrt{N}$, and $\lVert\cdot\rVert_p$ is the $\ell_p$ norm on $\bbZ$.
\end{theorem}

\begin{remark}\label{rmk:hmom}
We will use Theorem \ref{T:hmom} with $h=4$ to prove Lemma \ref{lem:4mmtTheta}. Therein, the coarse-grained disorder can be viewed as 
$\cZ_{N}^{\beta_{\tilde N}}(\varphi,\psi)$ with $N\in\{\epsilon\tilde{N},2\epsilon\tilde{N}\}$, where $\epsilon>0$ is small and will be sent to 0, leading to $N\leq \tilde N$ in Theorem \ref{T:hmom}. 
Moreover, we can see from the proof that for compactly supported $\varphi,\psi$, the assumption of continuity can be weakened to boundedness and measurability.
\end{remark}

\begin{proof}
	We should keep in mind that during the following, $\zeta_n=\zeta^{(\tilde{N})}_n$. Denote
	
	\begin{equation}\ba{l}
	\dis \quad M_{N,\tilde N,h}^{\varphi,\psi}:=\bbE\Big[\Big( \cZ_N^{\beta_{\tilde{N}}}(\varphi,\psi) - \bbE\big[\cZ^{\beta_{\tilde{N}}}_N(\varphi,\psi)\big] \Big)^h \Big] \\[10pt]
	\dis=\frac{1}{N^{\frac{h}{2}}} \bbE\Big[\Big(\sum_{r=1}^{\infty} \sum_{0<n_1<\ldots<n_r< N} q^N_{0,n_1}(\varphi,0)\zeta_{n_1}\Big\{
	\prod_{j=2}^{r}p_{n_{j-1},n_j}(0,0)\zeta_{n_j}\Big\} q^N_{n_j,N}(0,\psi)\Big)^h \Big].
	\ea\end{equation}
	
	Note that when expanding the $h$-fold product above, the summands are transition probabilities timing a sequence of disorders at the return times to $0$ for $h$ random walks. By taking expectation, at each return time appearing in the summand, the disorder contributes a factor of $\bbE[\zeta_n^{\#}]$, where $\#$ is the number of random walks that return to $0$ at time $n$. We have that (see equation (6.7) in \cite{CSZ16})
	\begin{equation}\label{eq:zeta_moments}
	\bbE[\zeta_n]=0,\quad\bbE[(\zeta_n)^2]=\sigma_{\tilde{N}}^2\sim\frac{2\pi}{\log\tilde{N}},\quad\text{and}\quad|\bbE[(\zeta_n)^l]|\leq C_h\sigma_{\tilde{N}}^l\quad\text{for}~3\leq l\leq h.
	\end{equation}
	
	Hence, if a summand is non-zero, then at each return time, there must be at least two random walks returning to $0$. Suppose that $l<m<n$ are three consecutive return times, and one of the $h$ random walks does not visit $0$ at $m$. Its contribution is then only given by a transition probability $\bP(S_l=x, S_m\neq0, S_n=y)\leq p_{n-l}(y-x)$ supposing $S_l=x, S_n=y$. That is the reason why we only consider a pinning model associated to a random walk. For a general renewal process, we lack such a simple expression for transition kernels in the renewal decomposition.

	At a return time $n$, the $h$ random walks can visit $m$ positions, where $1\leq m\leq h-1$. Note that $m$ cannot be $h$, since otherwise we have $\bbE[(\zeta_n)^1]=0$. We then partition the indices $\{1,\cdots,h\}$ into $m$ sets $I:=\{I(k): k=1,\cdots,m\}$, such that $\cup_{k=1}^m I(k)=\{1,\cdots,h\}$ and $I(j)\cap I(k)=\emptyset, \forall j\neq k$. For $i,j\in\{1,\cdots,h\}$, we write $i\overset{I}{\sim}j$ if $i,j\in I(k)$ for some $1\leq k\leq m$. In particular, we use $I(1)$ to record the indices of random walks that visit $0$. Apparently, we only need to consider the case $|I(1)|\geq2$. Hence, we have that at return time $n$,
	\begin{equation}\label{meanI}
	\bbE[\zeta^I]:=\bbE[(\zeta_n)^{|I(1)|}].
	\end{equation}
	
	For $\boldsymbol{x}, \boldsymbol{\tilde{x}}\in(\bbZ)^h$, we introduce the following $h$-fold transition kernels and the counterparts of averages \eqref{qphi0} and \eqref{qpsi0} by:
	\begin{equation}\label{def:Q}
	Q_t(\boldsymbol{x},\boldsymbol{\tilde{x}}):=\prod_{i=1}^h p_t(x_i,\tilde{x}_i),\quad Q_t^N(\varphi,\boldsymbol{x}):=\prod_{i=1}^h q_{0,t}^N(\varphi,x_i),\quad
	Q_t^N(\boldsymbol{x},\psi):=\prod_{i=1}^h q_{0,t}^N(x_i,\psi).
	\end{equation}
	Also, for a given partition $I$, we denote
	\begin{equation}
	\boldsymbol{x}\sim I,\quad\text{if}~x_j=x_k\Longleftrightarrow j\overset{I}{\sim}k,\quad\forall j,k=1,\cdots,h,
	\end{equation}
	and for two consecutive partitions $I,J$, we denote
	\begin{equation}
	Q_t^{I,J}(\boldsymbol{x},\boldsymbol{\tilde{x}}):=\ind_{\{\boldsymbol{x}\sim I, \boldsymbol{\tilde{x}}\sim J\}}Q_t(\boldsymbol{x}, \boldsymbol{\tilde{x}}).
	\end{equation}
	
	Finally, we can write
	\begin{equation}
	\begin{split}
	\quad M_{N,\tilde N,h}^{\varphi,\psi}=&\frac{1}{N^{\frac{h}{2}}}\sum\limits_{r=1}^{\infty}\sum\limits_{1\leq n_1<\cdots<n_r< n_{r+1}:=N\atop I_1,\cdots,I_r; \boldsymbol{x}_1,\cdots,\boldsymbol{x}_r\in(\bbZ)^h}Q_{n_1}^N(\varphi,\boldsymbol{x}_1)\bbE[\zeta^I]\\
	&\times\prod_{i=2}^r Q_{n_i-n_{i-1}}^{I_{i-1},I_i}(\boldsymbol{x}_{i-1},\boldsymbol{x}_i)\bbE[\zeta^I]\times\ind_{\{\boldsymbol{x}_r\sim I_r\}}Q_{n_{r+1}-n_r}^N(\boldsymbol{x}_r,\psi).
	\end{split}
	\end{equation}
	Note that since we are bounding $|M_{N,\tilde N,h}^{\varphi,\psi}|$, we will take absolute values of $\varphi,\psi$ and $\bbE[\zeta^I]$. Hence, we may assume that all these quantities are positive to lighten the notations.
	
	Recall \eqref{qpsi0} and suppose that for $\text{supp}(\psi)\subset B$ with $B_N=\sqrt{N}B$,
	\begin{equation}\label{eq:enlargeN}
	|q_{0,N-n_r}^N(y,\psi)|\leq\|\psi\|_\infty\sum\limits_{v\in \bbZ}p_{N-n_r}(v-y)\ind_{B}\Big(\frac{v}{\sqrt{N}}\Big)=\|\psi\|_\infty\bP^y(S_{N-n_r}\in \sqrt{N}B).
	\end{equation}
	The probability at the right-hand side above is of $O(1)$ if $N-n_r=O(N)$. Hence, there exists a constant $C$, such that uniformly for all $y$ and $1\leq n_r\leq N\leq n_{r+1}\leq 2N$,
	\begin{equation*}
	\bP^y(S_{N-n_r}\in\sqrt{N}B)\leq C\bP^y(S_{n_{r+1}-n_r}\in\sqrt{N}B)=C q_{0,n_{r+1}-n_r}^N(y,\ind_{B_N}).
	\end{equation*}
	By summing over $N\leq n_{r+1}\leq 2N$, we have that
	\begin{equation}\label{eq:Mbound1}
	\begin{split}
	\Big|M_{N,\tilde N,h}^{\varphi,\psi}\Big|\leq&\frac{C^h\|\psi\|_\infty^h}{N^{\frac{h}{2}+1}}\sum\limits_{r=1}^{\infty}\sum\limits_{1\leq n_1<\cdots<n_r\leq n_{r+1}:=2N\atop I_1,\cdots,I_r; \boldsymbol{x}_1,\cdots,\boldsymbol{x}_r\in(\bbZ)^h}Q_{n_1}^N(\varphi,\boldsymbol{x}_1)\bbE[\zeta^{I_1}]\\
	&\times\prod_{i=2}^r Q_{n_i-n_{i-1}}^{I_{i-1},I_i}(\boldsymbol{x}_{i-1},\boldsymbol{x}_i)\bbE[\zeta^{I_i}]\times\ind_{\{\boldsymbol{x}_r\sim I_r\}}Q_{n_{r+1}-n_r}^N(\boldsymbol{x}_r,\ind_{B_N}).
	\end{split}
	\end{equation}
	
	We first single out a special case that some consecutive $I_i, I_{i+1}, \dots, I_j$ are the same with $|I_i|=\cdots=|I_j|=h-1$. Note that in this case there exist $1\leq m\neq l\leq h$, such that, $I_k(1)=\{m, l\}$ for all $k=i, i+1, \cdots, j$, namely, random walks with indices $m$ and $l$ keep returning to $0$ consecutively at all return times, while other walks never return to $0$.
	
	We can simplify the notation for this case. Denote the common partition of $I_i, \cdots, I_j$ by $I$. For $1\leq s\leq t\leq 2N$ and $\boldsymbol{x},\boldsymbol{\tilde{x}}\in(\bbZ)^h$, we introduce
	\begin{equation}\label{eq:bigU}
	U^I_{t-s, \tilde{N}}(\boldsymbol{x},\boldsymbol{\tilde{x}}):=\ind_{\{\boldsymbol{x},\boldsymbol{\tilde{x}}\sim I\}}\sum\limits_{k=1}^\infty\bbE[\zeta^2]^k\sum\limits_{n_0:=s<n_1<\cdots<n_k:=t\atop \boldsymbol{y}_i\in(\bbZ)^h, \boldsymbol{y}_0:=\boldsymbol{x}, \boldsymbol{y}_k:=\boldsymbol{\tilde{x}}}\prod\limits_{i=1}^k Q_{n_i-n_{i-1}}^{I,I}(\boldsymbol{y}_{i-1},\boldsymbol{y}_i).
	\end{equation}
	Note that here $\zeta$ depends on $\tilde{N}$ and $U_{0,\tilde{N}}^{I}(\boldsymbol{x},\boldsymbol{\tilde{x}})$ is understood as $\ind_{\{\boldsymbol{x},\boldsymbol{\tilde{x}}\sim I\}}$. Then in \eqref{eq:Mbound1}, for any given first and last return times $s,t$, we can contract all consecutive $I_i=I_{i+1}=\cdots=I_j:=I$ for any $j-i\in\bbN$ to a single kernel $U_{t-s}^I(\cdot,\cdot)$. Recall $\overline{U}_N(n)$ from \eqref{def:U_bar}. We write
	\begin{equation}\label{def:U}
	U_N(n)=\overline{U}_N(n)/\sigma_N^2.
	\end{equation}
	If $I(1)=\{m,l\}$, then we have that
	\begin{equation}\label{eq:bigU2}
	U^I_{t-s, \tilde{N}}(\boldsymbol{x},\boldsymbol{\tilde{x}})\leq U_{\tilde{N}}(t-s)\prod_{i\in\{1,\cdots,h\}\backslash\{m,l\}}p_{s,t}(x_i, \tilde{x}_i).
	\end{equation}
	
	Before we continue, we introduce some extra notations for summing over $1\leq n_1<\cdots<n_r\leq 2N$ in \eqref{eq:Mbound1}. For partitions $I, J$, let
	\begin{align}
	\label{Qlambda}Q_{\lambda, N}^{I,J}(\boldsymbol{x},\boldsymbol{\tilde{x}})&:=\sum\limits_{n=1}^{2N}e^{-\lambda n}Q_n^{I,J}(\boldsymbol{x},\boldsymbol{\tilde{x}}),\quad\boldsymbol{x},\boldsymbol{\tilde{x}}\in(\bbZ)^h,\\
	\label{Ulambda}U_{\lambda,N,\tilde{N}}^J(\boldsymbol{x},\boldsymbol{\tilde{x}})&:=\sum\limits_{n=1}^{2N}e^{-\lambda n}U_{n,\tilde{N}}^J(\boldsymbol{x},\boldsymbol{\tilde{x}}),\quad\boldsymbol{x},\boldsymbol{\tilde{x}}\in(\bbZ)^h.
	\end{align}
	Then we define the operators from $(\bbZ)^h$ to $(\bbZ)^h$ by
	\begin{equation}\label{def:op}
	P_{\lambda}^{I,J}=P_{\lambda,N,\tilde{N}}^{I,J}:=\begin{cases}
	Q_{\lambda,N}^{I,J},\quad&\text{if}~|J|<h-1,\\
	Q_{\lambda,N}^{I,J}U_{\lambda,N,\tilde{N}}^{J},\quad&\text{if}~|J|=h-1.
	\end{cases}
	\end{equation}
	
	Now in each summand on the right-hand side of \eqref{eq:Mbound1}, we insert a term $e^{-\lambda\sum_{i=1}^{r+1}(n_i-n_{i-1})}$\\
$\times e^{2\lambda N}\geq1$, where $\lambda>0$ will be determined later. For the initial and terminal times, there is no constraint on the positions of the random walks, and we denote the ``partition'' by $*$. Finally, \eqref{eq:Mbound1} can be further bounded above by
	\begin{equation}\label{eq:Mbound2}
	\Big|M_{N,\tilde{N},h}^{\varphi,\psi}\Big|\leq\frac{C_h\|\psi\|_\infty^h e^{2\lambda N}}{N^{\frac{h}{2}+1}}\sum\limits_{r=1}^\infty\sum_{I_1,\cdots,I_r}\big\langle\varphi_N^{\otimes h},P^{*,I_1}_\lambda P^{I_1,I_2}_\lambda\cdots P_\lambda^{I_{r-1},I_r}Q_{\lambda}^{I_r,*}\ind_{B_N}^{\otimes h}\big\rangle\prod_{i=1}^r\bbE[\zeta^{I_i}].
	\end{equation}
	Here the summation is over all partitions $I_1,\cdots,I_r$ of $\{1,\cdots,h\}$, such that $|I_i|\leq h-1$ and no consecutive $I_i=I_{i+1}$ with $|I_i|=h-1$. The operation $\big\langle\cdot,\cdot\big\rangle$ is the classic notation for integral on $(\bbZ)^h\otimes(\bbZ)^h$. Given a function $f:\bbZ\rightarrow\bbR$, for $\boldsymbol{x}\in(\bbZ)^h$, $f^{\otimes h}(\boldsymbol{x}):=\prod_{i=1}^h f(x_i)$. 
	
	For $p,q\in(1,+\infty)$ such that $\frac1p+\frac1q=1$, by H\"{o}lder's inequality, \eqref{eq:Mbound2} is bounded above by
	\begin{equation}\label{eq:Mbound3}
	\begin{split}
	\Big|M_{N,\tilde{N},h}^{\varphi,\psi}\Big|\leq\frac{C_h\|\psi\|_\infty^h\|\varphi\|_p^h e^{2\lambda N}}{N^{\frac{h}{2}+1}}&\sum\limits_{r=1}^\infty\sum_{I_1,\cdots,I_r}\Big\|P_\lambda^{*,I_1}\Big\|_{\ell^q\to\ell^q}\Big\|P_\lambda^{I_1, I_2}\Big\|_{\ell^q\to\ell^q}\cdots\\
	\cdots&\Big\|P_\lambda^{I_{r-1},I_r}\Big\|_{\ell^q\to\ell^q}\Big\|Q_\lambda^{I_r,*}\Big\|_{\ell^q\to\ell^q}\Big\|\ind_{B_N}^{\otimes h}\Big\|_q\prod\limits_{i=1}^r\bbE[\zeta^{I_i}],
	\end{split}
	\end{equation}
	where the norm $\|\cdot\|_{\ell^q\to\ell^q}$ for an operator $A:(\bbZ)^h\rightarrow(\bbZ)^h$ is defined by
	\begin{equation*}
	\|A\|_{\ell^q\to\ell^q}:=\sup\limits_{f\not\equiv0}\frac{\|Af\|_q}{\|f\|_q}.
	\end{equation*}
	
	By choosing $\lambda:=\hat{\lambda}/N$ with $\hat{\lambda}$ large but fixed (to be determined later), we have the following proposition to control all the norms above.
	\begin{proposition}\label{prop:Opbounds}
		Let all the conditions in Theorem \ref{T:hmom} be satisfied. There exists some constant $c=c_{p,q,h,\hat{\lambda}}\in(0,+\infty)$ for $h\geq3$, such that uniformly for partitions $I, J$ for $\{1,\cdots,h\}$ with $1\leq|I|,|J|\leq h-1$ and $I\neq J$ for $|I|=|J|=h-1$, we have that for large $N\leq\tilde{N}$ and $\lambda=\frac{\hat{\lambda}}{{N}}$,
		\begin{align}
		\label{bound:Q_IJ}&\Big\|Q_{\lambda, N}^{I,J}\Big\|_{\ell^q\to\ell^q}\leq c,\\
		\label{bound:Q*}\Big\|Q_{\lambda, N}^{*,I}\Big\|_{\ell^q\to\ell^q}&\leq cN^{\frac1q},\quad\Big\|Q_{\lambda, N}^{I,*}\Big\|_{\ell^q\to\ell^q}\leq cN^{\frac1p},\\
		\label{bound:U_I}\text{and in particular, for}~|I|=h-1&,\quad\Big\|U_{\lambda,N,\tilde{N}}^J\Big\|_{\ell^q\to\ell^q}\leq\frac{c}{\log\big(\hat{\lambda}\frac{\tilde{N}}{N}\big)\sigma_{\tilde{N}}^2}.
		\end{align}
	\end{proposition}
	
	We postpone the proof of this proposition and apply it to completing the proof of Theorem \ref{T:hmom} first. For $h\geq3$, note that $\bbE[\zeta^I]=\sigma_{\tilde{N}}^2$ for $|I|=h-1$ and $|\bbE[\zeta^I]|\leq c'\sigma_{\tilde{N}}^3=O((\log\tilde{N})^{-{3/2}})$ for $|I|<h-1$. Then the above bounds for operators imply that for $2\leq i\leq r$,
	\begin{equation}\label{eq:Opbound1}
	\dis\bbE[\zeta^{I_i}] \Big\| P^{I_{i-1},I_i}_{\lambda}\Big\|_{\ell^q\to \ell^q}\leq\ind_{\{|I_i|=h-1\}}\frac{c^2}{\log\big(\hat{\lambda}\frac{\tilde{N}}{N}\big)}+\ind_{\{|I_i|<h-1\}}cc'\sigma_{\tilde{N}}^3\leq \frac{c''}{\log \big(\hat{\lambda} \frac{\tilde N}{N}\big)},
	\end{equation}
	and similarly,
	\begin{equation}\label{eq:Opbound2}
	\dis\bbE[\zeta^{I_1}] \Big\| P^{\ast,I_1}_{\lambda}\Big\|_{\ell^q\to \ell^q}\leq N^{\frac1q}\Big(\ind_{\{|I_1|=h-1\}}\frac{c}{\log \big(\hat{\lambda} \frac{\tilde N}{N}\big)}+\ind_{\{|I_1|<h-1\}}c'\sigma_{\tilde{N}}^3\Big)\leq \frac{c''}{\log (\hat{\lambda} \frac{\tilde N}{N})} N^{\frac{1}{q}}.
	\end{equation}
	
	For fixed $h$, the number of partitions $I_i$ is bounded by a constant $c_h$.
	Choose $\hat\lambda$ large enough such that $c_h\frac{c''}{\log \hat\lambda}\leq\frac{1}{2}$. Plug the bounds \eqref{eq:Opbound1} and \eqref{eq:Opbound2} into \eqref{eq:Mbound3}, and we have that for $N$ large enough,
	\begin{equation}
	\begin{split}
	\big|M_{N,\tilde N,h}^{\varphi,\psi}\big|\dis\leq&\frac{Ce^{2\hat\lambda} \lVert\psi\rVert^h_{\infty}}{N^{\frac{h}{2}+1}}\big\|\varphi_N\big\|^h_{\ell^p}\|  \ind_{B_N}\big\rVert^h_{\ell^q} N^{\frac{1}{p}+\frac{1}{q}} \sum_{r=1}^{\infty}
	\Big(\frac{c_h c''}{\log (\hat{\lambda} \frac{\tilde N}{N})}\Big)^r\\
	\leq&\frac{C_1}{N^{\frac{h}{2}}\log(1+\frac{\tilde N}{N})} \big\|\psi\big\rVert^h_{\infty} \big\| \varphi_N\big\rVert^h_{\ell^p} \big\lVert \ind_{B_N}\big\rVert^h_{\ell^q}.
	\end{split}
	\end{equation}
	
\end{proof}

It only remains to prove Proposition \ref{prop:Opbounds}. We start by bounding the following quantity
\begin{equation}\label{Qtran}
Q_{\lambda,N}(\boldsymbol{x}, \boldsymbol{y}):=\sum_{n=1}^{2N}e^{-\lambda n}\prod_{i=1}^{h}p_n(y_i-x_i),
\end{equation}
which is actually $Q_{\lambda,N}^{*,*}(\boldsymbol{x},\boldsymbol{y})$, a partition-free version of \eqref{Qlambda}. Since $Q_N(\boldsymbol{x},\boldsymbol{y}):=Q_{0,N}(\boldsymbol{x},\boldsymbol{y})=\sup_{\lambda\geq0}Q_{\lambda,N}(\boldsymbol{x},\boldsymbol{y})$, it suffices to bound $Q_N(\boldsymbol{x},\boldsymbol{y})=\sum_{n=1}^{2N}\prod_{i=1}^h p_n(y_i-x_i)$.

By Theorem \ref{thm:local_2}, there exist constants $\gamma,C>0$, such that uniformly in $x$ and $n$,
\begin{equation}\label{eq:apply_local_limit}
p_{n}(x)\leq g_n(x)+\frac{C}{(1+|x|^{1+\gamma})\sqrt{n}},
\end{equation}
where $g_t(x)$ is the heat kernel (see \eqref{heatkernel}). Then we have that
\begin{equation}\label{Q_tran2}
\begin{split}
Q_N(\boldsymbol{x},\boldsymbol{y})&\leq\sum\limits_{n=1}^{2N}\prod\limits_{i=1}^h\Big(g_n(y_i-x_i)+\frac{C}{(1+|y_i-x_i|^{1+\gamma})\sqrt{n}}\Big)\\
&=\sum\limits_{n=1}^{2N}\prod\limits_{i=1}^h g_n(y_i-x_i)+\sum\limits_{n=1}^{2N}\sum\limits_{K\subsetneqq\{1,\cdots,h\}}\prod\limits_{i\in K}g_n(y_i-x_i)\prod\limits_{j\notin K}\frac{C}{(1+|y_j-x_j|^{1+\gamma})\sqrt{n}}.
\end{split}
\end{equation}

To bound the above quantity, we need \cite[Lemma 6.7]{CSZ21}, which shows that there exists some constant $C\in(0,+\infty)$, such that uniformly in $x, y, N$ and $K\geq 3$,
\begin{equation}\label{Q_tran3}
\sum\limits_{n=1}^{2N}\prod\limits_{i\in K}g_n(y_i-x_i)\leq\begin{cases}
\displaystyle
\frac{C}{(1+|\boldsymbol{x}-\boldsymbol{y}|^2)^{\frac{K}{2}-1}},\quad&\text{for all}~\boldsymbol{x},\boldsymbol{y}\in(\bbZ)^{K},\\[16pt]
\displaystyle
\frac{C}{N^{\frac{K}{2}-1}}e^{-\frac{|\boldsymbol{y}-\boldsymbol{x}|^2}{CN}},\quad&\text{for}~|\boldsymbol{x}-\boldsymbol{y}|\geq\sqrt{N}.
\end{cases}
\end{equation}

Now we can first treat \eqref{bound:Q_IJ}, that is, $\|Q_{\lambda, N}^{I,J}\|_{\ell^q\to\ell^q}$. 

\begin{proof}[\textbf{Proof of \eqref{bound:Q_IJ}}]
	Note that this is equivalent to proving that there exists some constant $c\in(0,+\infty)$, such that uniformly for all $\varphi\in\ell^p((\bbZ)^h_I)$ and $\psi\in\ell^q((\bbZ)^h_J)$,
	\begin{equation}\label{Qbound}
	\sum\limits_{\boldsymbol{x}\in(\bbZ)^h_I,\boldsymbol{y}\in(\bbZ)^h_J}\varphi(\boldsymbol{x})Q_\lambda^{I,J}(\boldsymbol{x},\boldsymbol{y})\psi(\boldsymbol{y})\leq c\|\varphi\|_{\ell^p}\|\psi\|_{\ell^q}.
	\end{equation}
	
	Note that $n^{-1/2}=\sqrt{2\pi}g_n(0)$. By combining $g_n(y_i-x_i)$ for $i\in K$ and $n^{-1/2}$ for $j\notin K$, and by \eqref{Q_tran2} and \eqref{Q_tran3} together, we have that
	\begin{equation}\label{Qbound2}
	\begin{split}
	&\sum\limits_{\boldsymbol{x}\in(\bbZ)^h_I,\boldsymbol{y}\in(\bbZ)^h_J}\varphi(\boldsymbol{x})Q_{\lambda}^{I,J}(\boldsymbol{x},\boldsymbol{y})\psi(\boldsymbol{y})\leq\sum\limits_{\boldsymbol{x}\in(\bbZ)^h_I,\boldsymbol{y}\in(\bbZ)^h_J}\varphi(\boldsymbol{x})\frac{C}{(1+|\boldsymbol{x}-\boldsymbol{y}|^2)^{\frac{h}{2}-1}}\psi(\boldsymbol{y})\\
	&+\sum\limits_{K\subsetneqq\{1,\cdots,h\}}\sum\limits_{\boldsymbol{x}\in(\bbZ)^h_I,\boldsymbol{y}\in(\bbZ)^h_J}\varphi(\boldsymbol{x})\bigg(\frac{C}{(1+\sum_{i\in K}|y_i-x_i|^2)^{\frac{h}{2}-1}}\prod\limits_{j\notin K}\frac{C}{(1+|y_j-x_j|^{1+\gamma})}\bigg)\psi(\boldsymbol{y}).
	\end{split}
	\end{equation}
	Hence, we only need to show that both terms on the right-hand side above is bounded by $C\|\varphi\|_{\ell^p}\|\psi\|_{\ell^q}$.
	
	The result then follows by the following lemma.
	\begin{lemma}\label{L:tran}
		Fix $p,q>1$ with $\frac{1}{p}+\frac{1}{q}=1$ and an integer $h\geq 3$.
		Consider partitions $I,J$ with $1\leq|I|,|J|\leq h-1$, and $I\neq J$ if $|I|=|J|=h-1$.
		Let $\varphi\in\ell^p((\bbZ)^h_I)$ and $\psi\in\ell^q((\bbZ)^h_J)$.
		Then there exists $C=C_{p,q,h}<\infty$, independent of $f$ and $g$, such that
		\begin{equation}\label{HLS}
		\sum\limits_{\boldsymbol{x}\in(\bbZ)^h_I,\boldsymbol{y}\in(\bbZ)^h_J}\frac{\varphi(\boldsymbol{x})\psi(\boldsymbol{y})}{(1+|\boldsymbol{x}-\boldsymbol{y}|^2)^{\frac{h}{2}-1}}\leq C\|\varphi\|_{\ell^p}\|\psi\|_{\ell^q},
		\end{equation}
		and
		\begin{equation}\label{HLS+}
		\begin{split}
		\sum\limits_{K\subsetneqq\{1,\cdots,h\}}\sum\limits_{\boldsymbol{x}\in(\bbZ)^h_I,\boldsymbol{y}\in(\bbZ)^h_J}&\varphi(\boldsymbol{x})\bigg(\frac{C}{(1+\sum_{i\in K}|y_i-x_i|^2)^{\frac{h}{2}-1}}\prod\limits_{j\notin K}\frac{C}{(1+|y_j-x_j|^{1+\gamma})}\bigg)\psi(\boldsymbol{y})\\
		\leq&C\|\varphi\|_{\ell^p}\|\psi\|_{\ell^q}
		\end{split}
		\end{equation}
	\end{lemma}
	
	Next we devote to prove Lemma \ref{L:tran}.
\end{proof}

\begin{proof}[\textbf{Proof of Lemma \ref{L:tran}}]
	In the proof, we will repeatedly use the bound
	\begin{equation}\label{xsum}
	\sum_{x\in\Z}\frac{1}{(s+x^2)^r}\leq \frac{C}{s^{r-1/2}},\quad\text{with a uniform $C>0$ for}~s>1, r>\frac12.
	\end{equation} 
	
	We first treat \eqref{HLS} with $|I|=|J|=h-1$. Assume that the index sets for renewal random walks are $I(1)=\{k,l\}$ and $J(1)=\{m,n\}$, with $\{k,l\}\neq\{m,n\}$. Choose $0<a<\frac{1}{2(p\vee q)}$. By H\"older's inequality, the left-hand side of \eqref{HLS} is bounded from above by
	\begin{equation}\begin{array}{l}\label{HLS1}
	\displaystyle\left(\sum\limits_{\boldsymbol{x}\in(\bbZ)^h_I,\boldsymbol{y}\in(\bbZ)^h_J}\frac{\varphi(\boldsymbol{x})^p}{(1+\sum_{i=1}^{h}|x_i-y_i|^2)^{\frac{h}{2}-1}}\cdot \frac{(1+|x_m-x_n|^{2a})^p}{(1+|y_k-y_l|^{2a})^p} \right)^{1/p} \\ [5pt]
	\displaystyle\qquad\times\left(\sum\limits_{\boldsymbol{x}\in(\bbZ)^h_I,\boldsymbol{y}\in(\bbZ)^h_J}\frac{\psi(\boldsymbol{y})^q}{(1+\sum_{i=1}^{h}|x_i-y_i|^2)^{\frac{h}{2}-1}}\cdot \frac{(1+|y_k-y_l|^{2a})^q}{(1+|x_m-x_n|^{2a})^q} \right)^{1/q}.
	\end{array}\end{equation}
	
	We split the cases into (i) $\{k,l\}\cap\{m,n\}=\emptyset$ and (ii) $k=m,l\neq n$.
	
	\vspace{0.25cm}
	\noindent\textbf{Proof for Case (i).} For the first factor above, note that $x_k=x_l=y_m=y_n=0$ in $(\bbZ)^h_J$. Use (\ref{xsum}) to sum over the remaining $h-4$ variables $y_j$ with $j\neq k,l,m,n$. It is bounded above by a constant timing
	\begin{equation}\begin{array}{l}\label{HLS2}
	\displaystyle\Bigg(\sum_{\boldsymbol{x}\in(\bbZ)^h_I} \varphi(\boldsymbol{x})^p(1+|x_m-x_n|^{2a})^p\!\!\sum\limits_{y_k,y_l\in\bbZ}\frac{1}{(1+x_m^2+x_n^2+y_k^2+y_l^2) (1+|y_k-y_l|^{2a})^p}\Bigg)^{1/p}.
	\end{array}\end{equation}
	Let $\tilde y_1:=y_k-y_l$ and $\tilde y_2:=y_k+y_l$. By $\tilde y_1^2+\tilde y_2^2=2(y_k^2+y_l^2)$, the inner sum in (\ref{HLS2}) is bounded by
	\begin{equation}\begin{split}\label{HLS3}
	\displaystyle&\sum\limits_{\tilde y_1,\tilde y_2}\frac{C}{(1+x_m^2+x_n^2+\tilde y_1^2+\tilde y_2^2) (1+\tilde y_1^{2a})^p}\\
	 \displaystyle\leq&\sum\limits_{\tilde y_1}\frac{C}{(1+x_m^2+x_n^2+\tilde y_1^2)^{1/2} (1+\tilde y_1^{2a})^p} \leq \frac{C}{(1+x_m^2+x_n^2)^{ap}},
	\end{split}\end{equation}
	where the last inequality is achieved by summing over $|\tilde y_1|\leq \sqrt{x_m^2+x_n^2}$ and $|\tilde y_1|>\sqrt{x_m^2+x_n^2}$ separately. Then $(1+|x_m-x_n|^{2a})^p$ and \eqref{HLS3} cancel out, and thus the first factor of (\ref{HLS1}) is upper bounded by $C\lVert \varphi\rVert_{\ell^p}$. The second factor of \eqref{HLS1} can similarly be bounded by $C\|\psi\|_{\ell^q}$.
	
	\vspace{0.25cm}
	\noindent\textbf{Proof for Case (ii).} For the first factor in \eqref{HLS1}, note that $x_m=x_l=y_m=y_n=0$. Sum over the remaining $h-3$ variables $y_j$ with $j\neq l,m,n$.
	Similarly, we get an upper bound
	\begin{equation}\begin{array}{l}
	\displaystyle \quad C\Bigg(\sum_{\boldsymbol{x}\in(\bbZ)^h_I} \varphi(\boldsymbol{x})^p(1+|x_n|^{2a})^p 
	\sum_{y_l}\frac{1}{(1+x_n^2+y_l^2)^{1/2} (1+y_l^{2a})^p}\Bigg)^{1/p}\leq C\|\varphi\|_{\ell^p},\\[5pt]
	\end{array}\end{equation}
	where we sum over $|y_l|\leq|x_n|$ and $|y_l|\geq|x_n|$ separately. The second factor of (\ref{HLS1}) can be treated in the same way.
	
	When $\min\{|I|,|J|\}<h-1$, the proof is easier. If $|I|, |J|<h-1$, then we have two patterns: either there are two random walks visiting $0$ and two other random walks visiting the same other position, or there are three random walks visiting $0$. In either case, we only need to sum over $x_i$ or $y_j$ for at most $h-3$ times
	. Hence, we do not need to introduce the factor $(1+|x_m-x_n|^{2a})/(1+|y_k-y_l|^{2a})$. If $|I|<h-1$ and $|J|=h-1$, then we can find $k, l$ belonging to the same partition element of $I$ with $k, j\notin J(1)$. Then we can replace the factor $(1+|x_m-x_n|^{2a})/(1+|y_k-y_l|^{2a})$ in \eqref{HLS1} by $1/(1+|y_k-y_l|^{2a})$ (and also its reciprocal, resp.). The rest of the proof is then the same as that for $|I|=|J|=h-1$.
	
	Then we turn to \eqref{HLS+}. The treatment is similar. If $|I|=|J|=h-1$, then we have \textbf{Case (1)} and \textbf{Case (2)} as above. We will insert different factors according to $k,l,m,n$ and $K$:
	
	Recall that in \textbf{Case (i)}, $I(1)=\{k,l\}$ and $J(1)=\{m,n\}$ with $\{k,l\}\cap\{m,n\}=\emptyset$.
	We can insert a factor according to the following strategy:
	\begin{itemize}
		\item If $k,l\in K$, then we insert $1+|y_k|^{2a}$ and its reciprocal.
		\item If $k$ or $l$ does not belong to $K$, then we do not need to insert anything related to $y_j$.
		\item Similarly, if $m,l\in K$, then we insert $1+|x_m|^{2a}$ and its reciprocal.
		\item If $m$ or $n$ does not belong to $K$, then we do not need to insert anything related to $x_i$.
	\end{itemize}
	The reason behind this strategy is that when $j\notin K$, $(1+|y_j-x_j|^{1+\gamma})^{-1}$ is summable and we save the degree of power by $\frac12$.
	
	The \textbf{Case (2)} and the case $\min\{|I|,|J|\}<h-1$ can be treated similarly and we omit the details.
\end{proof}

Then we treat \eqref{bound:Q*}, that is, $\|Q_{\lambda, N}^{*,I}\|_{\ell^q\to\ell^q}$ and $\|Q_{\lambda, N}^{I, *}\|_{\ell^p\to\ell^p}$. The proof is similar to that of \eqref{bound:Q_IJ} and thus we only sketch it.

\begin{proof}[\textbf{Proof sketch for (\ref{bound:Q*})}]
	We only treat $Q_{\lambda, N}^{I,*}$, since the treatment for $Q_{\lambda,N}^{*,I}$ is the same. It suffices to prove that there exits come constant $c\in(0,+\infty)$, such that uniformly for all $N$ and $\varphi\in \ell^p((\bbZ)^h_I)$ and $\psi\in \ell^p((\bbZ)^h_J)$.
	\begin{equation}\label{tran}
	\sum\limits_{\boldsymbol{x}\in(\bbZ)^h_I,\boldsymbol{y}\in(\bbZ)^h} \varphi(\boldsymbol{x})Q^{I,*}_{\lambda,N}(\boldsymbol{x},\boldsymbol{y})\psi(\boldsymbol{y})\leq c N^{\frac{1}{p}} \|\varphi\|_{\ell^p} \|\psi\big\|_{\ell^q}.
	\end{equation}

	Without loss of generality, we may assume that $1,2\in I(1)$ so that $x_1=x_2=0$. By \eqref{Q_tran2}, we first treat the kernel $\sum_{n=1}^{2N}\prod_{i=1}^h g_n(y_i-x_i)$. We split the cases $|\boldsymbol{x}-\boldsymbol{y}|\leq\sqrt{N}$ and $|\boldsymbol{x}-\boldsymbol{y}|\geq\sqrt{N}$. The latter case is simple by the exponential decay in \eqref{Q_tran3}. For $|\boldsymbol{x}-\boldsymbol{y}|\leq\sqrt{N}$, we insert a factor $\log(1+\frac{N}{y_1^2+y_2^2})^{1/q}$ and its reciprocal, and then apply the H\"{o}lder's inequality as in \eqref{HLS1}. The bound then follows by noting that $\int(c+x^2)^{-\frac{1}{2}}\dd x=\ln(x+\sqrt{c+x^2})$.

	Then for any given $K\subsetneqq\{1,\cdots,h\}$, we are going to bound
	\begin{equation}\label{QIast3}
	\sum\limits_{\boldsymbol{x}\in(\bbZ)^h_I,\boldsymbol{y}\in(\bbZ)^h}\varphi(\boldsymbol{x})\bigg(\sum\limits_{n=1}^{2N}\prod\limits_{i\in K}g_n(y_i-x_i)\prod\limits_{j\notin K}\frac{C}{(1+|y_j-x_j|^{1+\gamma})\sqrt{n}}\bigg)\psi(\boldsymbol{y}),
	\end{equation}
	which can also be done by splitting the cases $|\boldsymbol{x}-\boldsymbol{y}|\leq\sqrt{N}$ and $|\boldsymbol{x}-\boldsymbol{y}|\geq\sqrt{N}$, after we have applied H\"{o}lder's inequality and summed over $x_j,y_j$ for $j\notin K$. For $|\boldsymbol{x}-\boldsymbol{y}|\leq\sqrt{N}$, if $1,2\in K$, then we insert a factor $\log(1+\frac{N}{y_1^2})$ and its reciprocal. If either $1$ or $2$ does not belong to $K$, then we do not need to insert any factor.
	%
	%
	%
	%
\end{proof}

Finally, we treat \eqref{bound:U_I}, that is, $\|U_{\lambda,N,\tilde{N}}^I\|_{\ell^q\to\ell^q}$.

\begin{proof}[\textbf{Proof of (\ref{bound:U_I})}]
	It suffices to prove that there exists some constant $c\in(0,+\infty)$, such that uniformly for all $N, \tilde{N}$ and $\varphi\in \ell^p((\bbZ)^h_I)$ and $\psi\in \ell^p((\bbZ)^h_J)$,
	\begin{equation}\label{Utran}
	\sum\limits_{\boldsymbol{x}\in(\bbZ)^h_I,\boldsymbol{y}\in(\bbZ)_I^h} \varphi(\boldsymbol{x})U_{\lambda,N,\tilde{N}}^I(\boldsymbol{x},\boldsymbol{y})\psi(\boldsymbol{y})\leq \frac{c}{\log\big(\hat{\lambda}\frac{\tilde{N}}{N}\big)\sigma_{\tilde{N}}^2}\|\varphi\|_{\ell^p} \|\psi\big\|_{\ell^q}.
	\end{equation}

	Recall \eqref{Ulambda} and \eqref{eq:bigU2} and we assume that $I(1)=\{1,2\}$, that is, only the random walks with indices 1 and 2 visit $0$. We have that 
	\begin{equation*}
	U_{\frac{\hat{\lambda}}{N},N,\tilde{N}}^I(\boldsymbol{x},\boldsymbol{y})\leq\ind_{\{\boldsymbol{x}=\boldsymbol{y}\}}+\sum_{n=1}^{2N}e^{-\hat{\lambda}\frac{n}{N}}U_{\tilde{N}}(n)\prod_{i=3}^h p_n(y_i-x_i)\leq 1+\sum\limits_{n=1}^{2N}e^{-\hat{\lambda}\frac{n}{N}}U_{\tilde{N}}(n),
	\end{equation*}
	where $U_{\tilde{N}}(n,x)$ is defined in \eqref{def:U} with $\sigma_N^2$ replaced by $\sigma_{\tilde{N}}^2\sim\sqrt{2\pi}(\log\tilde{N})^{-1}$.
	
	 Set $T:=\frac{\tilde{N}}{N}\geq1$. By \eqref{overlineU}, $\sum_{\boldsymbol{y}\in(\bbZ)_I^h}U_{\hat{\lambda}/N,N,\tilde{N}}^I(\boldsymbol{x},\boldsymbol{y})$ is bounded above by
	\begin{equation*}
	\begin{split}
	1+C\log\tilde{N}\sum\limits_{n=1}^{2\tilde{N}/T}e^{-\hat{\lambda}\frac{nT}{\tilde{N}}}G_{\vartheta}\Big(\frac{n}{\tilde{N}}\Big)\frac{1}{\tilde{N}}
	\leq 1+C\log\tilde{N}\int_{0}^{\frac2T}e^{-\hat{\lambda}Tt}G_{\vartheta}(t)\dd t.
	\end{split}
	\end{equation*}
	By \eqref{asym:Gtheta}, $G_{\vartheta}(t)$ is uniformly bounded on $[\frac12,2]$ and $G_{\vartheta}(t)\sim [t(\log\frac1t)^2]^{-1}$ as $t\to0$. We have that $\int_{0}^{2/T}e^{-\hat{\lambda}Tt}G_{\vartheta}(t)\dd t$ is bounded above by
	\begin{equation*}
	\begin{split}
	C\int_{\frac12}^{2}e^{-\hat{\lambda}Tt}\dd t+\int_0^{(\hat{\lambda}T)^{-\frac12}}\frac{\dd t}{t(\log\frac1t)^2}+e^{-\sqrt{\hat{\lambda}T}}\int_{(\hat{\lambda}T)^{-\frac12}}^{\frac12\wedge\frac2T}\frac{\dd t}{t(\log\frac1t)^2}\leq\frac{C}{\log\hat{\lambda}T},
	\end{split}
	\end{equation*}
	which yields \eqref{Utran} by H\"{o}lder's inequality and $\sigma_{\tilde{N}}^2\leq C(\log\tilde{N})^{-1}$.
\end{proof}

\section{Approximating $\cZ_N(\varphi,\psi)$ by the coarse-grained model $\Lcg(\varphi,\psi)$}\label{S6}
In this section, we approximate $\cZ_N$ by $\Znotriple(\varphi,\psi)$, and then approximate $\Znotriple(\varphi,\psi)$ by $\Zcg(\varphi,\psi)$, and finally approximate $\Zcg(\varphi,\psi)$ by $\Lcg(\varphi,\psi|\Theta)$.

\subsection{Approximating $\cZ_N$ by $\cZ_{N,\epsilon}^{\text{\rm (no triple)}}$}
Recall $\Znotriple$ from \eqref{Znotriple}. We control the $L_2$ distance between $\cZ_N$ and $\Znotriple$, which is
\begin{equation}\label{notriple1}
\begin{split}
&\Big\|\Big(\Znotriple-\cZ_N\Big)(\varphi, \psi)\Big\|_2^2=\Big(\sum\limits_{\cC_1}+\sum\limits_{\cC_2}\Big)\frac{1}{N}\sum\limits_{d_1\leq f_1\in\cT_{N,\epsilon}(\rmi_1),\cdots,d_k\leq f_k\in\cT_{N,\epsilon}(\rmi_k)}\\
&\left(q_{0,d_1}^N(\varphi,0)\right)^2\overline{U}_N(f_1-d_1)\Big\{\prod\limits_{j=2}^k u(d_j-f_{j-1})\overline{U}_N(f_j-d_j)\Big\}\left(q_{f_k,N}^N(0,\psi)\right)^2,
\end{split}
\end{equation}
where
\begin{equation*}
\cC_1:=\bigcup\limits_{k>(\log\frac{1}{\epsilon})^2}\Big\{0<\rmi_1<\cdots<\rmi_k<\frac1\epsilon\Big\},\quad\cC_2:=\bigcup\limits_{k=1}^{(\log\frac{1}{\epsilon})^2}\Big\{(\rmi_1,\cdots,\rmi_k)\in\Big(\Anotriple\Big)^c\Big\},
\end{equation*}
and $\overline{U}_N(n)$ and $u(n)$ are defined by \eqref{def:U_bar} and \eqref{def:u}. We denote the summation on $\cC_1$ and $\cC_2$ by $I_{N,\epsilon}$ and $J_{N,\epsilon}$ respectively, and we have the following upper bounds.

\begin{lemma}\label{lem:no_triple}
	We have that for some $C\in(0,+\infty)$, which may depend on $\vartheta$, such that
	\begin{align}
	\label{bd:In}&\limsup\limits_{N\to\infty}I_{N,\epsilon}\leq\frac{C}{\log\frac{1}{\epsilon}}\|\varphi\|_2^2\|\psi\|_\infty^2,\\
	\label{bd:Jn}&\limsup\limits_{N\to\infty}J_{N,\epsilon}\leq C\bigg(\int_0^{K_\epsilon\epsilon}\dd s\int_\bbR g_s^2(x)\dd x+\frac{(\log K_\epsilon)^2}{\log\frac{1}{\epsilon}}\bigg)\|\varphi\|_2^2\|\psi\|_\infty^2.
	\end{align}
	Consequently,
	\begin{equation*}
	\limsup\limits_{N\to\infty}\left\|\left(\cZ_{N,\epsilon}^{\text{\rm (no triple)}}\!-\!\cZ_N\right)\!(\varphi, \psi)\right\|_2^2\leq C\bigg(\int_0^{K_\epsilon\epsilon}\!\!\dd s\int_\bbR g_s^2(x)\dd x\!+\!\frac{(\log K_\epsilon)^2}{\log\frac{1}{\epsilon}}\bigg)\|\varphi\|_2^2\|\psi\|_\infty^2.
	\end{equation*}
\end{lemma}

\begin{proof}[Proof sketch]
	First, note that
	\begin{equation*}
	q_{f_k,N}^N(0,\psi)=\sum\limits_{v\in2\bbZ}p_{N-f_k}(v)\psi_N(v)\leq\|\psi_N\|_\infty\sum\limits_{v\in\bbZ}p_{N-f_k}(v)\leq\|\psi\|_\infty.
	\end{equation*}
	
	Next, we can use the following lemma to control the inner summation in \eqref{notriple1}.
	\begin{lemma}\label{lem:limit_var}
		Let $\bI:=(\rmi_1,\cdots,\rmi_k)\subset\{1,\cdots,1/\epsilon\}$ with $\rmi_1<\cdots<\rmi_k$. We have that
		\begin{equation}\label{limit_var}
		\begin{split}
		\lim\limits_{N\to\infty}\frac{1}{N}&\sum\limits_{d_1\leq f_1\in\cT_{\epsilon,N}(\rmi_1),\cdots,d_k\leq f_k\in\cT_{\epsilon,N}(\rmi_k)}\left(q_{0,d_1}^N(\varphi,0)\right)^2\overline{U}_N(f_1-d_1)\times\\
		&\bigg\{\prod\limits_{j=2}^k u(d_j-f_{j-1})\overline{U}_N(f_j-d_j)\bigg\}=\cI_{\epsilon}^{\varphi}(\bI)
		\end{split}
		\end{equation}
		with
		\begin{equation}\label{def:cI}
		\cI_{\epsilon}^{\varphi}(\bI):=\!\!\idotsint\limits_{a_1\leq b_1\in\cT_\epsilon(\rmi_1),\cdots,a_k\leq b_k\in\cT_\epsilon(\rmi_k)}\!\!g_{a_1}(\varphi,0)^2 G_\vartheta(b_1-a_1)\dd a_1\dd b_1\bigg\{\prod\limits_{j=2}^k\frac{G_\vartheta(b_j-a_j)}{a_j-b_{j-1}}\dd a_j\dd b_j\bigg\},
		\end{equation}
		where $\cT_\epsilon(\rmi):=(\epsilon(\rmi-1),\epsilon\rmi]$, $g_t(\varphi,0)$ is defined by \eqref{gphi}, and $G_\vartheta$ is introduced in \eqref{overlineU}-\eqref{asym:Gtheta}.
	\end{lemma}
	
	\begin{proof}
		Recall the local limit \eqref{def:u} for $u(n)$, the asymptotics \eqref{overlineUasym} and the uniform bound \eqref{overlineU} for $\overline{U}_N(n)$. The lemma then follows by a Riemann sum approximation, similar to the proof of Proposition \ref{prop:mean_var}. $\cI_{\epsilon}^\varphi(\bI)$ is well-defined by the renewal property (6.14) in \cite{CSZ19a}:
		\begin{equation}\label{G_renew}
		G_\vartheta(t)=\iint_{0<u<\bar{t}\leq v<t}G_\vartheta(u)\frac{1}{v-u}G_\vartheta(t-v)\dd u\dd v,\quad\forall t\in(\bar{t},+\infty).
		\end{equation}
	\end{proof}
	Note that by the renewal property \eqref{G_renew}, if for $\bI:=(\rmi_1,\cdots,\rmi_k)\subset\{1,\cdots,1/\epsilon\}$, we introduce
	\begin{equation}\label{def:cI_free}
	\cI_{\epsilon}^{s,t}(\bI):=\idotsint\limits_{s\leq b_1\in\cT_\epsilon(\rmi_1),\cdots,a_k\leq t\in\cT_\epsilon(\rmi_k)}G_\vartheta(b_1-s)\dd b_1\bigg\{\prod\limits_{j=2}^{k-1}\frac{G_\vartheta(b_j-a_j)}{a_j-b_{j-1}}\dd a_j\dd b_j\bigg\}\frac{G_\vartheta(t-a_k)}{a_k-b_{k-1}}\dd a_k,
	\end{equation}
	then $\forall~\rmj\leq\rmj',\forall~s\in\cT_{\epsilon}(\rmj)$, and $\forall~t\in\cT_{\epsilon}(\rmj')$, 
	\begin{equation}\label{G_renew2}
	\sum\limits_{k=1}^\infty\sum\limits_{\rmj=:\rmi_1<\rmi_2<\cdots<\rmi_{k-1}<\rmi_k:=\rmj'}\cI_{\epsilon}^{s,t}(\rmi_1,\cdots,\rmi_k)=G_\vartheta(t-s)
	\end{equation}
	
	By the above lemma, we have that
	\begin{align}
	\label{eq:I}\limsup\limits_{N\to\infty}I_{N,\epsilon}&\leq\|\psi\|_\infty\sum\limits_{k\geq(\log\frac1\epsilon)^2}\sum\limits_{0<\rmi_1<\cdots<\rmi_k\leq\frac1\epsilon}\cI_{\epsilon}^\varphi(\rmi_1,\cdots,\rmi_k)=:I_\epsilon,\\
	\label{eq:J}\limsup\limits_{N\to\infty} J_{N,\epsilon}&\leq\|\psi\|_\infty\sum\limits_{k=1}^{(\log\frac1\epsilon)^2}\sum\limits_{(\rmi_1,\cdots,\rmi_k)\in(\Anotriple)^c}\cI_{\epsilon}^\varphi(\rmi_1,\cdots,\rmi_k)=:J_\epsilon.
	\end{align}
	
	To bound \eqref{eq:I}, note that
	\begin{equation*}
	I_\epsilon\leq\frac{1}{(\log\frac{1}{\epsilon})^2}\sum\limits_{\bI\subset\{1,\cdots,\frac{1}{\epsilon}\},|\bI|\geq(\log\frac{1}{\epsilon})^2}|\bI|\cI_\epsilon^\varphi(\bI)\leq\frac{1}{(\log\frac{1}{\epsilon})^2}\sum\limits_{j=1}^{\frac{1}{\epsilon}}\sum\limits_{j\in\bI\subset\{1,\cdots,\frac{1}{\epsilon}\}}\cI_\epsilon^\varphi(\bI).
	\end{equation*}
	The result then follows by the renewal property \eqref{G_renew2} and direct computations.
	
	To bound \eqref{eq:J}, a fundamental tool is still the renewal property \eqref{G_renew2}. For the cases $\rmi_1<K_\epsilon$ and $\rmi_k>\frac{1}{\epsilon}-K_\epsilon$, since each $\cT_\epsilon(\rmi)$ is of length $\epsilon$, we have an integral on $[0,K_\epsilon\epsilon]$, whose length tends to $0$ by the choice of $K_\epsilon$ in \eqref{Keps}. For the case that there exists consecutive $\rmi-\rmi'<K_\epsilon$ and $\rmi''-\rmi'<K_\epsilon$, we need to consider separately the cases (1) $\rmi'-\rmi\geq2$ and $\rmi''-\rmi'\geq2$; (2) $\rmi'-\rmi=1$ or $\rmi''-\rmi'=1$. The case (1) is relatively simple, since we can bound the transition kernels between mesoscopic intervals from below uniformly. The case (2) needs a delicate computation, the readers are recommended to refer to \cite[Section 5.1]{CSZ21} for details.
\end{proof}



\subsection{Approximating $\cZ_{N,\epsilon}^{\text{\rm (no triple)}}$ by $\cZ_{N,\epsilon}^{\text{\rm (cg)}}$ and $\Lcg(\varphi,\psi)$}
Recall $\cZ_{N,\epsilon}^{\text{\rm (no triple)}}$, $\Zcg$ and $\Lcg$ from \eqref{Znotriple}, \eqref{cg_Z} and \eqref{cgmodel}. We have the following approximations.
\begin{lemma}\label{lem:cg}
	We have that for some $C\in(0,+\infty)$, which may depend on $\vartheta$, such that
	\begin{align}
	\label{approx:cg1}\limsup\limits_{N\to\infty}&\left\|\left(\Znotriple-\Zcg\right)(\varphi, \psi)\right\|_2^2\leq C\Big(\frac{1}{\log\frac1\epsilon}+\frac{(\log\frac1\epsilon)^2}{K_\epsilon^\kappa}\Big)\|\varphi\|_2^2\|\psi\|_\infty^2,\\
	\label{approx:cg2}\lim\limits_{N\to\infty}&\left\|\left(\Zcg-\Lcg(\varphi,\psi|\Theta)\right)(\varphi, \psi)\right\|_2^2=0.
	\end{align}
\end{lemma}

\begin{proof}
	We first prove \eqref{approx:cg1}. The approximation to $g_1(\varphi,\psi)$ by $q_{0,N}^N(\varphi,\psi)$ has been proved in \eqref{conv:mean}. There are two steps remained: (1) expanding $\Anotriple$ to $\vAnotriple$; (2) replacing $X_{d,f}$ by $\Theta_{N,\epsilon}(\vec{\rmi})$ and replacing the random walk transition kernels between $\Theta_{N,\epsilon}(\vec{\rmi})$ by heat kernels.
	
	\textbf{Step 1.} To expand $\Anotriple$ to $\vAnotriple$, we bind $X_{d,f}$ and $X_{d',f'}$ together via a transition probability $p_{d'-f}(0)$, where $d,f\in\cT_{\epsilon,N}(\rmi)$ and $d', f'\in\cT_{\epsilon, N}(\rmi')$ with $\rmi'-\rmi<K_\epsilon$. Since each $\vec{\rmi}=(\rmi,\rmi')$ can at most refer to 2 mesoscopic intervals. By expanding $\Anotriple$ to $\vAnotriple$, the terms refer to indices $\bI=(\rmi_1, \cdots, \rmi_k)$ with $(\log\frac1\epsilon)^2\leq k\leq 2(\log\frac1\epsilon)^2$ will be added, and the error has been controlled by \eqref{bd:Jn} and \eqref{eq:J}.

	
	\textbf{Step 2.} Then we replace the transition kernels by the heat kernels. Note that each time we bind two $X_{d_j,f_j}$ and $X_{d'_j,f'_j}$ together to get a $\Theta_{N,\epsilon}(\vec{\rmi}_j)$, we have an extra $\sqrt{\epsilon N}$. For the kernel connecting $\Theta_{N,\epsilon}(\vec{\rmi}_{j-1})$ and $\Theta_{N,\epsilon}(\vec{\rmi}_j)$, we have an approximation $\sqrt{\epsilon N}p_{d_j-f'_{j-1}}(0)\approx g_{(\rmi_j-\rmi'_{j-1})}(0)$. We need to show that such kernel replacement is controllable. By Theorem \ref{thm:local_1}, for fixed $\epsilon>0$ and $N$ large enough,
	\begin{equation}\label{eq:apply_local_limit2}
	\begin{split}
\Big|p_{d_j-f'_{j-1}}(0)-&\frac{1}{\sqrt{\epsilon N}}g_{(\rmi_j-\rmi'_{j-1})}(0)\Big|	\leq\frac{C}{(d_j-f'_{j-1})^{(1+\kappa)/2}}+\Big|g_{d_j-f'_{j-1}}(0)-\frac{1}{\sqrt{\epsilon N}}g_{(\rmi_j-\rmi'_{j-1})}(0)\Big|\\
	\leq&\frac{C}{(d_j-f'_{j-1})^{(1+\kappa)/2}}+\frac{C}{\sqrt{\epsilon N}(\rmi_j-\rmi'_{j-1})^{3/2}}\leq\frac{C}{\sqrt{\epsilon N}(\rmi_j-\rmi'_{j-1})^{(1+\kappa)/2}}.
	\end{split}
	\end{equation}

	For the summand in \eqref{cg_Z} and \eqref{Znotriple}, each time we perform the kernel replacement, we gain an $L^2$ error as follows (without loss of generality, we take $d_2,f_1$ as an example):
	
	\begin{equation*}
	\begin{split}
	\bigg\|&\frac{1}{\sqrt{N}}\!\!\sum\limits_{d_1\leq f_1\in\cT_{N,\epsilon,}(\rmi_1),\cdots,d_k\leq f_k\in\cT_{N,\epsilon}(\rmi_k)}\!\!q_{0,d_1}^N(\varphi,0)X_{d_1,f_1}\Big(p_{d_2-f_1}(0)-\frac{1}{\sqrt{\epsilon N}}g_{(\rmi_2-\rmi_1)}(0)\Big)X_{d_2,f_2}\\
	&\times\bigg\{\prod\limits_{j=3}^k p_{d_j-f_{j-1}}(0)X_{d_j,f_j}\bigg\}q_{f_k,N}^N(0,\psi)\bigg\|_2^2\\
	\leq&\frac{C\|\psi\|_\infty^2}{N}\!\!\!\!\!\!\sum\limits_{d_1\leq f_1\in\cT_{N,\epsilon}(\rmi_1),\cdots,d_k\leq f_k\in\cT_{N,\epsilon}(\rmi_k)}\!\!\!\!\!\!\!\big(q_{0,d_1}^N(\varphi,0)\big)^2\overline{U}_N(f_1-d_1)\frac{1}{\epsilon N(\rmi_2-\rmi_1)^{1+\kappa}}\overline{U}_N(f_2-d_2)\\
	&\times\bigg\{\prod\limits_{j=3}^k u(d_j-f_{j-1})\overline{U}_N(f_j-d_j)\bigg\}.
	\end{split}
	\end{equation*}
	Note that $1/(\epsilon N(\rmi_2-\rmi_1))\leq C/(d_2-f_1)$ and $\rmi_2-\rmi_1\geq K_\epsilon$ by the ``no triple'' condition in $\vAnotriple$. By Lemma \ref{lem:limit_var}, taking $\limsup_{N\to\infty}$, the above term is bounded by $C\|\psi\|_\infty^2(K_\epsilon)^{-\kappa}\cI_{\epsilon}^{\varphi}(\rmi_1,\cdots,\rmi_k)$. Since there are at most $2(\log\frac1\epsilon)^2$ kernel replacements, by triangular inequality, the total $L^2$ error for kernel replacements is bounded by $C\|\psi\|_\infty^2$ timing $(\log\frac1\epsilon)^2(K_\epsilon)^{-\kappa}\cI_{\epsilon}^{\varphi}(\rmi_1,\cdots,\rmi_k)$. Finally, by summing over all $k$ and $\rmi_1,\cdots,\rmi_k$, and by the orthogonality of different $\rmi_1,\cdots,\rmi_k$, the upper bound for kernel replacements is
	\begin{equation*}
	\begin{split}
	&\limsup\limits_{N\to\infty}C\|\psi\|_\infty^2\big(\log\frac1\epsilon\big)^2 (K_\epsilon)^{-\kappa}\var(\cZ_N(\varphi,\psi))\\
	\leq&\frac{C\big(\log\frac1\epsilon\big)^2\|\psi\|_\infty^2}{K_\epsilon^\kappa}\int_0^1 G_\vartheta(t)\dd t\int_0^{1}g_s(\varphi,0)^2\dd s\leq\frac{C_\vartheta\big(\log\frac1\epsilon\big)^2}{K_\epsilon^\kappa}\|\psi\|_\infty^2\|\varphi\|_2^2.
	\end{split}
	\end{equation*}
	
	Combine \textbf{step 1} and \textbf{step 2} and we conclude \eqref{approx:cg1}.
	
	\vspace{0.25cm}
	Finally we prove \eqref{approx:cg2}, the approximation to $\Zcg$ by $\Lcg(\varphi,\psi)$. This is just a Riemann sum approximation from $\sum\limits_{\rma\sqrt{\epsilon N}<u\leq(\rma+1)\sqrt{\epsilon N},u\in\bbZ}\varphi_N(u)$ and $\sum\limits_{\rmb\sqrt{\epsilon N}<v\leq(\rmb+1)\sqrt{\epsilon N},v\in\bbZ}\psi_N(v)$ to $\varphi_\epsilon(\rma)$ and $\psi_\epsilon(\rmb)$ respectively.
\end{proof}

\section{Moment estimates for $\sL_\epsilon^{\text{\rm (cg)}}(\varphi,\psi|\Theta)$}\label{S8}
In this section, we first bound the second and fourth moments of the coarse-grained disorder $\Theta_{N,\epsilon}(\vec{\rmi})$ (see \eqref{cgvariable}), and then we will use these bounds to bound the moments of the coarse-grained model $\Lcg(\varphi,\psi|\Theta)$.

\subsection{Moment bounds for the coarse-grained variable $\Theta_{N,\epsilon}$}
For the second moment of $\Theta_{N,\epsilon}$, we have the following bound.
\begin{lemma}\label{lem:varTheta}
	For any $\vec{\rmi}$, we have that $\bbE[\Theta_{N,\epsilon}(\vec{\rmi})]=0$ and the second moment of $\Theta_{N,\epsilon}(\vec{\rmi})$ converges, that is, there exists some $\sigma_{\epsilon}^2(\vec{\rmi})\geq0$, such that
	\begin{equation}
	\label{varTheta}\sigma_\epsilon^2(\vec{\rmi}):=\lim\limits_{N\to\infty}\bbE\Big[\big(\Theta_{N,\epsilon}(\vec{\rmi})\big)^2\Big].
	\end{equation}
	Furthermore, there exists some constant $C>0$, independent of $\vec{\rmi}$ and $\epsilon$, such that
	\begin{equation}
	\label{varbdTheta}\sigma_\epsilon^2(\vec{\rmi})\leq\frac{C}{\big(\log\frac{1}{\epsilon}\big)^{1+\mathbbm{1}_{\{|\vec{\rmi}|\geq2\}}}|\vec{\rmi}|}.
	\end{equation}
\end{lemma}

\begin{proof}[Proof sketch]
	The computation is similar to those for Proposition \ref{prop:mean_var} and Lemma \ref{lem:limit_var}. When $|\vec{\rmi}|\leq2$,  $\Theta_{N,\epsilon}(\vec{\rmi})$ is an average of $X_{d,f}$ (see \eqref{def:X}),which has the same structure as $\cZ_N(\varphi,\psi)$ with $\varphi(x)=\ind_0(x)$ and $\psi(y)=\ind_0(y)$. When $|\vec{\rmi}|\geq3$, we take advantage of the fact that $p_{d'-f}(0)$ in \eqref{cgvariable} is uniformly bounded below by $C(|\vec{\rmi}|\epsilon N)^{-1/2}$ and decouple $X_{d,f}$ and $X_{d',f'}$ by the independence of $\zeta_n$.
\end{proof}

For the fourth moment of $\Theta_{N,\epsilon}$, we have the following bound.
\begin{lemma}\label{lem:4mmtTheta}
	For any $\vec{\rmi}$ and sufficiently small $\epsilon>0$, there exists a uniform constant $C>0$, such that
	\begin{equation*}
	\limsup\limits_{N\to\infty}\bbE\Big[\big(\Theta_{N,\epsilon}(\vec{\rmi})\big)^4\Big]\leq\frac{C}{\big(\log\frac{1}{\epsilon}\big)^{1+\ind_{|\vec{\rmi}|\geq3}}|\vec{\rmi}|^2}.
	\end{equation*}
\end{lemma}

\begin{proof}[Proof sketch]
	The strategy is the same as in Lemma \ref{lem:varTheta} and the difference is that we apply Theorem \ref{T:hmom}. We split the cases $|\vec{\rmi}|\leq2$ and $|\vec{\rmi}|\geq 3$. Note that in the former case, the length of polymer chains is $|\vec{\rmi}|\epsilon N$ and we can apply Theorem \ref{T:hmom} with $\varphi(x)=\ind_0(x)$ and $\psi(y)=\ind_0(y)$ (see Remark \ref{rmk:hmom}). In the latter case, as in the proof of Lemma \ref{lem:varTheta}, we can decouple the fourth moments of $X_{d,f}$ and $X_{d',f'}$ in \eqref{cgvariable} by the independence.
\end{proof}

\subsection{Fourth moment bounds for $\Lcg(\varphi,\psi|\Theta)$}
Recall from \eqref{cgmodel} that
\begin{equation*}
\begin{split}
\Lcg(\varphi,\psi|\Theta)&:=g_1(\varphi,\psi)+\sqrt{\epsilon}\sum\limits_{r=1}^{(\log\frac{1}{\epsilon})^2}\sum\limits_{(\vec{\rmi}_1,\cdots,\vec{\rmi}_r)\in\vAnotriple}\sum\limits_{\rma,\rmb\in\bbZ}\\
&\varphi_\epsilon(\rma)g_{\rmi_1}(\rma)\Theta(\vec{\rmi}_1)\times\bigg\{\prod_{j=2}^r g_{(\rmi_j-\rmi'_{j-1})}(0)\Theta(\vec{\rmi}_j)\bigg\}g_{(\lfloor\frac{1}{\epsilon}\rfloor-\rmi'_r)}(\rmb)\psi_\epsilon(\rmb).
\end{split}
\end{equation*}
We have the following fourth moment bound for $\Lcg(\varphi,\psi)=\Lcg(\varphi,\psi|\Theta)$, which is an analogue to Theorem \ref{T:hmom}.
\begin{theorem}\label{Thm:4mom}
	Recall $\Lcg(\varphi,\psi)$ from \eqref{cgmodel}. Suppose that $\psi$ is supported on $B$ (could be $\bbR$) with $\|\psi\|_\infty<+\infty$. Then for any $p,q\in(1,+\infty)$ with $\frac1p+\frac1q=1$ and $w:\bbR\to(0,+\infty)$ such that $\log w$ is Lipschitz continuous, there exists a constant $C\in(0,+\infty)$, such that uniformly for any $\epsilon\in(0,1)$,
	\begin{equation*}
	\limsup\limits_{N\to\infty}\bbE\Big[\big(\Lcg(\varphi,\psi)-\bbE\big[\Lcg(\varphi,\psi)\big]\big)^4\Big]\leq C\epsilon^{\frac2p}\bigg\|\frac{\varphi_\epsilon}{w_\epsilon}\bigg\|^4_{\ell^p}\|\psi\|_\infty^4\|w\ind_B\|^4_{\ell^q},
	\end{equation*}
	where $\varphi_\epsilon$ and $w_\epsilon$ are defined through \eqref{epstestfunc}.
\end{theorem}

\begin{remark}\label{rmk:hmom2}
	We introduce $w$ since in Section \ref{S9} we will consider $\psi$ with unbounded support. A typical choice of $w(x)$ is $e^{-|x|}$.
\end{remark}

\begin{proof}
	We have that
	\begin{equation}\label{def:cM}
	\begin{split}
	\cM_{N,\epsilon}^{\varphi,\psi}&:=\bbE\Big[\big(\Lcg(\varphi,\psi)-\bbE\big[\Lcg(\varphi,\psi)\big]\big)^4\Big]=\epsilon^2\bbE\bigg[\bigg(\sum\limits_{r=1}^{(\log\frac{1}{\epsilon})^2}\sum\limits_{(\vec{\rmi}_1,\cdots,\vec{\rmi}_r)\in\vAnotriple}\sum\limits_{\rma,\rmb\in\bbZ}\\
	&\quad\quad\quad\varphi_\epsilon(\rma)g_{\rmi_1}(\rma)\Theta_{N,\epsilon}(\vec{\rmi}_1)\times\Big\{\prod_{j=2}^r g_{(\rmi_j-\rmi'_{j-1})}(0)\Theta_{N,\epsilon}(\vec{\rmi}_j)\Big\}g_{(\lfloor\frac{1}{\epsilon}\rfloor-\rmi'_r)}(\rmb)\psi_\epsilon(\rmb)\bigg)^4\bigg].
	\end{split}
	\end{equation}
	
	Similar to the proof of Theorem \ref{T:hmom} (see \eqref{eq:enlargeN}-\eqref{eq:Mbound1}), we can find some some constant $C\in(0,\infty)$, such that
	
	
	
	\begin{equation}\label{bd:4mom1}
	\begin{split}
	\cM_{N,\epsilon}^{\varphi,\psi}\leq C&\|\psi\|_\infty^4\epsilon^3\sum\limits_{\rmn=1}^{2/\epsilon}\bbE\bigg[\bigg(\sum\limits_{r=1}^{\infty}\sum\limits_{(\vec{\rmi}_1,\cdots,\vec{\rmi}_r)\in\vAnotriple}\sum\limits_{\rma,\rmb\in\bbZ\atop|\rma|\leq M_\epsilon\sqrt{\rmi_1},|\rmb|\leq M_\epsilon\sqrt{\lfloor\frac{1}{\epsilon}\rfloor-\rmi_r'}}\\
	\varphi_\epsilon(\rma)&g_{\rmi_1}(\rma)\Theta_{N,\epsilon}(\vec{\rmi}_1)\times\Big\{\prod_{j=2}^r g_{(\rmi_j-\rmi'_{j-1})}(0)\Theta_{N,\epsilon}(\vec{\rmi}_j)\Big\}g_{(\rmn-\rmi'_r)}(\rmb)\ind_{B_\epsilon}(\rmb)\bigg)^4\bigg].
	\end{split}
	\end{equation}
	Here $2/\epsilon$ plays the role of $2N$ in \eqref{eq:Mbound1}.
	
	\vspace{0.25cm}
	\noindent\textbf{Summation rearrangement}
	
	Next, we rearrange the summation in \eqref{bd:4mom1}. If we view $\vec{\rmi}_j$ as a renewal time, then after expanding the fourth moment, we have 4 sequences of renewal times. Furthermore, if $|\vec{\rmi}_j|\geq2$, then $\rmi_j,\rmi'_j$ are recorded as 2 distinct renewal times. We denote each sequence of renewal times by $\{\rmi_i^j\}_{i=1}^{r_j}$ for $j=1,2,3,4$, and call them \textit{mesoscopic renewal sequences}.
	
	We then perform the following rearrangement for $\{\rmi_i^j\}_{i=1}^{r_j}$, $j=1,2,3,4$:
	
	(1) Sum over the set $\bigcup_{j=1}^4\bigcup_{i=1}^{r_j}\{\rmi_i^j\}:=\{n_1,\cdots,n_r\}$ (all renewal times).
	
	(2) For each $n_i$, $1\leq i\leq r$, sum over the indices set $J_{n_i}\subset\{1,2,3,4\}$. Here $J_{n_i}$ records the indices of the 4 renewal sequences which have a renewal time $n_i$.
	
	Recall that for $\vec{\rmi}\neq\vec{\rmj}$, $\bbE[\Theta_{N,\epsilon}(\vec{\rmi})\Theta_{N,\epsilon}(\vec{\rmj})]=0$ by the property of $\zeta_n$ (see \eqref{zeta_mean_var}). Hence, a summand in (1)-(2) above is nonzero if and only if for any $n_i$, we must have $|J_{n_i}|\geq2$. We also have the following constraints on $(n_1,\cdots, n_r)$ and $J_1,\cdots,J_r$ by $(\vec{\rmi}_1,\cdots,\vec{\rmi}_r)\in\vAnotriple$:
	
	(a) $K_\epsilon\leq n_1<n_2<\cdots<n_r\leq\frac1\epsilon-K_\epsilon$ (recall $K_\epsilon$ from \eqref{Keps}).
	
	(b) $(n_1,\cdots,n_r)$ can be partitioned into consecutive blocks $\cD_1,\cdots,\cD_m$, where each $\cD_i$ consists of consecutive integers and $\max\cD_i\leq\min\cD_{i+1}-2$.
	We must have $|\cD_i|\leq4$ for all $i=1,\cdots,m$. The reason is that $|J_{n_i}|\geq2$ for all $i=1,\cdots,r$, and by the definition of $\vAnotriple$, there are at most two renewal times within a time period $K_\epsilon$. Hence, $|\cD_i|\leq4$ can be obtained by simple counting.
	
	Starting from the constraint (b) above, we can further group the blocks $\cD_1,\cdots,\cD_m$. If for consecutive $\cD_{i_1},\cD_{i_1+1},\cdots,\cD_{i_2}$ ($i_1$ may equal to $i_2$, that is, we have a single $D_{i_1}$), and for any $n_j\in\bigcup_{i=i_1}^{i_2}\cD_i$, we always have that $J_{n_j}=\{k,l\}$ with $k\neq l$, then we only record $s:=\min\cD_{i_1}$ and $t:=\max\cD_{i_2}$ and identify the integer sequence $\bigcup_{i=i_1}^{i_2}\cD_i$ by $[s, t]$ (similar to the construction of \eqref{eq:bigU}). For other (single) $\cD_i$, we also denote $s_i=\min\cD_i$ and $t_i=\max\cD_i$. Then we can further group $\cD_1,\cdots,\cD_m$ according to intervals $\cI_i=[s_i,t_i]\cap\bbN$, $1\leq i\leq k$. We can now rearrange the summing strategy (1) and (2) as follows:
	
	(1') Sum over integers $K_\epsilon<s_1\leq t_1<s_2\leq t_2<\cdots<s_k\leq t_k<\rmn\leq\frac2\epsilon$ with $s_{i+1}-t_i\geq2$ for $i=2,\cdots,k-1$. Recall that $\rmn$ is introduced in \eqref{bd:4mom1}.
	
	(2') Suppose that $\cI_i=[s_i, t_i]$ represents some blocks $\bigcup_{i=i_1}^{i_2}\cD_i$ with $\min\cD_{i_1}=s_i$ and $\max\cD_{i_2}=t_i$. Recall the indices set $J_{n_i}$ introduced in (2) above. We write $\cJ_i=\bigcup_{n_i} J_{n_i}$, where $n_i$ runs over all renewal times in $\bigcup_{i=i_1}^{i_2}\cD_i$, that is, $\cJ_i$ records all indices of the mesoscopic renewal processes that renew at least once at some time $n_i\in\bigcup_{i=i_1}^{i_2}\cD_i$. If $|\cJ_i|=2$, then we call the related $\cI_i$ a \textit{group of type U}, since it is a counterpart of \eqref{eq:bigU}. If $|\cJ_i|\geq3$, then we call the related $\cI_i$ a \textit{group of type V} and it is a counterpart of $\zeta_n$ in the expansion of $\cZ_N^{\beta_N}(\varphi,\psi)$. Note that as we discussed above, two consecutive U blocks must have different indices set $\cJ_i$ and $\cJ_{i+1}$, while a V block must have length $|\cI_i|\leq4$. Finally, for each $\cI_i=[s_i,t_i]$, we sum over the indices set $\cJ_i$.
	
	\vspace{0.25cm}
	\noindent\textbf{Adjustments on coarse-grained disorders and heat kernels}
	
	To treat the summations (1') and (2'), we use the framework in the proof of Theorem \ref{T:hmom}. We start with some surgery on $\Theta_{N,\epsilon}(\vec{\rmi})$, the heat kernel, and the random walk transition probability.
	
	(A) We unbind $X_{d,f}$ and $X_{d',f'}$ in the definition of $\Theta_{N,\epsilon}(\vec{\rmi})$ (see \eqref{cgvariable}) when $|\vec{\rmi}|>1$ and $\rmi,\rmi'$ belong to two different groups $\cI_i,\cI_{i'}$, so that we have two coarse-grained disorder $\Theta_{N,\epsilon}(\vec{\rmi}_1)$ and $\Theta_{N,\epsilon}(\vec{\rmi}_2)$ with $\vec{\rmi}_1=(\rmi, \rmi)$ and $\vec{\rmi}_2=(\rmi',\rmi')$.
	
	(B) Decompose the heat kernel by Chapman-Kolmogorov equation, so that we can sum over all 4 mesoscopic renewal processes at any time $n$, since some of them renew but some others do not.
	
	For step (A), note that since $\cI_i$ and $\cI_{i'}$ are different groups, we have that $s_{i'}-t_i\geq2$ . Hence, $d'-f\geq(s_{i'}-t_i-1)\epsilon N\geq\epsilon N$. Moreover, If $\cI_i$ is a U group, then $\rmi=t_i$. Otherwise, if $\rmi<t_i$, then $\rmi<\rmi'\leq t_i$ since the indices of renewal processes in a U group cannot change, and thus $\rmi'$ cannot belong to another group $\cI_{i'}$. If $\cI_i$ is a V group, then $\rmi\geq s_i\geq t_i-3$, since the length of a V group can at most be $4$. Similarly, if $\cI_{i'}$ is a U group, then $\rmi'=s_{i'}$ and if $\cI_{i'}$ is a V group, then $\rmi'\leq t_{i'}\leq s_{i'}+3$. In all cases, we have that $d'-f\leq(t_{i'}+1-s_i)\leq(s_{i'}-t_i+7)\epsilon N$. By local limit theorem, we have that uniformly in all $f\in\cT_{N,\epsilon}(\rmi)$ and $d'\in\cT_{N,\epsilon}(\rmi')$, there exists some $c\in(0,+\infty)$, such that
	\begin{equation*}
	p_{f,d'}(0)\leq\frac{C}{\sqrt{2\pi(d'-f)}}\leq\frac{C'}{\sqrt{\epsilon N}}g_{(s_{i'}-t_i)}(0).
	\end{equation*}
	Note that we are considering the moment bounds and we will bound every moment of $\Theta_{N,\epsilon}(\vec{\rmi})$ by its absolute value, which can be done by bounding the absolute value of moments of all $\zeta_n$'s (see Theorem \ref{T:hmom}). Hence, we can upper bound any moment of $\Theta_{N,\epsilon}((\rmi,\rmi'))$ above by the related moment of $C'\Theta_{N,\epsilon}((\rmi,\rmi))g_{(s_{i'}-t_i)}(0)\Theta_{N,\epsilon}((\rmi',\rmi'))$. We will see that the constant $C'$ is not important.
	
	For step (B), note that after the modifications in step (A), any two different groups $\cI_i$ and $\cI_j$ that are visited consecutively by one mesoscopic renewal process are always connected by a heat kernel. If the heat kernel comes from a modification of $p_{f,d'}(0)$, then it is now $g_{s_j-t_i}(0)$. If the heat kernel is an original one, for example $g_{\rmj-\rmi}(0)$ for some $\rmi\in\cI_i$ and $\rmj\in\cI_j$, then we must have $\rmj-\rmi>K_\epsilon$. Similarly to the discussion in step (A), $t_i-3\leq\rmi\leq t_i$ and $s_j\leq\rmj\leq s_j+3$. Hence, this heat kernel can also be replaced by $Cg_{s_j-t_i}(0)$. Next, we can squeeze the length of each V group to $1$. 
	By squeezing V groups, for two consecutive renewal times $\rmi\in\cI_i$ and $\rmj\in\cI_j$, the new distance $u$ between the two renewal times $\rmi$ and $\rmj$ satisfies $(s_j-t_i+7)\geq u\geq\frac14(\rmj-\rmi)\geq\frac14(s_j-t_i)$, since each V group has length at most $4$. From now on, we can assume that for any V group $\cI_i$, its length is one, that is $s_i=t_i$, and $g_u(0)\leq 2g_{(s_j-t_i)}(0)$. To this end, all heat kernels are $Cg_{(s_j-t_i)}(0)$.
	
	By Chapman-Kolmogorov equation, for any $t:=u_0<u_1<\cdots<u_k=:s$, we have that
	\begin{equation*}
	\begin{split}
	g_{s-t}(0)&=\idotsint\limits_{x_1,\cdots,x_{k-1}\in\bbR}g_{u_1-u_0}(x_1)\prod\limits_{i=2}^{k-1}g_{(u_i-u_{i-1})}(x_i-x_{i-1})g_{(u_k-u_{k-1})}(-x_{k-1})\dd x_1\cdots\dd x_{k-1}\\
	&\leq C^k\sum\limits_{z_1,\cdots,z_{z-1}\in\bbZ}g_{u_1-u_0}(z_1)\prod\limits_{i=2}^{k-1}g_{(u_i-u_{i-1})}(z_i-z_{i-1})g_{(u_k-u_{k-1})}(z_{k-1}),
	\end{split}
	\end{equation*} 
	for some $C$ uniformly. By this decomposition, we introduce ``virtual spatial positions'', so that when a mesoscopic renewal process does not renew, we can allocate a ``position'' to it. Then, we can treat the summation similar to the proof of Theorem \ref{T:hmom}.
	
	\vspace{0.25cm}
	\noindent\textbf{Moments estimates}
	
	We now consider the expectation of a U group $\cI=[s,t]$. Without loss of generality, we assume that the related indices set for mesoscopic renewal processes is $\cJ=\{1,2\}$. Taking expectation in \eqref{bd:4mom1}, each U group leads to a following quantity
	\begin{equation}\label{mean:Ugroup}
	\overline{U}_{N,\epsilon}^{(\rm cg)}(t-s):=\sum\limits_{(\vec{\rmi}_1,\cdots,\vec{\rmi}_r)\in\vAnotriple\atop\rmi_1=s,\rmi'_r=t}\bbE\big[(\Theta_{N,\epsilon}(\vec{\rmi}_1))^2\big]\prod\limits_{j=2}^r g^2_{(\rmi_j-\rmi'_{j-1})}(0)\bbE\big[(\Theta_{N,\epsilon}(\vec{\rmi}_j))^2\big].
	\end{equation}
	For a fixed $\epsilon>0$, there are only finite many terms in the above summation. Hence, by Lemma \ref{lem:varTheta}, the following limit exists
	
	\begin{equation}\label{mean:limUgroup}
	\overline{U}_{\infty,\epsilon}^{\rm (cg)}(t-s):=\lim\limits_{N\to\infty}\overline{U}_{N,\epsilon}^{(\rm cg)}(t-s).
	\end{equation}
	
	We then consider the expectation of a V group $\cI=[s,t]$ with $t-s\leq3$. Let $\cJ\subset\{1,2,3,4\}$ be its related indices set for mesoscopic renewal processes. Note that we must have $|\cJ|\geq 3$ to have a nonzero mean ($|\cJ|=2$ refers to a U group). For each $j\in\cJ$, the $j$-th mesoscopic renewal process can renew in $\cI$ for at most two times. Hence, the relevant coarse-grained disorder can always be written by $\Theta_{N,\epsilon}((\rmi_j,\rmi'_j))$, where $\rmi_j$ is the entering time and $\rmi'_j$ is the exiting time (it is possible that $\rmi_j=\rmi'_j$). Then, taking expectation in \eqref{bd:4mom1}, each $V$ group leads to a following quantity
	\begin{equation}\label{mean:Vgroup}
	V_{N,\epsilon}^{\rm(cg)}(\cJ):=\bbE\Big[\prod_{j\in\cJ}\Theta_{N,\epsilon}((\rmi_j,\rmi'_j))\Big].
	\end{equation}
	By H\"{o}lder's inequality and Lemma \ref{lem:4mmtTheta}, we have that
	\begin{equation}\label{mean:limVgroup}
	\limsup_{N\to\infty}V_{N,\epsilon}^{\rm(cg)}(\cJ)\leq\limsup_{N\to\infty}\prod\limits_{j\in\cJ}\bbE\big[\Theta_{N,\epsilon}((\rmi_j,\rmi'_j))^4\big]^{\frac14}:=V_{\infty,\epsilon}^{\rm(cg)}(\cJ)\leq\frac{C}{(\log\frac1\epsilon)^{\frac{3}{4}}}.
	\end{equation}
	
	We now derive an upper bound for $\limsup_{N\to\infty}\cM_{N,\epsilon}^{\varphi,\psi}$ similar to \eqref{eq:Mbound1}. For this purpose, we replace the indices set $\cJ$ by partition $I$ to make the notations consistent. Note that we actually have $\cJ=I(1)$ (recall the definition of $I$ from the paragraph above \eqref{meanI}). Recall the summing strategy (1') and (2') above. $\limsup_{N\to\infty}M_{N,\epsilon}^{\varphi,\psi}$ can be expressed as follows:
	
	(1'') Suppose that there are $r$ groups $\cI_1,\cI_2,\cdots,\cI_r$. We sum over $K_\epsilon<s_1\leq t_1<s_2\leq t_2<\cdots<s_r\leq t_r<s_{r+1}\leq\frac2\epsilon$, where $\cI_i=[s_i, t_i]$ and $s_{r+1}$ plays the role of $\rmn$ in \eqref{bd:4mom1}.
	
	(2'') For each $1\leq i\leq r$, we sum over $I_i$ with $|I_i|\leq3$. When $|I_i|=3$, we have that $\cI_i$ is a U group, while when $|I_i|<3$, we have that $\cI_i$ is a V group. Note that to have a nonzero expectation, we cannot have $|I_i|=4$ for any $i$.
	
	(3'') We sum over $r\geq1$ (note that $r=0$ refers to the mean, which has been subtracted).
	
	We now introduce some notations similar to \eqref{def:Q} and \eqref{eq:bigU}. For $\boldsymbol{x},\boldsymbol{x}'\in(\bbZ)^4$, let
	\begin{align}\label{def:cQ}
	\cQ_t(\boldsymbol{x},\boldsymbol{x}')&:=\prod\limits_{i=1}^4 g_t(x'_i-x_i),\\
	\cQ_t(\varphi,\boldsymbol{x}')&:=\prod\limits_{i=1}^4\Big(\sum\limits_{y_i\in\bbZ}\varphi(y_i)g_t(x'_i)\Big),\\
	\cQ_t(\boldsymbol{x},\psi)&:=\prod\limits_{i=1}^4\Big(\sum\limits_{y'_i\in\bbZ}g_t(x_i)\psi(y'_i)\Big).
	\end{align}
	Recall \eqref{mean:limUgroup}, \eqref{mean:limVgroup}, and $\cJ=I(1)$, we define
	\begin{align}
	\label{def:bU}\bU^{\text{(cg)},I}_{\infty,\epsilon}(t;\boldsymbol{x},\boldsymbol{x}')&:=\overline{U}^{\text{(cg)}}_{\infty,\epsilon}(t)\prod\limits_{j\notin I(1)}g_{t}(x'_j-x_j),\\
	\label{def:bV}\bV^{\text{(cg)},I}_{\infty,\epsilon}(\boldsymbol{x},\boldsymbol{x}')&:=V^{\text{(cg)}}_{\infty,\epsilon}(I(1))\prod\limits_{j\notin I(1)}g_{0}(x'_j-x_j),
	\end{align}
	where for $\bV^{\text{(cg)},I}_{\infty,\epsilon}(\boldsymbol{x},\boldsymbol{x}')$, we actually have $\boldsymbol{x}=\boldsymbol{x}'$, since we have squeezed the length of V group to 1. To lighten the notations, we combine \eqref{def:bU} and \eqref{def:bV} by introducing
	\begin{equation}\label{def:bW}
	\bW_{\infty,\epsilon}^{{\rm (cg)}, I}(t; \boldsymbol{x}, \boldsymbol{x}'):=\ind_{\{|I|=3\}}\bU^{\text{(cg)},I}_{\infty,\epsilon}(t;\boldsymbol{x},\boldsymbol{x}')+\ind_{\{|I|<3\}}\bV^{\text{(cg)},I}_{\infty,\epsilon}(\boldsymbol{x},\boldsymbol{x}').
	\end{equation}
	
	Finally, we can write down an upper bound for $\cM_{\infty,\epsilon}^{\varphi,\psi}:=\limsup_{N\to\infty}M_{N,\epsilon}^{\varphi,\psi}$ according to summation (1'')-(3''):
	
	\begin{equation}\label{eq:4mom2}
	\begin{split}
	\cM_{\infty,\epsilon}^{\varphi,\psi}\leq c&\|\psi\|_\infty^4\epsilon^3\sum\limits_{r=1}^\infty C^r\!\!\!\!\!\!\sum\limits_{I_1,\cdots,I_r,|I_i|\leq3, 1\leq i\leq r\atop K_\epsilon<s_1\leq t_1<\cdots<s_r\leq t_r<s_{r+1}\leq\frac2\epsilon}\sum\limits_{\boldsymbol{x}_1, \cdots, \boldsymbol{x}_r\in(\bbZ)^4\atop\boldsymbol{x}'_1,\cdots,\boldsymbol{x}'_r\in(\bbZ)^4}\!\!\cQ_{s_1}(\varphi_\epsilon,\boldsymbol{x}_1)\cQ_{(s_{r+1}-t_r)}(\boldsymbol{x}'_r,\ind_{B_\epsilon})\\
	\times&\Big\{\prod\limits_{i=1}^{r-1}\bW_{\infty,\epsilon}^{{\rm (cg)}, I_i}(t_i-s_i; \boldsymbol{x}_i, \boldsymbol{x}'_i)\cQ_{(s_{i+1}-t_i)}(\boldsymbol{x}'_i,\boldsymbol{x}_{i+1})\Big\}\bW_{\infty,\epsilon}^{{\rm (cg)}, I_r}(t_r-s_r; \boldsymbol{x}_r, \boldsymbol{x}'_r).
	\end{split}
	\end{equation}
	Here the factor $C^r$ comes from all heat kernel adjustments. Note that those adjustments are only needed at renewal times $s_i$ and $t_i$ for $1\leq i\leq r$ and thus we have a power $r$. Also note that if $|I_i|<3$, then $s_i=t_i$ by our squeezing on V group.
	
	Similar to \eqref{Qlambda}-\eqref{def:op}, we define
	\begin{align}
	\label{cQlambda}\cQ_{\lambda,\epsilon}^{I,J}(\boldsymbol{x},\boldsymbol{x}')&:=\sum\limits_{n=1}^{2/\epsilon}e^{-\lambda n}\cQ_n(\boldsymbol{x},\boldsymbol{x}'),\\
	\label{bUlambda}\bU_{\lambda,\epsilon}^I(\boldsymbol{x},\boldsymbol{x}')&:=\sum\limits_{n=0}^{2/\epsilon}e^{-\lambda n}\bU^{\text{(cg)},I}_{\infty,\epsilon}(t;\boldsymbol{x},\boldsymbol{x}'),\\
	\label{bVlambda}\bV_{\lambda,\epsilon}^{I}(\boldsymbol{x}.\boldsymbol{x}')&:=\bV^{\text{(cg)},I}_{\infty,\epsilon}(\boldsymbol{x},\boldsymbol{x}'),\\
	\label{def:cgop}\cP_{\lambda,\epsilon}^{I,J}&:=\begin{cases}
	\cQ_{\lambda,\epsilon}^{I,J}\bV_{\lambda,\epsilon}^{J},\quad&\text{if}~|J|<3,\\
	\cQ_{\lambda,\epsilon}^{I,J}\bU_{\lambda,\epsilon}^{J},\quad&\text{if}~|J|=3.
	\end{cases}	
	\end{align}
	
	Then, we insert $w$ to define a weighted version of $\cQ_{\lambda,\epsilon}^{I,J}(\boldsymbol{x},\boldsymbol{x}')$ by
	\begin{equation}\label{weighted_op}
	\hat{\cQ}_{\lambda,\epsilon}^{I,J}(\boldsymbol{x},\boldsymbol{x}'):=\frac{w_\epsilon^{\otimes4}(\boldsymbol{x})}{w_\epsilon^{\otimes4}(\boldsymbol{x}')}\cQ_{\lambda,\epsilon}^{I,J}(\boldsymbol{x},\boldsymbol{x}').
	\end{equation}
	The weighted operators $\hat{\bU}_{\lambda,\epsilon}^I(\boldsymbol{x},\boldsymbol{x}')$, $\hat{\bV}_{\lambda,\epsilon}^{I}(\boldsymbol{x}.\boldsymbol{x}')$ and $\hat{\cP}_{\lambda,\epsilon}^{I,J}$ are defined in the same way.
	
	Introducing a factor $e^{2\lambda/\epsilon}e^{-\lambda\sum_{i=1}^{r+1}s_i-t_{i-1}-\lambda\sum_{i=1}^r(t_i-s_i)}\geq1$ with $t_0:=0$, \eqref{eq:4mom2} is then bounded by
	\begin{equation}\label{eq:4mom3}
	\Big|\cM_{\infty,\epsilon}^{\varphi,\psi}\Big|\leq ce^{\frac{2\lambda}{\epsilon}}\|\psi\|_\infty^4\epsilon^3\sum\limits_{r=1}^\infty C^r\sum\limits_{I_1,\cdots,I_r\atop|I_i|\leq 3, 1\leq i\leq r}\Big\langle\frac{\varphi_\epsilon^{\otimes4}}{w_\epsilon^{\otimes4}},\hat{\cP}_{\lambda,\epsilon}^{*,I_1}\hat{\cP}_{\lambda,\epsilon}^{I_1,I_2}\cdots\hat{\cP}_{\lambda,\epsilon}^{I_{r-1},I_r}\hat{\cQ}_{\lambda,\epsilon}^{I_r,*}w_\epsilon\ind_{B_\epsilon}^{\otimes4}\Big\rangle,
	\end{equation}
	where ``$*$'' is the partition of $\{1,2,3,4\}$ consisting of all singletons, and $\lambda:=\hat{\lambda}\epsilon$ with $\hat{\lambda}$ a large but fixed constant, such that $e^{2\lambda/\epsilon}$ is bounded. We have the following proposition to control \eqref{eq:4mom3}
	
	\begin{proposition}\label{prop:cg_op_bound}
		There exists some constant $c>0$, such that
		\begin{align}
		\label{bd:Q_IJ}\Big\|\hat{\cQ}_{\lambda,\epsilon}^{I,J}\Big\|_{\ell^q\to\ell^q}&\leq c,\\
		\label{bd:Q_*}\Big\|\hat{\cQ}_{\lambda,\epsilon}^{*,I}\Big\|_{\ell^q\to\ell^q}\leq c\epsilon^{-\frac{1}{q}},&\quad\quad\Big\|\hat{\cQ}_{\lambda,\epsilon}^{I,*}\Big\|_{\ell^q\to\ell^q}\leq c\epsilon^{-\frac{1}{p}},\\
		\label{bd:cU}\Big\|\hat{\bU}_{\hat{\lambda}\epsilon,\epsilon}^I\Big\|_{\ell^q\to\ell^q}&\leq\frac{c}{\log\hat{\lambda}},\quad\text{for}~|I|=3,\\
		\label{bd:cV}\Big\|\hat{\bV}_{\lambda,\epsilon}^I\Big\|_{\ell^q\to\ell^q}&\leq\frac{c}{(\log\frac{1}{\epsilon})^\frac{3}{4}},\quad\text{for}~|I|\leq2.
		\end{align}
	\end{proposition}
	
	We postpone the proof of Proposition \ref{prop:cg_op_bound} and first complete the proof of Theorem \ref{Thm:4mom}. By H\"{o}lder inequality, we can bound \eqref{eq:4mom3} by
	
	\begin{equation*}
	\begin{split}
	\Big|\cM_{\infty,\epsilon}^{\varphi,\psi}\Big|\leq ce^{\frac{2\lambda}{\epsilon}}\|\psi\|_\infty^4\epsilon^3\sum\limits_{r=1}^\infty C^r\sum\limits_{I_1,\cdots,I_r\atop|I_i|\leq 3, 1\leq i\leq r}&\Big\|\varphi_\epsilon^{\otimes4}\Big\|_{\ell^p}\Big\|\hat{\cP}_{\lambda,\epsilon}^{*,I_1}\Big\|_{\ell^q\to\ell^q}\Big\|\hat{\cP}_{\lambda,\epsilon}^{I_1,I_2}\Big\|_{\ell^q\to\ell^q}\\
	&\cdots\Big\|\hat{\cP}_{\lambda,\epsilon}^{I_{r-1},I_r}\Big\|_{\ell^q\to\ell^q}\Big\|\hat{\cQ}_{\lambda,\epsilon}^{I_r,*}\Big\|_{\ell^q\to\ell^q}\Big\|w_\epsilon^{\otimes4}\ind_{B_\epsilon}^{\otimes4}\Big\|_{\ell^q}.
	\end{split}
	\end{equation*}
	Recall the definition of $\hat{\cP}_{\lambda,\epsilon}^{I,J}$ from \eqref{def:cgop}. Since we will let $\epsilon\to0$, we have that
	\begin{equation*}
	\Big\|\hat{\cP}_{\lambda,\epsilon}^{I,J}\Big\|_{\ell^q\to\ell^q}\leq c\Big(\frac{1}{\log\hat{\lambda}}\vee\frac{1}{(\log\frac1\epsilon)^{\frac34}}\Big)=\frac{c}{\log\hat{\lambda}}.
	\end{equation*} 
	Hence, noting that the number of partitions of $\{1,2,3,4\}$ is finite, we then have that
	
	\begin{equation*}
	\Big|\cM_{\infty,\epsilon}^{\varphi,\psi}\Big|\leq ce^{\frac{2\lambda}{\epsilon}}\epsilon^2\|\psi\|_\infty^4\bigg\|\frac{\varphi_\epsilon}{w_\epsilon}\bigg\|_{\ell^p}^4\|w_\epsilon\ind_{B_\epsilon}\|_{\ell^q}^4\sum\limits_{r=1}^\infty \Big(\frac{C'}{\log\hat{\lambda}}\Big)^r\leq C_{\hat{\lambda}}\epsilon^{\frac2p}\|\psi\|_\infty^4\bigg\|\frac{\varphi_\epsilon}{w_\epsilon}\bigg\|_{\ell^p}^4\|w\ind_B\|^4_{\ell^q}
	\end{equation*}
	by choosing $\hat{\lambda}$ large enough.
\end{proof}

We now prove Proposition \ref{prop:cg_op_bound}, which is similar to Proposition \ref{prop:Opbounds}. Some extra work is needed to treat $w$. First note that by \eqref{Q_tran3}, we have that
\begin{equation}\label{Q_tran4}
\cQ_{\lambda,\epsilon}^{I,J}(\boldsymbol{x},\boldsymbol{y})\leq\sum\limits_{n=1}^{2/\epsilon}\prod\limits_{i=1}^4 g_n(y_i-x_i)\leq\begin{cases}
\displaystyle
\frac{C}{1+|\boldsymbol{x}-\boldsymbol{y}|^2},\quad&\text{for all}~\boldsymbol{x},\boldsymbol{y}\in(\bbZ)^{4},\\[16pt]
\displaystyle
C\epsilon e^{-\frac{\epsilon|\boldsymbol{x}-\boldsymbol{y}|^2}{C}},\quad&\text{for}~|\boldsymbol{x}-\boldsymbol{y}|\geq\frac{1}{\sqrt{\epsilon}}.
\end{cases}
\end{equation}

We split the space by $A_\epsilon:=\{|\boldsymbol{x}-\boldsymbol{y}|<C/\sqrt{\epsilon}\}$ and $A_\epsilon^c$. Note that $\log w$ is Lipschitz continuous. Hence, $|\log w_\epsilon(x)-\log w_\epsilon(y)|\leq C'\sqrt{\epsilon}|x-y|$, which implies that
\begin{equation}\label{eq:w}
\hat{\cQ}_{\lambda,\epsilon}^{I,J}(\boldsymbol{x},\boldsymbol{y}):=\frac{w_\epsilon^{\otimes4}(\boldsymbol{x})}{w_\epsilon^{\otimes4}(\boldsymbol{y})}\cQ_{\lambda,\epsilon}^{I,J}(\boldsymbol{x},\boldsymbol{y})\leq\begin{cases}
\displaystyle
\frac{Ce^{C_0}}{1+|\boldsymbol{x}-\boldsymbol{y}|^2}:=\frac{C}{1+|\boldsymbol{x}-\boldsymbol{y}|^2},&\text{on}~A_\epsilon,\\[16pt]
\displaystyle C\epsilon e^{-\frac{\epsilon|\boldsymbol{x}-\boldsymbol{y}|^2+C'\sqrt{\epsilon}|\boldsymbol{x}-\boldsymbol{y}|}{C}}:=C_1\epsilon e^{-C_1\epsilon|\boldsymbol{x}-\boldsymbol{y}|},~~&\text{on}~A_\epsilon^c.
\end{cases}
\end{equation}

To prove \eqref{bd:Q_IJ}, \eqref{bd:Q_*} and \eqref{bd:cU}, we work on $A_\epsilon$ and $A_\epsilon^c$ separately, and the computations are the same as those in the proof of Proposition \ref{prop:Opbounds}. Therefore, we omit the details. Note that in \eqref{eq:bigU}, $k$ transition kernels are equipped with $\bbE[\zeta^2]^k$, but in \eqref{mean:Ugroup}, $r-1$ heat kernels are equipped with $(\sigma_{\epsilon}^2)^r$. Hence an extra $\sigma_\epsilon^2$ is not needed in \eqref{bd:cU}, compared to \eqref{bound:U_I}.

\begin{proof}[\textbf{Proof for (\ref{bd:cV})}]
	We need to show that there exists some $C>0$, such that uniformly for any $\varphi\in\ell^p((\bbZ)_I^h)$ and $\psi\in\ell^q((\bbZ)_J^h)$,
	\begin{equation}\label{pf:4}
	\sum\limits_{\boldsymbol{x},\boldsymbol{y}\in(\bbZ)_I^4}\varphi(\boldsymbol{x})\hat{\bV}_{\infty,\epsilon}^I(\boldsymbol{x},\boldsymbol{y})\psi(\boldsymbol{y})\leq\frac{C}{\left(\log\frac{1}{\epsilon}\right)^{\frac{3}{4}}}\|\varphi\|_{\ell^p}\|\psi\|_{\ell^q}.
	\end{equation}
	Recall \eqref{bVlambda}, \eqref{def:bV} and \eqref{mean:limVgroup}. We know that
	\begin{equation*}
	\hat{\bV}_{\infty,\epsilon}^I(\boldsymbol{x},\boldsymbol{y})\leq\frac{C}{\big(\log\frac1\epsilon\big)^{\frac34}}\frac{w_\epsilon^{\otimes4}(\boldsymbol{x})}{w_\epsilon^{\otimes4}(\boldsymbol{y})}\prod\limits_{j\notin I(1)}\ind_{\{x_j=y_j\}}.
	\end{equation*}
	Then $\hat{\bV}_{\infty,\epsilon}^I(\boldsymbol{x},\boldsymbol{y})\neq 0$ if and only if $\boldsymbol{x}=\boldsymbol{y}$. By H\"{o}lder's inequality, \eqref{pf:4} is bounded above by
	\begin{equation*}
	\begin{split}
	\frac{C}{\left(\log\frac{1}{\epsilon}\right)^{\frac{3}{4}}}\Big(\sum\limits_{\boldsymbol{x},\boldsymbol{y}\in(\bbZ)_I^4}|\varphi(\boldsymbol{x})|^p\ind_{\{\boldsymbol{x}=\boldsymbol{y}\}}\Big)^{\frac1p}\Big(\sum\limits_{\boldsymbol{x},\boldsymbol{y}\in(\bbZ)_I^4}|\psi(\boldsymbol{y})|^p\ind_{\{\boldsymbol{x}=\boldsymbol{y}\}}\Big)^{\frac1q}\leq\frac{C}{\left(\log\frac{1}{\epsilon}\right)^{\frac{3}{4}}}\|\varphi\|_{\ell^p}\|\psi\|_{\ell^q}.
	\end{split}
	\end{equation*}
\end{proof}

\section{Proof of Theorem \ref{thm:2}}\label{S9}
We prove Theorem \ref{thm:2} by showing that for $\varphi\in C_c(\bbR)$ and $\psi\in C_b(\bbR)$, $\cZ_N=\tilde\cZ_N^{\beta_N}(\varphi,\psi)$ converges in distribution to a unique limit.

\begin{proof}[\textbf{Proof of the convergence of $\cZ_N$}] 
First note that by Proposition \ref{prop:mean_var}, the limit $\lim_{N\to\infty}\bbE[(\cZ_N)^2]$ exists. Hence, $(\cZ_N)_{N\geq1}$ are tight and have sequential weak limits. To show the limit is unique, it suffices to show that for any bounded $f:\bbR\to\bbR$ with uniformly bounded first three derivatives, the limit $\lim_{N\to\infty}\bbE[f(\cZ_N)]$ exists (\textit{cf.} \cite{B13}). We then show that $\bbE[f(\cZ_N)]$ is an Cauchy sequence.
	
	Again, by Proposition \ref{prop:mean_var} and its proof, we can see that given any $\epsilon>0$, we can truncate and mollify $\psi$ to get $\psi_\epsilon\in C_c(\bbR)$, such that by dominated convergence theorem,
	\begin{equation*}
	\limsup\limits_{N\to\infty}\big\|\cZ_N(\varphi,\psi)-\cZ_N(\varphi,\psi_\epsilon)\big\|_2\leq\epsilon.
	\end{equation*}
	Then by the Lipschitz continuity of $f$, we have that,
	\begin{equation*}
	|\bbE[f(\cZ_N(\varphi,\psi))]-\bbE[f(\cZ_N(\varphi,\psi_\epsilon))]|\leq\|f'\|_\infty\big\|\cZ_N(\varphi,\psi)-\cZ_N(\varphi,\psi_\epsilon)\big\|_2=o_\epsilon(1).
	\end{equation*}
	Hence, in the following, we only need to work on $\varphi,\psi\in C_c(\bbR)$.

	Note that we have shown that $\cZ_N$ can be approximated by coarse-grained model $\Lcg(\varphi,\psi|\Theta_{N,\epsilon})$ in $L^2$. By the Lipschitz continuity of $f$, we have that as $\epsilon\to0$,
	\begin{equation*}
	\begin{split}
	\big|\bbE[f(\cZ_N)]-\bbE[f(\Lcg(\varphi,\psi|\Theta_{N,\epsilon}))]\big|\leq \|f'\|_\infty\big\|\cZ_N-\Lcg(\varphi,\psi|\Theta_{N,\epsilon})\big\|_{L^2}=o_{\epsilon}(1).
	\end{split}
	\end{equation*}
	Hence, it suffices to show that
	\begin{equation}\label{eq:cglimit}
	\lim\limits_{\epsilon\to0}\limsup_{N\to\infty}\sup_{m,n\geq N}\big|\bbE\big[f(\Lcg(\varphi,\psi|\Theta_{n,\epsilon}))\big]-\bbE\big[f(\Lcg(\varphi,\psi|\Theta_{m,\epsilon}))\big]\big|=0.
	\end{equation}
	
	To lighten the notations, we write $\sL_\epsilon(\Theta_N)=\Lcg(\varphi,\psi|\Theta_{N,\epsilon})$. We will prove \eqref{eq:cglimit} by a generalized Lindeberg principle (\textit{cf.} \cite{MOO10,R13,CSZ21}). We first verify the local dependence Assumption \ref{assump:local} for $\Theta(\vec{\rmi})$. Recall that for indices
	\begin{equation*}
	\vec{\rmi}\in\bbT_\epsilon:=\{\vec{\rmi}=(\rmi,\rmi'): |\vec{\rmi}|=\rmi'-\rmi-1\leq K_\epsilon\},
	\end{equation*}
	its dependency neighborhood is
	\begin{equation}\label{def:neighbor}
	A_{\vec{\rmi}}=\{\vec{\rmj}=(\rmj,\rmj')\in\bbT_\epsilon: \{\rmi,\rmi'\}\cap\{\rmj,\rmj'\}\neq\emptyset\}.
	\end{equation}
	It is not hard to see that $|A_{\vec{\rmi}}|\leq  4K_\epsilon$. Besides, for $\vec{\rmi}\in\bbT_\epsilon$ and $\vec{\rmj}\in A_{\vec{\rmi}}$, the dependency neighborhood for $\vec{\rmi}$ and $\vec{\rmj}$ is denoted by $A_{\vec{\rmi},\vec{\rmj}}=A_{\vec{\rmi}}\cup A_{\vec{\rmj}}$.
	
	We now verify the condition \eqref{cond:lindeberg} in order to apply Lemma \ref{lem:A3}, that is,
	\begin{equation}\label{verifycond}
	\forall~\vec{\rmi}_1\in\bbT_\epsilon, \vec{\rmi}_2\in A_{\vec{\rmi}_1}, \vec{\rmi}_3\in A_{\vec{\rmi}_1,\vec{\rmi}_2},\quad\partial_{\vec{\rmi}_1\vec{\rmi}_2}^2\sL_\epsilon(\Theta)=\partial_{\vec{\rmi}_1\vec{\rmi}_3}^2\sL_\epsilon(\Theta)=\partial_{\vec{\rmi}_2\vec{\rmi}_3}^2\sL_\epsilon(\Theta)=0,
	\end{equation}
	where $\partial_{\vec{\rmi}}$ is the partial derivative with respect to $\Theta(\vec{\rmi})$. Note that $\sL_\epsilon(\Theta)$ is a polynomial of $\Theta$. We only need to verify that in each term of $\sL_\epsilon(\Theta)$, there is at most one $\Theta(\vec{\rmi}_k)$, $k=1,2,3$. Given $\Theta(\vec{\rmi}_1)$ and $\Theta(\vec{\rmi}_2)$ in $\sL_\epsilon(\Theta)$, we must have $\rmi_2>\rmi'_1$ or $\rmi_1>\rmi'_2$. However, by the $\vAnotriple$ condition, we must have $\rmi_2-\rmi'_1>K_\epsilon$ or $\rmi_1-\rmi'_2>K_\epsilon$. Hence, $\vec{\rmi}_2\not\in A_{\vec{\rmi}_1}$, that is, no terms contain the product $\Theta(\vec{\rmi}_1)\Theta(\vec{\rmi}_2)$ for $\vec{\rmi}_2\in A_{\vec{\rmi}_1}$. By the same argument, given $\Theta(\vec{\rmi}_1), \Theta(\vec{\rmi}_2)$ and $\Theta(\vec{\rmi}_3)$ in $\sL_\epsilon(\Theta)$, we must have $\vec{\rmi}_3\notin A_{\vec{\rmi}_1}$ and $\vec{\rmi}_3\notin A_{\vec{\rmi}_2}$. We have now verified \eqref{verifycond}.
	
	By Lemma \ref{lem:A1}-\ref{lem:A3}, we can bound \eqref{eq:cglimit} by
	\begin{equation}\label{eq:cglimit2}
	\big|\bbE\big[f(\sL_\epsilon(\Theta_m)\big]-\bbE\big[f(\sL_\epsilon(\Theta_n))\big]\big|\leq I_1^{(m)}+I_1^{(n)}+I_2^{(m)}+I_2^{(n)}+I_3^{(m,n)}.
	\end{equation}
	Here $I_1^{(m)}, I_1^{(n)}, I_2^{(m)}$ and $I_2^{(n)}$ comes from Lemma \ref{lem:A1} and $I_3^{(m,n)}$ come from Lemma \ref{lem:A2}. We have applied the lemmas to $X_k=\Theta_k$, $h(\cdot)=f(\sL_\epsilon(\cdot))$ and $Z_k=\Theta^{\rm(G)}_k$, where $Z_k$ is a family of centered Gaussian random variables with the same covariance structure as $\Theta_k$, for $k=m,n$. We will show that
	\begin{align}
	\label{limI1}\lim\limits_{\epsilon\to0}\limsup\limits_{n\to\infty}I_1^{(n)}&=0,\\
	\label{limI2}\lim\limits_{\epsilon\to0}\limsup\limits_{n\to\infty}I_2^{(n)}&=0,\\
	\label{limI3}\lim\limits_{\epsilon\to0}\limsup\limits_{n,m\to\infty}I_3^{(m,n)}&=0.
	\end{align}
	
	\vspace{0.25cm}
	\noindent\textbf{Proof for \eqref{limI1}.} By \eqref{eq:a4}, we have that
	\begin{equation}\label{limI1a}
	\begin{split}
	|I_1^{(n)}|\leq&\frac12\|f'''\|_\infty\sup_{\vec{\rmi}_1\in\bbT_\epsilon}\bbE\big[|\Theta_n(\vec{\rmi}_1)|^3\big]\sum\limits_{\vec{\rmi}_1\in\bbT_\epsilon,\vec{\rmi}_2\in A_{\vec{\rmi}_1}, \vec{\rmi}_3\in A_{\vec{\rmi}_1,\vec{\rmi}_2}}\\
	&\sup_{s,t,u}\bbE\Big[\big|\partial_{\vec{\rmi}_1}\sL_\epsilon(W_{s,t,u}^{\vec{\rmi}_1,\vec{\rmi}_2})\big|^3\Big]^{\frac13}\sup_{s,t,u}\bbE\Big[\big|\partial_{\vec{\rmi}_2}\sL_\epsilon(W_{s,t,u}^{\vec{\rmi}_1,\vec{\rmi}_2})\big|^3\Big]^{\frac13}\sup_{s,t,u}\bbE\Big[\big|\partial_{\vec{\rmi}_3}\sL_\epsilon(W_{s,t,u}^{\vec{\rmi}_1,\vec{\rmi}_2})\big|^3\Big]^{\frac13},
	\end{split}
	\end{equation}
	where for $s,t,u\in[0,1]$ and $\Theta_n^A(\vec{\rmi}):=\Theta_n(\vec{\rmi})\ind_{\{\vec{\rmi}\in A\}}$, 
	\begin{equation*}
	W_{s,t,u}^{\vec{\rmi}_1,\vec{\rmi}_2}:=su\sqrt{t}\Theta_n^{A_{\vec{\rmi}_1}}+u\sqrt{t}\Theta_n^{A_{\vec{\rmi}_1,\vec{\rmi}_2}\backslash A_{\vec{\rmi}_1}}+\sqrt{t}\Theta_n^{A_{\vec{\rmi}_1,\vec{\rmi}_2}^c}+\sqrt{1-t}\Theta_n^{\rm(G)}.
	\end{equation*}

	First note that there are finite many terms in $\sL_\epsilon(\Theta_n)$, since $\varphi,\psi\in C_c(\bbR)$. Hence, the number of indices $\vec{\rmi}$ is finite and we can pass $\lim_{n\to\infty}$ inside the summation. By Lemma \ref{lem:4mmtTheta} and the boundedness of $f'''$, we readily have that
	\begin{equation}\label{eq:4mmt}
	\limsup_{n\to\infty}\sup_{\vec{\rmi}_1\in\bbT_\epsilon}\|f'''\|_\infty\bbE\big[|\Theta_n(\vec{\rmi}_1)|^3\big]\leq C\limsup_{n\to\infty}\sup_{\vec{\rmi}_1\in\bbT_\epsilon}\bbE\big[|\Theta_n(\vec{\rmi}_1)|^4\big]^{\frac34}\leq\frac{C}{(\log\frac1\epsilon)^{\frac34}}.
	\end{equation}
	
	Then, by AM-GM inequality, the summation in \eqref{limI1a} is bounded by
	\begin{equation}\label{bd:sumI1}
	\sum\limits_{\vec{\rmi}_1\in\bbT_\epsilon,\vec{\rmi}_2\in A_{\vec{\rmi}_1}, \vec{\rmi}_3\in A_{\vec{\rmi}_1,\vec{\rmi}_2}}\frac13\sum\limits_{k=1}^3\sup_{s,t,u}\bbE\Big[\Big|\partial_{\vec{\rmi}_k}\sL_\epsilon\Big(W_{s,t,u}^{\vec{\rmi}_1,\vec{\rmi}_2}\Big)\Big|^4\Big]^{\frac34}.
	\end{equation}
	The next step is to compute the partial derivative of $\sL_\epsilon(W_{s,t,u}^{\vec{\rmi}_1,\vec{\rmi}_2})$.
	
	Note that we can write
	\begin{equation}\label{def:W}
	W_{s,t,u}^{\vec{\rmi}_1,\vec{\rmi}_2}(\vec{\rmi})=\begin{cases}
	su\sqrt{t}\Theta_n(\vec{\rmi})+\sqrt{1-t}\Theta_n^{\rm(G)}(\vec{\rmi}),\quad&\text{for}~\vec{\rmi}\in A_{\vec{\rmi}_1},\\
	u\sqrt{t}\Theta_n(\vec{\rmi})+\sqrt{1-t}\Theta_n^{\rm(G)}(\vec{\rmi}),\quad&\text{for}~\vec{\rmi}\in A_{\vec{\rmi}_1,\vec{\rmi}_2}\backslash A_{\vec{\rmi}_1},\\
	\sqrt{t}\Theta_n(\vec{\rmi})+\sqrt{1-t}\Theta_n^{\rm(G)}(\vec{\rmi}),\quad&\text{for}~\vec{\rmi}\in A_{\vec{\rmi}_1,\vec{\rmi}_2}^c.
	\end{cases}
	\end{equation}
	Hence, by recalling the definition of $\sL_\epsilon$ from \eqref{cgmodel}, we have that
	\begin{equation*}
	\begin{split}
	\sL_\epsilon\Big(W_{s,t,u}^{\vec{\rmj}_1,\vec{\rmj}_2}&\Big)= g_1(\varphi,\psi)+\sqrt{\epsilon}\sum\limits_{r=1}^{(\log\frac{1}{\epsilon})^2}\sum\limits_{(\vec{\rmi}_1,\cdots,\vec{\rmi}_r)\in\vAnotriple}\sum\limits_{\rma,\rmb\in\bbZ}\\
	&\varphi_\epsilon(\rma)g_{\rmi_1}(\rma)W_{s,t,u}^{\vec{\rmj}_1,\vec{\rmj}_2}(\vec{\rmi}_1)\times\bigg\{\prod_{j=2}^r g_{(\rmi_j-\rmi'_{j-1})}(0)W_{s,t,u}^{\vec{\rmj}_i,\vec{\rmj}_2}(\vec{\rmi}_j)\bigg\}g_{(\lfloor\frac{1}{\epsilon}\rfloor-\rmi'_r)}(\rmb)\psi_\epsilon(\rmb),
	\end{split}
	\end{equation*}
	where we use $\vec{\rmj}_1,\vec{\rmj}_2$ to avoid the notations abuse for $\vec{\rmi}_1, \vec{\rmi}_2$ in the summation. Suppose that we are differentiating $\sL_\epsilon$ with respect to $\Theta_n(\vec{\rmi})$ in the above expression with $\vec{\rmi}=\vec{\rmi}_k$ for some $1\leq k\leq r$. Then we only differentiate $
	W_{s,t,u}^{\vec{\rmj}_1,\vec{\rmj}_2}(\vec{\rmi})$ and get that
	\begin{equation}\label{diff:W}
	\partial_{\vec{\rmi}} W_{s,t,u}^{\vec{\rmj}_1,\vec{\rmj}_2}(\vec{\rmi})=\begin{cases}
	su\sqrt{t},\quad&\text{for}~\vec{\rmi}\in A_{\vec{\rmj}_1},\\
	u\sqrt{t},\quad&\text{for}~\vec{\rmi}\in A_{\vec{\rmj}_1,\vec{\rmj}_2}\backslash A_{\vec{\rmj}_1},\\
	\sqrt{t},\quad&\text{for}~\vec{\rmi}\in A_{\vec{\rmj}_1,\vec{\rmj}_2}^c.
	\end{cases}
	\end{equation}
	Then we can see that the consecutive product stops at $g_{(\rmi_k-\rmi'_{k-1})}$ and restarts at $g_{(\rmi_{k+1}-\rmi'_k)}$. Also, the constant term $g_1(\varphi,\psi)$ disappears. We find that
	\begin{equation}\label{sL:decouple}
	\partial_{\vec{\rmi}}\sL_\epsilon\Big(W_{s,t,u}^{\vec{\rmi}_1,\vec{\rmi}_2}\Big)=C(\vec{\rmi})\frac{1}{\sqrt{\epsilon}}\overline{\sL}_{\epsilon,[0,\rmi]}\Big(\varphi,\delta_0\Big|W_{s,t,u}^{\vec{\rmi}_1,\vec{\rmi}_2}\Big)\overline{\sL}_{\epsilon,[\rmi',\frac1\epsilon]}\Big(\delta_0,\psi\Big|W_{s,t,u}^{\vec{\rmi}_1,\vec{\rmi}_2}\Big),
	\end{equation}
	where $C(\rmi)$ is the constant determined by \eqref{diff:W}, $\delta_0$ is the Dirac measure, and $\overline{\sL}_{\epsilon,[t_1,t_2]}$ is a coarse-grained model without the deterministic term and having the same structure as $\sL_\epsilon$ with initial time $t_1$ and terminal time $t_2$. Here we need to multiply an $1/\sqrt{\epsilon}$ because we decompose one coarse-grained model by two coarse-grained models, and thus an extra $\sqrt{\epsilon}$ is needed by the definition of $\sL_\epsilon$.
	
	We then have
	\begin{equation}\label{bd:I1}
	\bbE\Big[\Big|\partial_{\vec{\rmi}}\sL_\epsilon\Big(W_{s,t,u}^{\vec{\rmi}_1,\vec{\rmi}_2}\Big)\Big|^4\Big]\leq\frac{1}{\epsilon^2}\bbE\Big[\Big|\overline{\sL}_{\epsilon,[0,\rmi]}\Big(\delta_0,\varphi|W_{s,t,u}^{\vec{\rmi}_1,\vec{\rmi}_2}\Big)\Big|^4\Big]\bbE\Big[\Big|\overline{\sL}_{\epsilon,[\rmi',\frac1\epsilon]}\Big(\delta_0,\psi|W_{s,t,u}^{\vec{\rmi}_1,\vec{\rmi}_2}\Big)\Big|^4\Big],
	\end{equation}
	where we have used $C(\rmi)\leq1$ by $s,u,t\in[0,1]$ and the time reversal invariance for $\overline{\sL}_{\epsilon}$ for the first expectation. Note that $\overline{\sL}_\epsilon(W_{s,t,u}^{\vec{\rmi}_1,\vec{\rmi}_2})$ has mean $0$ and we have the following analogue of Theorem \ref{Thm:4mom} for coarse-grained disorders $W_{s,t,u}^{\vec{\rmi}_1,\vec{\rmi}_2}$.
	\begin{lemma}\label{lem:bound_W}
		Let all settings be the same as in Theorem \ref{Thm:4mom} and recall the definition of $W_{s,t,u}^{\vec{\rmi},\vec{\rmj}}$ from \eqref{def:W}. We have that, there exists some constant $C\in(0,+\infty)$, such that for $w(x)=\exp(-|x|)$,
		\begin{equation*}
		\bbE\Big[\Big|\overline{\sL}_{\epsilon,[0,1]}\Big(\varphi,\psi|W_{s,t,u}^{\vec{\rmi},\vec{\rmj}}\Big)\Big|^4\Big]\leq C\epsilon^{\frac2p}\bigg\|\frac{\varphi_\epsilon}{w_\epsilon}\bigg\|^4_{\ell^p}\|\psi\|_\infty^4\|w\ind_B\|^4_{\ell^q},
		\end{equation*}
		uniformly for all $\rmi,\rmj\in\bbT_\epsilon$, $s,u,t\in[0,1]$ and $n$ large (note that $n$ comes from $\Theta_n$).
	\end{lemma}
	
	The proof of Lemma \ref{lem:bound_W} is similar to that of Theorem \ref{Thm:4mom}. We only need to obtain some moment bounds on $W_{s,t,u}^{\vec{\rmi},\vec{\rmj}}$. We postpone the proof of the lemma later. Note that in the proof for Theorem \ref{Thm:4mom}, we do not use the continuity of $\varphi$ and $\psi$ but only the boundedness of their values and supports (see Remark \ref{rmk:hmom}). Hence, the above bound is also valid for $\varphi(x)=\ind_0(x)$. Also note that by assuming $\bbE[\zeta_n^3]\geq0$, the above fourth moment is increasing as the time interval getting larger. By \eqref{def:neighbor}, \eqref{bd:I1} and Lemma \ref{lem:bound_W}, we have that
	\begin{equation*}
	\begin{split}
	&\sum\limits_{\vec{\rmi}_1\in\bbT_\epsilon,\vec{\rmi}_2\in A_{\vec{\rmi}_1}, \vec{\rmi}_3\in A_{\vec{\rmi}_1,\vec{\rmi}_2}}\frac13\sum\limits_{k=1}^3\sup_{s,t,u}\bbE\Big[\Big|\partial_{\vec{\rmi}_k}\sL_\epsilon\Big(W_{s,t,u}^{\vec{\rmi}_1,\vec{\rmi}_2}\Big)\Big|^4\Big]^{\frac34}\\
	\leq&\sum\limits_{\vec{\rmi}_1\in\bbT_\epsilon,\vec{\rmi}_2\in A_{\vec{\rmi}_1}, \vec{\rmi}_3\in A_{\vec{\rmi}_1,\vec{\rmi}_2}}C\bigg\|\frac{(\delta_0)_\epsilon}{w_\epsilon}\bigg\|^6_{\ell^p}\|\psi\|_\infty^6\|w\ind_B\|^6_{\ell^q}\epsilon^{\frac{3}{p}-\frac{3}{2}}\leq C_{\psi}\epsilon^{\frac3p-\frac52}(K_\epsilon)^2,
	\end{split}
	\end{equation*}
	where $C_\psi$ does not depend on $B$, since $w(x)=e^{-|x|}$ (see Remark \ref{rmk:hmom2}). Hence, by choosing $1<p<\frac65$, there exist some $C=C_{\psi,p}$ and $\gamma>0$ such that $C\epsilon^\gamma\geq\limsup_{n\to\infty}|I_1|$ by the choice of $K_\epsilon$ in \eqref{def:K}, which concludes \eqref{limI1}.
	
	\vspace{0.25cm}
	\noindent\textbf{Proof for \eqref{limI2}.} By \eqref{eq:a5}, we have that
	\begin{equation}\label{limI2a}
	\begin{split}
	|I_1^{(n)}|\leq&\frac12\|f'''\|_\infty\sup_{\vec{\rmi}_1\in\bbT_\epsilon}\bbE\big[|\Theta_n(\vec{\rmi}_1)|^3\big]\sum\limits_{\vec{\rmi}_1\in\bbT_\epsilon,\vec{\rmi}_2\in A_{\vec{\rmi}_1}, \vec{\rmi}_3\in A_{\vec{\rmi}_1,\vec{\rmi}_2}}\\
	&\sup_{t,u}\bbE\Big[\big|\partial_{\vec{\rmi}_1}\sL_\epsilon(W_{t,u}^{\vec{\rmi}_1,\vec{\rmi}_2})\big|^3\Big]^{\frac13}\sup_{t,u}\bbE\Big[\big|\partial_{\vec{\rmi}_2}\sL_\epsilon(W_{t,u}^{\vec{\rmi}_1,\vec{\rmi}_2})\big|^3\Big]^{\frac13}\sup_{t,u}\bbE\Big[\big|\partial_{\vec{\rmi}_3}\sL_\epsilon(W_{t,u}^{\vec{\rmi}_1,\vec{\rmi}_2})\big|^3\Big]^{\frac13},
	\end{split}
	\end{equation}
	where for $t,u\in[0,1]$,
	\begin{equation*}
	W_{t,u}^{\vec{\rmi}_1,\vec{\rmi}_2}:=u\sqrt{t}\Theta_n^{A_{\vec{\rmi}_1,\vec{\rmi}_2}}+\sqrt{t}\Theta_n^{A_{\vec{\rmi}_1,\vec{\rmi}_2}^c}+\sqrt{1-t}\Theta_n^{\rm(G)}.
	\end{equation*}
	
	The proof is exactly the same as that for \eqref{limI1}. We can find some constant $C\in(0,+\infty)$ and $\gamma>0$, such that $\limsup_{n\to\infty}|I_2|\leq C\epsilon^\gamma$, which concludes \eqref{limI2}. Note that $W_{t,u}^{\vec{\rmi}_1,\vec{\rmi}_2}$ and $W_{s,t,u}^{\vec{\rmi}_1,\vec{\rmi}_2}$ have the same structure. In particular, both of them can be expressed by $c_1\Theta_N+c_2\Theta_N^{\rm(G)}$ with $c_1,c_2\in[0,1]$ (see \eqref{def:W}). Hence, we can have a bound similar to Lemma \ref{lem:bound_W}, which is
	\begin{equation*}
	\bbE\Big[\Big|\overline{\sL}_{\epsilon,[0,1]}\Big(\varphi,\psi|W_{t,u}^{\vec{\rmi},\vec{\rmj}}\Big)\Big|^4\Big]\leq C\epsilon^{\frac2p}\bigg\|\frac{\varphi_\epsilon}{w_\epsilon}\bigg\|^4_{\ell^p}\|\psi\|_\infty^4\|w\ind_B\|^4_{\ell^q},
	\end{equation*}
	and we omit its proof.
	
	\vspace{0.25cm}
	\noindent\textbf{Proof for \eqref{limI3}.} By \eqref{eq:a6} and note that $(\Theta_n({\vec{\rmi}}))_{\vec{\rmi}\in\bbT_\epsilon}$ are uncorrelated, we have that
	\begin{equation}\label{limI3a}
	|I_3|\leq\frac12\|f''\|_\infty\sum\limits_{\vec{\rmi}\in\bbT_\epsilon}\Big(\bbE\big[\Theta_n^2(\vec{\rmi})\big]-\bbE\big[\Theta_m^2(\vec{\rmi})\big]\Big)\sup\limits_{t\in[0,1]}\bbE\big[|\partial_{\vec{\rmi}}\sL_\epsilon(W_t)|^2\big],
	\end{equation}
	where for $t\in[0,1]$, $W_t:=\sqrt{t}\Theta_n^{\rm(G)}+\sqrt{1-t}\Theta_m^{\rm(G)}$.
	
	First note that by Lemma \ref{lem:varTheta}, for any fixed $\epsilon>0$, $\lim_{m,n\to\infty}\big(\bbE[\Theta_n^2(\vec{\rmi})]-\bbE[\Theta_m^2(\vec{\rmi})]\big)=0$. Then note that $\partial_{\vec{\rmi}}\sL_\epsilon(W_t)$ is a multilinear polynomial of $W_t(\vec{\rmi})$ (see the discussion for \eqref{sL:decouple}). Hence, $\bbE[|\partial_{\vec{\rmi}}\sL_\epsilon(W_t)|^2]$ is a multilinear polynomial of $\bbE[W_t^2(\vec{\rmi})]$, which is a linear combination of $\bbE[\Theta_n^2(\vec{\rmi})]$ and $\bbE[\Theta_m^2(\vec{\rmi})]$ (we can assume that $\Theta_n^{\rm(G)}$ and $\Theta_m^{\rm(G)}$ are independent). Again by Lemma \ref{lem:varTheta} and noting that there are finite many $\bbE[W_t^2(\vec{\rmi})]$ for fixed $\epsilon$, we have shown that $\lim_{n,m\to\infty}I_3^{(m,n)}=0$.
	
	Combine \eqref{limI1}-\eqref{limI3} together and we have shown that $\cZ_N$ converges to a unique limit.
\end{proof}

\begin{remark}\label{rml:fdd}
	We can actually extend Theorem \ref{thm:2} to a convergence in finite-dimensional distribution. For any fixed positive integer $k$, Let $-\infty<s_i<t_i<+\infty$ and $\varphi_i,\psi_i\in C_c(\bbR)$ for $i=1,2,\cdots,k$. Each $\cZ_{\lfloor s_i N\rfloor, \lfloor t_i N\rfloor}^{\beta_N}(\varphi_i,\psi_i)$ can be approximated in $L_2$ by a coarse-grained model $\sL_{\epsilon,[s_i,t_i]}(\varphi_i,\psi_i|\Theta_{N,\epsilon})$. For the random vector $(\cZ_{\lfloor s_i N\rfloor, \lfloor t_i N\rfloor}^{\beta_N}(\varphi_i,\psi_i))_{i=1,\cdots,k}$, we then have an $L_2$ approximation in vector level for fixed $k$. For any bounded test function $f:\bbR^k\to\bbR$ with bounded derivatives up to third order. The Lindeberg principle also holds for $f(\sL_{\epsilon, [s_1,t_1]}(\cdot),\cdots,\sL_{\epsilon, [s_k,t_k]}(\cdot))$ (see remark \ref{rmk:vector}). The proof for single $\cZ_N$ can be applied here without any change.
\end{remark}

We are left to prove Lemma \ref{lem:bound_W}.
\begin{proof}[Proof of Lemma \ref{lem:bound_W}]
	We write
	\begin{equation}\label{W4mmt}
	\begin{split}
	\bM_{n,\epsilon}^{\varphi,\psi}:&=\bbE\Big[\Big(\overline{\sL}_{\epsilon,[0,1]}\Big(\varphi,\psi|W_{s,t,u}^{\vec{\rmj}_1,\vec{\rmj}_2}\Big)\Big)^4\Big]=\epsilon^2\bbE\bigg[\bigg(\sum\limits_{r=1}^{(\log\frac{1}{\epsilon})^2}\sum\limits_{(\vec{\rmi}_1,\cdots,\vec{\rmi}_r)\in\vAnotriple}\sum\limits_{\rma,\rmb\in\bbZ}\\
	&\quad\quad\varphi_\epsilon(\rma)g_{\rmi_1}(\rma)W_{s,t,u}^{\vec{\rmj}_1,\vec{\rmj}_2}(\vec{\rmi}_1)\times\Big\{\prod_{j=2}^r g_{(\rmi_j-\rmi'_{j-1})}(0)W_{s,t,u}^{\vec{\rmj}_1,\vec{\rmj}_2}(\vec{\rmi}_j)\Big\}g_{(\lfloor\frac{1}{\epsilon}\rfloor-\rmi'_r)}(\rmb)\psi_\epsilon(\rmb)\bigg)^4\bigg].
	\end{split}
	\end{equation}	
	
	Recall the definition of $W_{s,t,u}^{\vec{\rmj}_1,\vec{\rmj}_2}$ from \eqref{def:W}. We can write $W_{s,t,u}^{\vec{\rmj}_1,\vec{\rmj}_2}(\vec{\rmi})=c_1\Theta_n(\vec{\rmi})+c_2\Theta_n^{\rm(G)}(\vec{\rmi})$, where $c_1$ and $c_2$ are determined by $\vec{\rmi}=(\rmi,\rmi')$.
	
	We can adapt the same indices rearrangement procedure in the proof of Theorem \ref{Thm:4mom}. Since the centered Gaussian family $\Theta^{\rm(G)}$ has the same covariance structure as $\Theta$, and since $\Theta^{\rm(G)}$ and $\Theta$ are independent, we have that for $\rmi\neq\rmj$, $\bbE[W_{s,t,u}^{\vec{\rmj}_1,\vec{\rmj}_2}(\vec{\rmi})W_{s,t,u}^{\vec{\rmj}_1,\vec{\rmj}_2}(\vec{\rmj})]=0$. Hence, the argument in ``\textbf{summation rearrangement}'' part in the proof of Theorem \ref{Thm:4mom} can be transferred here without change.
	
Then we need to check that the decoupling procedure, that is, the step (A) in ``\textbf{adjustments on coarse-grained disorders and heat kernels}'' part in the proof of Theorem \ref{Thm:4mom} is still valid. We need to check that for $|\vec{\rmi}|\geq2$, the second to fourth moments of $W_{s,t,u}^{\vec{\rmj}_1,\vec{\rmj}_2}(\vec{\rmi})$ can be controlled by the corresponding moments of $W_{s,t,u}^{\vec{\rmj}_1,\vec{\rmj}_2}((\rmi,\rmi))g_{\rmi'-\rmi}(0)W_{s,t,u}^{\vec{\rmj}_1,\vec{\rmj}_2}((\rmi',\rmi'))$.

	To lighten the notations, we denote
	\begin{equation*}
	\cE_W^k(\vec{\rmi}):=\bbE\Big[\Big(W_{s,t,u}^{\vec{\rmj}_1,\vec{\rmj}_2}(\vec{\rmi})\Big)^k\Big],\quad\cE_\Theta^k(\vec{\rmi}):=\bbE\Big[\big(\Theta_n(\vec{\rmi})\big)^k\Big],\quad\cE_G^k(\vec{\rmi}):=\bbE\Big[(\Theta_n^{\rm(G)}(\vec{\rmi}))^k\Big].
	\end{equation*}
	We only need to treat $k=2,3,4$. Note that since $\Theta^{\rm(G)}$ is Gaussian and has the same covariance structure as $\Theta$, we have that
	\begin{equation*}
	\cE_G^1(\vec{\rmi})=0,\quad\cE_G^2(\vec{\rmi})=\cE_\Theta^2(\vec{\rmi}),\quad\cE_G^3(\vec{\rmi})=0,\quad\cE_G^4(\vec{\rmi})=3(\cE_G^2(\vec{\rmi}))^2=3(\cE_\Theta^2(\vec{\rmi}))^2.
	\end{equation*}
	
	For $k=2,4$, by the proof of Lemma \ref{lem:varTheta}, we have that
	\begin{equation*}
	\cE_W^k(\vec{\rmi})\leq C\cE_W^k((\rmi,\rmi))g_{\rmi'-\rmi}^k(0)\cE_W^k((\rmi',\rmi'))=\bbE\bigg[\bigg(W_{s,t,u}^{\vec{\rmj}_1,\vec{\rmj}_2}((\rmi,\rmi))g_{\rmi'-\rmi}(0)W_{s,t,u}^{\vec{\rmj}_1,\vec{\rmj}_2}((\rmi',\rmi'))\bigg)^k\bigg].
	\end{equation*}
	For $k=3$, we have that $\cE_W^3(\vec(\rmi))=c_1^3\cE_\Theta^3(\vec{\rmi})$ and
	\begin{equation*}
	\begin{split}
	\bbE\bigg[\bigg(W_{s,t,u}^{\vec{\rmj}_1,\vec{\rmj}_2}((\rmi,\rmi))g_{\rmi'-\rmi}(0)W_{s,t,u}^{\vec{\rmj}_1,\vec{\rmj}_2}((\rmi',\rmi'))\bigg)^3\bigg]&=c_1^6\cE_W^3((\rmi,\rmi))g_{\rmi'-\rmi}^3(0)\cE_W^3((\rmi',\rmi'))\\
	&=c_1^6\cE_\Theta^3((\rmi,\rmi))g_{\rmi'-\rmi}^3(0)\cE_\Theta^3((\rmi',\rmi')).
	\end{split}
	\end{equation*}
	In the proof of Theorem \ref{Thm:4mom}, we have explained that $|\cE_\Theta^3(\vec{\rmi})|\leq C|\cE_\Theta^3((\rmi,\rmi))g_{\rmi'-\rmi}^3(0)\cE_\Theta^3((\rmi',\rmi'))|$, and thus, the decoupling procedure can also be transferred here.

	Finally, by the above treatment, the upper bound of \eqref{W4mmt} has the same structure as $\Lcg(\Theta_n)$. Proposition \ref{prop:cg_op_bound} also holds given that we have the same second and fourth moments bounds on $W_{s,t,u}^{\vec{\rmj}_1,\vec{\rmj}_2}(\vec{\rmi})$. We are left to obtain analogues of Lemma \ref{lem:varTheta} and Lemma \ref{lem:4mmtTheta}.
	
	For $\cE_W^2(\vec{\rmi})$, we have that
	\begin{equation*}
	\cE_W^2(\vec{\rmi})=c_1^2\cE_\Theta^2(\vec{\rmi})+c_2^2\cE_G^2(\vec{\rmi})=(c_1^2+c_2^2)\cE_\Theta^2(\vec{\rmi}).
	\end{equation*}
	For $\cE_W^4(\vec{\rmi})$, we have that by Jensen's inequality,
	\begin{equation*}
	\cE_W^4(\vec{\rmi})=c_1^4\cE_\Theta^4(\vec{\rmi})+6c_1^2 c_2^2\cE_\Theta^2(\vec{\rmi})\cE_G^2(\vec{\rmi})+c_2^4\cE_G^4(\vec{\rmi})\leq(c_1^4+6c_1^2 c_2^2+3 c_2^4)\cE_\Theta^4(\vec{\rmi}).
	\end{equation*}
	Note that $\cE_\Theta^2(\vec{\rmi})$ and $\cE_\Theta^4(\vec{\rmi})$ have been bounded by Lemma \ref{lem:varTheta} and Lemma \ref{lem:4mmtTheta}.
	
	Since $c_1,c_2$ can be $su\sqrt{t},u\sqrt{t}$ or $\sqrt{t}$ for $s,u,t\in[0,1]$, all constants here are bounded and the proof is completed.
\end{proof}

\section{Proof for Theorem \ref{thm:2+}}\label{S:sde}
For simplicity, we only consider the measure $u^{\delta}(\dd x):=\int_{y\in\bbR} u^{\delta}_{0,1}(\dd x,\dd y)$. We write
\begin{equation}\label{eq:spde_integral}
u^{\delta}[f]:=\iint_{\bbR^2} f(x)u^{\delta}_{0,1}(\dd x,\dd y).
\end{equation}
We show that $u^{\delta}[f]$ can be approximated by a coarse-grained model $\Lcg(f,\ind_\bbR)$ (similar to \eqref{cgmodel}) that will be specified later. Then the proof is completed by the argument in Section \ref{S9}. To proceed, we will first introduce some preliminaries in the continuous setting and compute the second moment of $u^{\delta}[f]$.

\subsection{Preliminaries}\label{sec:pre_she}
By a Feymann-Kac representation and the definition of the Wick exponential, we can rewrite \eqref{eq:spde_integral} by a Wiener chaos expansion
\begin{equation}\label{eq:wick_exp}
\begin{split} 
u^{\delta}[f]=&\int_\bbR f(x)\dd x+\sum\limits_{k=1}^\infty \beta_\delta^k\idotsint\limits_{0<t_1<\cdots<t_k<\delta^{-2}\atop x_0,\cdots,x_k\in\bbR}f(x_0)g_{t_1}\big(x_1-\frac{x_0}{\delta}\big)\rho(x_1)\dd x_0\dd x_1\dd W_{t_1}\\
&\times\prod\limits_{i=2}^k g_{t_i-t_{i-1}}(x_i-x_{i-1})\rho(x_i)\dd x_i\dd W_{t_i},
\end{split}
\end{equation}
where $g_t(x)$ is the heat kernel (see \eqref{heatkernel} and d$W_t$ is a 1-dimensional white noise in time. 

Since $\dd W_t$ is centered, it directly follows that $\bbE\big[u^{\delta}[f]\big]=\int_\bbR f(x)\dd x.$ By a change of variable $x_0\to\delta x_0$, \eqref{eq:wick_exp} equals to
\begin{equation}\label{eq:wick_exp2}
\begin{split}
u^{\delta}[f]=&\int_\bbR f(x)\dd x+\delta\sum\limits_{k=1}^\infty \beta_\delta^k\idotsint\limits_{0<t_1<\cdots<t_k<\delta^{-2}\atop x_0,\cdots,x_k\in\bbR}f(\delta x_0)\dd x_0\prod\limits_{i=1}^k g_{t_i-t_{i-1}}(x_i-x_{i-1})\rho(x_i)\dd x_i\dd W_{t_i},
\end{split}
\end{equation}
where $t_0:=0$ by convention. Hence,
\begin{equation}\label{eq:square_wick_exp}
\begin{split}
\bbE\big[(u^{\delta}[f])^2\big]=&\Big(\int_\bbR f(x)\dd x\Big)^2+\delta^2\sum\limits_{k=1}^{\infty}\beta_\delta^{2k}\idotsint\limits_{0<t_1<\cdots<t_k<\delta^{-2}\atop x_0,\cdot,x_k,y_0,\cdots,y_k\in\bbR}f(\delta x_0)f(\delta y_0)\dd x_0\dd y_0\\
&\times\prod\limits_{i=1}^k g_{t_i-t_{i-1}}(x_i-x_{i-1})g_{t_i-t_{i-1}}(y_i-y_{i-1})\rho(x_i)\rho(y_i)\dd x_i\dd y_i\dd t_i.
\end{split}
\end{equation}

We only need to analyze the summation above. To have a continuous analogue of $u(\cdot)$ (see \eqref{def:u}), by noting that $\rho(\cdot)$ has a compact support, we introduce
\begin{equation*}
\tilde r(t):=\iint_{\bbR^2}g_t(x-x')g_t(y-y')\rho(x)\rho(y)\dd x\dd y\begin{cases}
\leq\|\rho\|_\infty^2,&\quad\forall~t>0,\\
=\frac{1}{2\pi t}+O(t^{-2}),&\quad\text{as}~t\to+\infty.
\end{cases}
\end{equation*}

We first cite some known facts about $\tilde r(t)$ from \cite[Section 8.1]{CSZ19b}, which will be intensively used in the following computations. Let $r(t):[0,+\infty)\to(0,+\infty)$ such that
\begin{equation}\label{eq:contin_renew_kernel}
r(t)=\frac{1}{2\pi t}(1+o(1)),\quad\text{as}~t\to+\infty. 
\end{equation}
Given $\delta>0$, let $\{\cT_i^\delta\}_{i\geq1}$ be i.i.d.\ random variables with density
\begin{equation}\label{def:renew_T}
\bP(\cT_i^{\delta}\in\dd t)=\frac{r(t)}{\cR_\delta}\ind_{[0,\delta^{-2}]}(t)\dd t,\quad\text{where}~\cR_\delta=\int_0^{\delta^{-2}}r(t)\dd t
\end{equation}
is the normalization constant. It is not hard to see that $\cT_i^{\delta}$ is a continuous analogue of $\iota^{(N)}_1$ (see \eqref{new_renew}). Note that the following properties hold for generic $r(t)$, that is, the exact value of $o(1)$ and the value of $r(t)$ in a neighborhood of $0$ is not important:
\begin{itemize}
	\item  There exists some $c\in(0,1)$, such that
	\begin{equation}\label{eq:tail_renew_T}
	\bP(\cT_1^\delta+\cdots+\cT_{\lfloor s\log\delta^{-2}\rfloor}^\delta\leq\delta^{-2})\leq e^{s-c\log s},\quad\forall~\delta\in(0,1), \forall~s\in[0,+\infty).
	\end{equation}
	\item Given $\vartheta\in\bbR$, let
	\begin{equation*}
	\lambda_\delta:=1+\frac{\vartheta}{\log\delta^{-2}}(1+o(1)).
	\end{equation*}
	Then for any $T\in[0,1]$,
	\begin{equation}\label{eq:reimann_renew_T}
	\lim\limits_{\delta\to0}\frac{1}{\log\delta^{-2}}\sum\limits_{k=1}^\infty\lambda_\delta^k\bP(\cT_1^\delta+\cdots+\cT_k^\delta\leq\delta^{-2}T)=\int_0^\infty e^{\vartheta u}\bP(Y_u\leq T)\dd u,
	\end{equation}
	where $(Y_s)_{s\geq0}$ is the Dickman subordinator with density $f_s(t)$ (see \eqref{Gtheta}).
\end{itemize}

Note that by \eqref{eq:tail_renew_T}, the main contribution in \eqref{eq:reimann_renew_T} comes from $k\leq A(\log\delta^{-1})$ for large $A$. Moreover, we can further restrict the summation to ``no consecutive increments less than $(\log\delta^{-1})^2$'', since by a union bound and \eqref{def:renew_T},
\begin{equation}\label{eq:no_consec_short}
\begin{split}
&\bP\Big(\bigcup\limits_{k=1}^{A\log\delta^{-1}}\big\{\cT_k^{\delta}\leq(\log\delta^{-1})^2, \cT_{k+1}^\delta\leq(\log\delta^{-1})^2\big\}\Big)\\
\leq& A\log\delta^{-1}\bP\big(\cT_1^\delta\leq(\log\delta^{-1})^2\big)^2\leq C_A\frac{(\log\log\delta^{-1})^2}{\log\delta^{-1}}.
\end{split}
\end{equation} 

We now compute the limit of the second moment of $u^{\delta}[f]$ as $\delta\to0$.

\subsection{Second moment of $u^{\delta}[f]$}\label{sec:she_L2}
We only need to compute the limit for the summation on the right-hand side of \eqref{eq:square_wick_exp}. Our approach is an adaption of \cite[Section 8]{CSZ19b}.

We can write the summation as
\begin{equation}\label{eq:square_wick_exp2}
\delta^2\iint_{\bbR^2}f(\delta x)f(\delta y)K^\delta(x,y)\dd x\dd y,
\end{equation}
where
\begin{equation}\label{eq:K_delta}
\begin{split}
&K^\delta(x,y):=\\
\sum\limits_{k=1}^\infty\beta_\delta^{2k}\idotsint\limits_{0<t_1<\cdots<t_k<\delta^{-2}\atop x_1,\cdots,x_k,y_1,\cdots,y_k\in\bbR}&\prod\limits_{i=1}^k g_{t_i-t_{i-1}}(x_i-x_{i-1})g_{t_i-t_{i-1}}(y_i-y_{i-1})\rho(x_i)\rho(y_i)\dd x_i\dd y_i\dd t_i
\end{split}
\end{equation}
with the convention $x_0:=x$, $y_0:=y$. By a change of variable $t_1=\delta^{-2}t$ and the translation invariance of $g_t(x)$, we have that
\begin{equation}\label{eq:K_delta2}
K^\delta(x,y)=\int_0^1\dd t\int_{\bbR^2}g_t(\delta(x'-x))g_t(\delta(y'-y))\rho(x')\rho(y')K_{1-t}^\delta(x',y')\dd x'\dd y',
\end{equation}
where
\begin{equation}\label{eq:K_delta_t}
\begin{split}
&K^{\delta}_{1-t}(x',y'):= \beta_\delta^2+\sum\limits_{k=1}^\infty\beta_\delta^{2(k+1)}\\
\times\idotsint\limits_{0<t_1<\cdots<t_k<\delta^{-2}(1-t)\atop x_1,\cdots,x_k,y_1,\cdots,y_k\in\bbR}&\prod\limits_{i=1}^k g_{t_i-t_{i-1}}(x_i-x_{i-1})g_{t_i-t_{i-1}}(y_i-y_{i-1})\rho(x_i)\rho(y_i)\dd x_i\dd y_i\dd t_i
\end{split}
\end{equation}
with the convention $x_0:=x'$, $y_0:=y'$. We will see that \eqref{eq:K_delta_t} plays the role of $\overline{U}_N(n)$ (see \eqref{def:U_bar}) in the discrete setting with $N=\lfloor\delta^{-2}\rfloor$ and $n=\lfloor t\rfloor$.

Now with $\bar{\cR}_\delta=\int_0^{\delta^{-2}}\bar{r}(t)\dd t$, we write
\begin{equation*}
\bar{r}(t)=\sup_{x',y'\in\text{supp}(\rho)}\iint_{\bbR^2}g_t(x-x')g_t(y-y')\rho(x)\rho(y)\dd x\dd y=\frac{1}{2\pi t}(1+o(1)),
\end{equation*}
and note that $\beta_\delta^2\bar{\cR}_\delta\leq 1+\frac{c}{\log\delta^{-2}}$, since $\beta_\delta^2=\frac{\pi}{\log\delta^{-1}}+\frac{\theta+o(1)}{(\log\delta^{-1})^2}$. Furthermore, by \eqref{eq:tail_renew_T}-\eqref{eq:no_consec_short}, for $t,x',y'$ uniformly, we can restrict the summation in \eqref{eq:K_delta_t} to $k\leq A\log\delta^{-1}$ for $A$ sufficiently large and restrict the domain of integrals there to 
\begin{equation}\label{set:J}
\begin{split}
\cJ^{\delta,k}_{1-t}:=\{0<t_1&<\cdots<t_k<\delta^{-2}(1-t): t_1>(\log\delta^{-1})^2, \text{and}~\forall~1\leq i\leq r-1,\\
&\text{either}~t_i-t_{i-1}>(\log\delta^{-1})^2~\text{or}~t_{i+1}-t_i>(\log\delta^{-1})^2\}.
\end{split}
\end{equation}

We denote \eqref{eq:K_delta_t} by $K_{1-t,\cJ}^{\delta,A}(x',y')$ after the summation and integrals are restricted. Let
\begin{equation}\label{def:r(t)}
r(t):=\idotsint_{\bbR^4}\rho(x')\rho(y')g_t(x-x')g_t(y-y')\rho(x)\rho(y)\dd x\dd x'\dd y\dd y'=\frac{1}{2\pi t}(1+o(1)).
\end{equation}
We can replace $g_{t_i-t_{i-1}}(x_i-x_{i-1})g_{t_i-t_{i-1}}(y_i-y_{i-1})$ by $r(t)$ when $t_i-t_{i-1}>(\log\delta^{-1})^2$, where an extra factor $\exp(O(\log\delta^{-1})^{-2})$ appears. After all these replacements, there is a factor $\exp(O(\log\delta^{-1})^{-1})$, which converges to $1$ as $\delta\to0$.

We then treat the integrals in $K_{1-t,\cJ}^{\delta,A}$ in the following manners:
\begin{itemize}
	\item If $t_i-t_{i-1}>(\log\delta^{-1})^2$, then we replace $g_{t_i-t_{i-1}}(x_i-x_{i-1})g_{t_i-t_{i-1}}(y_i-y_{i-1})$ by $r(t_i-t_{i-1})$ and integrate $x_i$ and $y_i$, which simply gives $\iint_{\bbR^2}\rho(x_i)\rho(y_i)\dd x_i\dd y_i=1$.
	\item If $t_i-t_{i-1}\leq(\log\delta^{-1})^2$, then we integrate $x_i,x_{i-1},y_i,y_{i-1}$ as \eqref{def:r(t)}, which gives a $r(t_i-t_{i-1})$. 
\end{itemize}
Note that since there are no consecutive $t_i-t_{i-1}$ and $t_{i+1}-t_i$ smaller than $(\log\delta^{-1})^2$ and we have forced $t_1>(\log\delta^{-1})^2$, there is no $x_i$ or $y_i$ being integrated twice. 

Finally, we have that $K_{1-t}^{\delta,A}(x',y')=(1+o(1))\hat{K}_{1-t}^{\delta}+o(1)$ as $\delta\to0$ and $A\to\infty$, where
\begin{equation}\label{eq:K_hat}
\begin{split}
\hat{K}_{1-t}^{\delta}&=\sum\limits_{k=0}^\infty\beta_\delta^{2(k+1)}\idotsint\limits_{0<t_1<\cdots<t_k<\delta^{-2}(1-t)}\prod\limits_{i=1}^k r(t_i-t_{i-1})\dd t_i\\
&=\frac{2\pi+o(1)}{\log\delta^{-2}}\sum\limits_{k=0}^\infty(\beta_\delta^2\cR_\delta)^k\bP(\cT_1^\delta+\cdots+\cT_k^\delta\leq\delta^{-2}(1-t)),
\end{split}
\end{equation}
with $T_i^\delta$ and $\cR_\delta$ defined by \eqref{def:renew_T}.

Suppose that $\beta_\delta^2\cR_\delta\sim1+\frac{\vartheta}{\log\delta^{-2}}$. Then by \eqref{eq:reimann_renew_T} and \eqref{Gtheta}, we have that
\begin{equation}\label{eq:limit_K_hat}
\lim\limits_{\delta\to0}\hat{K}_{1-t}^\delta\!=\!2\pi\int_0^\infty\!\! e^{\vartheta u}\bP(Y_u\leq1\!-\!t)\dd u\!=\!2\pi\int_0^\infty\!\! e^{\vartheta u}\dd u\int_0^{1-t}\!\!f_u(r)\dd r=2\pi\int_0^{1-t}\!\! G_{\vartheta}(r)\dd r.
\end{equation}
Hence, by a change of variables $\delta x=x_0$, $\delta y=y_0$ and $r=s-t$, the second moment of $u^\delta[f]$ equals to $(\int_\bbR f(x)\dd x)^2$ plus $(1+o(1))2\pi$ multiplying by
\begin{equation}\label{contin_L2}
\begin{split}
\idotsint_{\bbR^4}&f(x_0)f(y_0)g_t(\delta x'-x_0)g_t(\delta y'-y_0)\rho(x')\rho(y')\dd x_0\dd y_0\dd x'\dd y'\!\!\!\!\iint\limits_{0<t<s<1}\!\!G_\vartheta(s-t)\dd s\dd t\\
&\longrightarrow\idotsint\limits_{0<t<s<1\atop x,y\in\bbR}f(x)f(y)g_t(x)g_t(y)G_{\vartheta}(s-t)\dd s\dd t\dd x\dd y,\quad\text{as}~\delta\to0.
\end{split}
\end{equation}


We now determine $\vartheta$ and then the second moment computation is completed. Note that
\begin{equation}
\cR_\delta=\idotsint_{\bbR^4}\rho(x')\rho(y')\rho(x)\rho(y)\dd x'\dd y'\dd x\dd y\int_0^{\delta^{-2}}g_t(x-x')g_t(y-y')\dd t.
\end{equation}
We have that by the end of \cite[Section 8.2]{CSZ19b},
\begin{equation*}
\int_0^{\delta^{-2}}g_t(x-x')g_t(y-y')\dd t=\frac{1}{2\pi}\Big[\log\Big(1+\frac{2}{\delta^2[(x-x')^2+(y-y')^2]}\Big)-\gamma+o(1)\Big],
\end{equation*}
and thus by $\log(1+x)\sim\log x$ as $x\to\infty$,
\begin{equation*}
\begin{split}
R_\delta=\frac{1}{2\pi}\Big(&\log\delta^{-2}+\log2-\gamma+\\
&\idotsint_{\bbR^4}\!\rho(x')\rho(y')\log\!\Big[\frac{1}{(x-x')^2+(y-y')^2}\Big]\rho(x)\rho(y)\dd x'\dd y'\dd x\dd y+o(1)\Big).
\end{split}
\end{equation*}
Hence, by $\beta_\delta^2=\frac{2\pi}{\log\delta^{-2}}+\frac{\theta+o(1)}{(\log\delta^{-2})^2}$ and $\beta_\delta^2\cR_\delta=1+\frac{\vartheta+o(1)}{\log\delta^{-2}}$, we have that

\begin{equation}\label{contin_theta}
\vartheta=\log2-\gamma+\idotsint_{\bbR^4}\!\!\rho(x')\rho(y')\log\Big[\frac{1}{(x-x')^2+(y-y')^2}\Big]\rho(x)\rho(y)\dd x'\dd y'\dd x\dd y+\frac{\theta}{2\pi}.
\end{equation}

\subsection{Approximation by the coarse-grained model}
We will approximate $u^\delta[f]$ by a coarse-grained model $\Lcg(f,\ind_{\bbR}|\Theta)$ (see \eqref{cgmodel}) with $\Theta$ being determined later in \eqref{def:she_theta}.

We partition the time interval $[0,\delta^{-2}]$ into mesoscopic intervals
\begin{equation}\label{def:meso}
\cI_{\delta,\epsilon}(\rmi):=[(\rmi-1)\epsilon\delta^{-2},\epsilon\delta^{-2}],\quad\rmi=1,\cdots,\frac{1}{\epsilon}.
\end{equation}
Then rearrange \eqref{eq:wick_exp2} according to the transitions in $\cI_{\delta,\epsilon}(\rmi)$ and we have that
\begin{equation}\label{def:meso_u}
\begin{split}
u^{\delta}[f]=&\int_\bbR f(x)\dd x+\delta\sum\limits_{k=1}^{1/\epsilon}\sum\limits_{1\leq\rmi_1<\cdots<\rmi_k<\frac{1}{\epsilon}}\int_{\bbR}f(\delta x_0)\dd x_0\sum\limits_{j_0:=0<j_1<\cdots<j_k}\times\\
&\bigg(\prod\limits_{\ell=1}^k\beta_\delta^{j_\ell-j_{\ell-1}}\idotsint\limits_{\quad t_{j_{\ell-1}+1}<\cdots<t_{j_\ell}\in\cI_{\delta,\epsilon}(\rmi_\ell)\atop x_{j_{\ell-1}+1},\cdots,x_{j_\ell}\in\bbR}\prod\limits_{i=j_{\ell-1}+1}^{j_\ell}g_{t_i-t_{i-1}}(x_i-x_{i-1})\rho(x_i)\dd x_i\dd W_{t_i}\bigg).
\end{split}
\end{equation}

By the independent increments of $W_t$, when computing the second moment of \eqref{def:meso_u}, the cross terms related to different $\rmi_1<\cdots<\rmi_k$ and $\rmi'_1<\cdots<\rmi'_l$ have no contributions. 
By the second moment computation in Subsection \ref{sec:she_L2}, we have an continuous analogue of Lemma \ref{lem:limit_var}, that is, as $\delta\to0$, for fixed $\bI=(\rmi_1,\cdots,\rmi_k)$, the related second moment converges to $\cI_{\epsilon}^f(\bI)$ defined in \eqref{def:cI}. Thus, we can restrict the summation for $\rmi_1,\cdots,\rmi_k$ to $\vAnotriple$ defined in \eqref{def:notriple2} as we have done in Section \ref{S6}.

Next, we approximate $u^{\delta}[f]$ by the coarse-grained model $\Lcg(f,\ind_\bbR|\Theta)$ defined by \eqref{cgmodel} with

\begin{equation}\label{def:she_theta}
\Theta(\vec{\rmi}):=\begin{cases}
\displaystyle
\begin{split}
\frac{\delta}{\sqrt{\epsilon}}&\sum\limits_{k=2}^\infty\beta_\delta^{k}\int_{\bbR}\rho(x_1)\dd x_1\times\\
&\idotsint\limits_{t_1<\cdots<t_k\in\cI_{\delta,\epsilon}(\rmi)\atop x_2,\cdots,x_k\in\bbR}\dd W_{t_1}\prod\limits_{i=2}^k g_{t_i-t_{i-1}}(x_i-x_{i-1})\rho(x_i)\dd x_i\dd W_{t_i}
\end{split}
,&\text{if}~|\vec{\rmi}|=1,\\[55pt]
\displaystyle
\begin{split}
\frac{\delta}{\sqrt{\epsilon}}&\sum\limits_{k=1}^\infty\beta_\delta^{k}\int_{\bbR}\rho(x_1)\dd x_1\times\\
&\idotsint\limits_{t_1<\cdots<t_k\in\cI_{\delta,\epsilon}(\rmi)\atop x_2,\cdots,x_k\in\bbR}\dd W_{t_1}\prod\limits_{i=2}^{k} g_{t_i-t_{i-1}}(x_i-x_{i-1})\rho(x_i)\dd x_i\dd W_{t_i}\times\\
&\sum\limits_{j=1}^\infty\beta_\delta^j\int_{\cI_{\delta,\epsilon}(\rmi')}\int_\bbR g_{s_1-t_k}(y_1-x_k)\rho(y_1)\dd y_1\dd W_{s_1}\times\\
&\idotsint\limits_{s_2<\cdots<s_j\in\cI_{\delta,\epsilon}(\rmi')\atop y_2,\cdots,y_j\in\bbR}\prod\limits_{i=2}^j g_{s_i-s_{i-1}}(y_i-y_{i-1})\rho(y_i)\dd y_i\dd W_{s_i}
\end{split}
,&\text{if}~|\vec{\rmi}|>1,
\end{cases}
\end{equation}
where for $k=1$, the integral is interpreted by $1$.

Suppose that $\Theta(\vec{\rmi})$ and $\Theta(\vec{\rmj})$ are two consecutive coarse-grained disorders. We only need to show that the replacement of $g_{t-s}(y-x)$ by $\frac{\delta}{\sqrt{\epsilon}}g_{\rmj-\rmi'}(0)$ has sufficiently small costs uniformly for all $t\in\cT_{\delta,\epsilon}(\rmj)$ and $s\in\cT_{\delta,\epsilon}(\rmi')$ with $\rmj-\rmi'>K_\epsilon$, and for all $x,y\in\text{supp}(\rho)$. The proof is the same as the \textbf{Step 2} in proving Lemma \ref{lem:cg}, and we omit the details. Hence, $u^{\delta}[f]$ can be made arbitrarily close to $\Lcg(f,\ind_{\bbR})$.


Finally, we estimate the second and fourth moments of $\Theta(\vec{\rmi})$. Once we establish the analogues of Lemma \ref{lem:varTheta} and Lemma \ref{lem:4mmtTheta}, then we have done the proof of Theorem \ref{thm:2+}, since the rest arguments are exactly the same as those in the discrete setting (see Section \ref{S9}).




Indeed, the second moment can be computed as in Subsection \ref{sec:she_L2}. For the fourth moment, similar to Lemma \ref{lem:4mmtTheta}, if we already compute the fourth moment of $\Theta(\vec{\rmi})$ for $|\vec{\rmi}|\leq2$, then for $|\vec{\rmi}|\geq3$, we can decouple $\Theta(\rmi,\rmi)$ and $\Theta(\rmi',\rmi')$ and the result follows immediately. Hence, it only remains to compute the case $|\vec{\rmi}|\leq 2$. Without loss of generality, we only need to treat the case $|\vec{\rmi}|=1$. We explain the methodology below and omit tedious computations.

First note that by the computations in Sections \ref{sec:pre_she}-\ref{sec:she_L2}, we can restrict the number of transition kernels $g_{t_i-t_{i-1}}(x_i-x_{i-1})\rho(x_i)$ in \eqref{def:meso_u} by $(\log\delta^{-1})^2$ for $\delta$ small enough. Consequently, since $\Theta(\vec{\rmi})$ is a block in \eqref{def:meso_u}, we can also apply these restrictions to $\Theta(\vec{\rmi})$. When expanding $\Theta(\vec{\rmi})$, we obtain a summation with at most $(\log\delta^{-1})^8$ terms, where each summand is the product of four multiple stochastic integrals with orders $k_i, i=1,2,3,4$. We merge each summand to a $\sum_{i=1}^4 k_i$-multiple stochastic integral. Then by the property of $\dd W_t$, if a summand contains a $\dd W_t$, $(\dd W_t)^3$ or $(\dd W_t)^4$ term, then it has $0$ expectation (we can switch the expectation and summation since there are only finite many terms). Hence, by comparing to the discrete case, now we have four Brownian motions and we should record the time when exactly two of them visit the support of the mollifier $\rho$ (we may call it ``collision''). 

Next, we can rearrange the summation according to the recording time above, just as we have done in Section \ref{S4} and Section \ref{S8}. Also note that now we only have counterparts of function $U$ in \eqref{eq:bigU} or type $U$ groups in Section \ref{S8}, since other cases therein refer to a collision for at least three Brownian motions. Then the bound similar to Lemma \ref{lem:4mmtTheta} can be obtained by the frameworks in Section \ref{S4} and Section \ref{S8}. The details for computing the collision for two Brownian motions have been given in Section \ref{sec:she_L2} (recall that this collision is closely related to the Dickman subordinator).

\begin{appendix}
\section{Enhanced Lindeberg Principle}\label{A1}
To make this paper more self-contained, we cite \cite[Appendix A]{CSZ21} here.

\begin{assump}[Local dependence]\label{assump:local}
	Let $X=(X_i)_{i\in\bbT}$ be a family of random variables, such that
	
	(1) $\bbE[X_i]=0$ and $\bbE[X_i X_j]=\sigma_{ij}$.
	
	(2) For any $k\in\bbT$, there is an $A_k\subset\bbT$, such that $X_k$ is independent of $(X_i)_{i\in A_k^c}$.
	
	(3) For any $k\in\bbT$ and any $l\in A_k$, there is an $A_{kl}\subset\bbT$ such that $(X_k, X_l)$ is independent of $(X_i)_{i\in A_{kl}^c}$.
	
	The set $(A_k)_{k\in\bbT}$ and $(A_{kl})_{k\in\bbT,l\in A_k}$ will be called dependency neighborhoods of $X=(X_i)_{i\in\bbT}$.
\end{assump}

Let $Z=(Z_i)_{i\in\bbT}\sim (\cN(0,\sigma_{ii}))_{i\in\bbT}$ be a family of Gaussian random variables, which is independent of $X$ but has the same covariance matrix as $X$. For $u,t\in[0,1]$, $k\in\bbT$ and $l\in A_k$, we define
\begin{equation}\label{def:W_tukl}
W_{t,u}^{k,l}:=u\sqrt{t}X^{A_{kl}}+\sqrt{t}A^{A_{kl}^c}+\sqrt{1-t}Z,
\end{equation}
where $X^B$ denotes $X^B_i=X_i\ind_{\{i\in B\}}$ for $B\subset\bbT$. For $s,u,t\in[0,1]$, $k\in\bbT$ and $l\in A_k$, we further define
\begin{equation}\label{def:W_stukl}
W_{s,t,u}^{k,l}:=su\sqrt{t}X^{A_k}+u\sqrt{t}X^{A_{kl}\backslash A_k}+\sqrt{t}X^{A_{kl}^c}+\sqrt{1-t}Z.
\end{equation}

We will apply following lemmas in proving our main results, which are \cite[Lemma A.2-A.4]{CSZ21}. Note that we use $\partial_k$ to denote the partial derivative w.r.t.\ $X_k$.
\begin{lemma}\label{lem:A1}
	Let $X, Z, W_{t,u}^{k,l}$ and $W_{s,t,u}^{k,l}$ be defined as above. Let $h:\bbR^{|\bbT|}\to\bbR$ be a bounded and thrice differentiable function. Then $\bbE[h(X)]-\bbE[h(Z)]=I_1+I_2$,
	where
	\begin{align}
	\label{eq:a1}I_1:=&\frac12\iiint_{[0,1]^3}\sum\limits_{k\in\bbT,l\in A_k,m\in A_{kl}}\\
	&\bbE\Big[X_l X_k X_m(s\ind_{\{m\in A_K\}})+\ind_{\{m\in A_{kl}\backslash A_k\}}\sqrt{t}\partial^3_{klm}h\big(W_{s,t,u}^{k,l}\big)\Big]\dd s\dd t\dd u,\nonumber\\
	\label{eq:a2}I_2:=&-\frac12\iint_{[0,1]^2}\sum\limits_{k\in\bbT,l\in A_k,m\in A_{kl}}\sigma_{kl}\bbE\Big[X_m\sqrt{t}\partial^3_{klm}h\big(W_{t,u}^{k,l}\big)\Big]\dd t\dd u.
	\end{align}
\end{lemma}

\begin{lemma}\label{lem:A2}
	Let $Z=(Z_i)_{i\in\bbT}\sim(\cN(0,\sigma_{ii}))_{i\in\bbT}$ and $\tilde{Z}=(\tilde{Z}_i)_{i\in\bbT}\sim(\cN(0,\tilde{\sigma}_{ii}))_{i\in\bbT}$ be two families of centered Gaussian random variables with covariance matrices $(\sigma_{ij})_{i,j\in\bbT}$ and $(\tilde{\sigma}_{ij})_{i,j\in\bbT}$ respectively. For any bounded and twice differentiable function $h:\bbR^{|\bbT|}\to\bbR$, we have that
	\begin{equation}\label{eq:a3}
	\bbE[h(Z)]-\bbE[h(\tilde{Z})]=I_3:=\frac12\sum_{k,l\in\bbT}(\tilde{\sigma}_{kl}-\sigma_{kl})\int_0^1\bbE\big[\partial^2_{kl}h(W_t)\big]\dd t,
	\end{equation}
	where $W_t:=\sqrt{t}\tilde{Z}+\sqrt{1-t}Z$.
\end{lemma}

\begin{lemma}\label{lem:A3}
	Let $f:\bbR\to\bbR$ be a bounded function with bounded first three derivatives. Assume $h(X):=f(\Lambda(X))$ with
	\begin{equation*}
	\Lambda(X)=\sum\limits_{I\subset\bbT}c_I\prod_{i\in I}x_i,
	\end{equation*}
	a polynomial with coefficients $c_I\in\bbR$. Furthermore, assume that
	\begin{equation}\label{cond:lindeberg}
	\forall~k\in\bbT, l\in A_k, m\in A_{kl},\quad\partial_{km}^2\Lambda=\partial_{lm}^2\Lambda=\partial_{kl}^2\Lambda=0.
	\end{equation}
	Then for $I_1, I_2$ and $I_3$ in \eqref{eq:a1}-\eqref{eq:a3}, we have that
	
	\begin{align}
	\label{eq:a4}|I_1|\leq\frac12\|&f'''\|_\infty\sup_{k\in\bbT}\bbE\big[|X_k|^3\big]\sum\limits_{k\in\bbT,l\in A_k, m\in A_{kl}}\\
	&\sup_{s,t,u}\bbE\Big[\big|\partial_k\Lambda(W_{s,t,u}^{k,l})\big|^3\Big]^{\frac13}\sup_{s,t,u}\bbE\Big[\big|\partial_l\Lambda(W_{s,t,u}^{k,l})\big|^3\Big]^{\frac13}\sup_{s,t,u}\bbE\Big[\big|\partial_m\Lambda(W_{s,t,u}^{k,l})\big|^3\Big]^{\frac13},\nonumber\\
	\label{eq:a5}|I_2|\leq\frac12\|&f'''\|_\infty\sup_{k\in\bbT}\bbE\big[|X_k|^3\big]\sum\limits_{k\in\bbT,l\in A_k, m\in A_{kl}}\\
	\quad&\sup_{t,u}\bbE\Big[\big|\partial_k\Lambda(W_{t,u}^{k,l})\big|^3\Big]^{\frac13}\sup_{t,u}\bbE\Big[\big|\partial_l\Lambda(W_{t,u}^{k,l})\big|^3\Big]^{\frac13}\sup_{t,u}\bbE\Big[\big|\partial_m\Lambda(W_{t,u}^{k,l})\big|^3\Big]^{\frac13},\nonumber\\
	\label{eq:a6}|I_3|\leq\frac12\|&f''\|_\infty\sum\limits_{k\in\bbT,l\in A_k}(\tilde{\sigma}_{kl}-\sigma_{kl})\sup_t\bbE\Big[\big|\partial_k\Lambda(W_t)\big|^2\Big]^{\frac12}\sup_t\bbE\Big[\big|\partial_l\Lambda(W_t)\big|^2\Big]^{\frac12}.
	\end{align}
\end{lemma}

\begin{remark}\label{rmk:vector}
	Lemma \ref{lem:A1}-\ref{lem:A3} can be extended to the vector case. Suppose that in the above lemmas, 
		\begin{equation*}
		h(\cdot)=f(\Lambda_1(\cdot),\cdots,\Lambda_k(\cdot)),
		\end{equation*}
		where $k$ is a finite positive integer and $f:\bbR^k\to\bbR$ is a bounded function with bounded derivatives up to third order. Similar bounds as \eqref{eq:a4}-\eqref{eq:a6} still hold. For \eqref{eq:a4} and \eqref{eq:a5}, we need to replace $\|f'''\|_\infty$ by $\max_{1\leq i,j,l\leq k}\|\partial_{ijl}^3f\|_\infty$, replace the three occurrences of $\Lambda$ by $\Lambda_i,\Lambda_j$ and $\Lambda_l$, and sum over $1\leq i,j,l\leq k$. For \eqref{eq:a6}, we need to replace $\|f''\|_\infty$ by $\max_{1\leq i,j\leq k}\|\partial_{ij}^2f\|_\infty$, replace the two occurrences of $\Lambda$ by $\Lambda_i$ and $\Lambda_j$, and sum over $1\leq i,j\leq k$.
\end{remark}

\section{Errors of Local Limit Theorem}\label{A2}
The following result is well-known (see \cite[Theorem 4.5.3]{IL71}).
\begin{theorem}\label{thm:local_1}
Let $S$ be an irreducible and aperiodic random walk specified in Theorem \ref{thm:1}. Then for any $\kappa>0$ in \eqref{assump:RW},
\begin{equation}\label{eq:local_error1}
\sup_{k\in\bbN}\Big|\sqrt{n}\bP(S_n=k)-\varphi\Big(\frac{k}{\sqrt{n}}\Big)\Big|=O(n^{-\frac{\kappa\wedge1}{2}}),
\end{equation}
where $\varphi$ is the density of the standard Gaussian distribution.

Conversely, if a random walk $S$ satisfies \eqref{eq:local_error1} for $0<\kappa<1$, then $\bE[|S_1|^2\ind_{\{|S_1|>k\}}]=O(k^{-\kappa})$ and $S$ has span $1$.
\end{theorem}

Then we can prove the following refined local limit theorem, which to the best of our knowledge, has not appeared in literature.
\begin{theorem}\label{thm:local_2}
Let $S$ be as specified in Theorem \ref{thm:local_1}. Then for any $\kappa>0$ in \eqref{assump:RW}, there exist constants $\gamma, C>0$, such that uniformly in $n\in\bbN$ and $k\in\bbZ$,
\begin{equation}\label{eq:local_error2}
\Big|\sqrt{n}\bP(S_n=k)-\varphi\Big(\frac{k}{\sqrt{n}}\Big)\Big|\leq\frac{C}{1+|k|^{1+\gamma}}.
\end{equation}
\end{theorem}
\begin{proof}
First, we have that for any $0<\theta<2$,
\begin{equation}\label{eq:local_error3}
\lim\limits_{n\to\infty}\sup\limits_{k\in\bbZ}\bigg|\frac{k}{\sqrt{n}}\bigg|^\theta\Big|\sqrt{n}\bP(S_n=k)-\varphi\Big(\frac{k}{\sqrt{n}}\Big)\Big|=0,
\end{equation}
which is a special case of \cite[Theorem 1]{M85}. We present a proof of \eqref{eq:local_error3} here for the completeness. Denote $K_n(A):=\{k\in\bbZ: k\geq A\sqrt{n}\}$ for $A>0$ and
\begin{equation*}
\Delta_{n,k}^{\theta}:=\bigg|\frac{k}{\sqrt{n}}\bigg|^\theta\Big|\sqrt{n}\bP(S_n=k)-\varphi\Big(\frac{k}{\sqrt{n}}\Big)\Big|.
\end{equation*}
Then
\begin{equation*}
\begin{split}
\sup\limits_{k\in\bbZ}\Delta_{n,k}^\theta\leq&\sup_{k\in K_n(A)}\Delta_{n,k}^\theta+\sup_{k\in K_n(A)^c}\Delta_{n,k}^\theta\\
\leq&\sup_{k\in K_n(A)}\bigg|\frac{k}{\sqrt{n}}\bigg|^{\theta}\sqrt{n}\bP(S_n=k)+\sup_{k\in K_n(A)}\bigg|\frac{k}{\sqrt{n}}\bigg|^{\theta}\varphi\Big(\frac{k}{\sqrt{n}}\Big)\\
&+A^\theta\sup_{k\in K_n(A)^c}\Big|\sqrt{n}\bP(S_n=k)-\varphi\Big(\frac{k}{\sqrt{n}}\Big)\Big|.
\end{split}
\end{equation*}
For any $\epsilon>0$, by choosing $A$ large, the second term is smaller than $\epsilon$ for all $n$. For fixed $A$, the last term is smaller than $\epsilon$ for large enough $n$ by the classic local limit theorem. We are left to treat the first term above. We only need to show that for any $0<\theta<2$,
\begin{equation}\label{eq:local_error4}
\sup_{n\in\bbN}\sup_{k\in\bbZ}\bigg|\frac{k}{\sqrt{n}}\bigg|^\theta\sqrt{n}\bP(S_n=k)<\infty,
\end{equation}
since we can apply \eqref{eq:local_error4} to some $\theta'\in(\theta, 2)$ and note that $\sup_{k\in K_n(A)}|k/\sqrt{n}|^{\theta-\theta'}\leq A^{\theta-\theta'}$ is smaller than $\epsilon$ by choosing $A$ large.

We now prove \eqref{eq:local_error4}. By \cite[Lemma 5.2.2]{IL71}, for any $0<\theta<2$,
\begin{equation}\label{eq:frac_moment}
\sup\limits_{n\in\bbZ}\bE\bigg[\bigg|\frac{S_n}{\sqrt{n}}\bigg|^\theta\bigg]=\sup\limits_{n\in\bbN}\sum\limits_{k\in\bbZ}\bigg|\frac{k}{\sqrt{n}}\bigg|^\theta\bP(S_n=k)<\infty.
\end{equation}
Write
\begin{equation}\label{eq:convolution}
\bigg|\frac{k}{\sqrt{n}}\bigg|^\theta\sqrt{n}\bP(S_n=k)=\sum\limits_{\ell\in\bbZ}\bigg|\frac{k}{\sqrt{n}}\bigg|^\theta\sqrt{n}\bP(S_{\frac{n}{2}}=\ell)\bP(S_{\frac{n}{2}}=k-\ell). 
\end{equation}
For any $\theta>0$, we have $|k|^\theta\leq 2^\theta\big(|\ell|^\theta+|k-\ell|^\theta\big)$. To see this, note that for any $x,y\in\bbR$, $|x+y|^\theta\leq 2^\theta(\max\{|x|,|y|\})^\theta\leq 2^\theta(|x|^\theta+|y|^\theta)$. Hence, by the classic local limit theorem, $\sup_n\sup_k\sqrt{n}\bP(S_{n/2}=k)\leq C$, and thus \eqref{eq:convolution} is bounded by $2^\theta$ times
\begin{equation*}
\begin{split}
&\sum\limits_{\ell\in\bbZ}\bigg(\bigg|\frac{\ell}{\sqrt{n}}\bigg|^\theta+\bigg|\frac{k-\ell}{\sqrt{n}}\bigg|^\theta\bigg)\sqrt{n}\bP(S_{\frac{n}{2}}=\ell)\bP(S_{\frac{n}{2}}=k-\ell)\\
\leq&C\sum\limits_{\ell\in\bbZ}\bigg|\frac{\ell}{\sqrt{n}}\bigg|^\theta\bP(S_{\frac{n}{2}}=\ell)+C\sum\limits_{\ell\in\bbZ}\bigg|\frac{k-\ell}{\sqrt{n}}\bigg|^\theta\bP(S_{\frac{n}{2}}=k-\ell),
\end{split}
\end{equation*}
and then \eqref{eq:local_error4} is proved by \eqref{eq:frac_moment}.

With \eqref{eq:local_error1} and \eqref{eq:local_error3} at hand, we now prove \eqref{eq:local_error2}. It suffices to treat small $\kappa>0$. We show that $\gamma=\frac{\kappa^2}{6}$ is valid for \eqref{eq:local_error2}. We consider two cases: (1) $|k|\leq n^{\frac{1}{2}+\frac{\kappa}{4}}$ and (2) $|k|\geq n^{\frac{1}{2}+\frac{\kappa}{4}}$.

For case (1), note that
\begin{equation*}
\bigg|\frac{k}{\sqrt{n}}\bigg|^{1+\gamma}\leq n^{\frac{\kappa}{4}(1+\frac{\kappa^2}{6})}\ll n^{\frac{\kappa}{2}},
\end{equation*}
and then \eqref{eq:local_error2} follows from \eqref{eq:local_error1}.

For case (2), apply \eqref{eq:local_error3} to $\theta=\kappa$,
\begin{equation*}
\begin{split}
o(1)&=\sup\limits_{|k|\geq n^{\frac{1}{2}+\frac{\kappa}{4}}}\bigg|\frac{k}{\sqrt{n}}\bigg|^{1+\kappa}\bigg|\sqrt{n}\bP(S_n=k)-\varphi\bigg(\frac{k}{\sqrt{n}}\bigg)\bigg|\\
&=\sup\limits_{|k|\geq n^{\frac{1}{2}+\frac{\kappa}{4}}}\frac{|k|^{1+\frac{\kappa^2}{6}}}{\sqrt{n}}\frac{|k|^{\kappa-\frac{\kappa^2}{6}}}{\sqrt{n}^{\kappa}}\bigg|\sqrt{n}\bP(S_n=k)-\varphi\bigg(\frac{k}{\sqrt{n}}\bigg)\bigg|.
\end{split}
\end{equation*}
Note that in case (2),
\begin{equation*}
\frac{|k|^{\kappa-\frac{\kappa^2}{6}}}{\sqrt{n}^{\kappa}}\geq n^{\frac{\kappa^2}{3}-\frac{\kappa^3}{24}}\longrightarrow\infty,\quad\text{as}~n\to\infty. 
\end{equation*}
Hence, we must have
\begin{equation*}
o(1)=\sup\limits_{|k|\geq n^{\frac12+\frac{\kappa}{4}}}\frac{|k|^{1+\frac{\kappa^2}{6}}}{\sqrt{n}}\bigg|\sqrt{n}\bP(S_n=k)-\varphi\bigg(\frac{k}{\sqrt{n}}\bigg)\bigg|,
\end{equation*}
which completes the proof.
\end{proof}

\end{appendix}

\begin{acks}[Acknowledgments]
We thank Francesco Caravenna for very helpful discussion and for pointing out the connection between disordered pinning models and stochastic Volterra equations. We also thank Quentin Berger, Rongfeng Sun and Nikolaos Zygouras for reading the first version of the paper and providing invaluable suggestions. Finally, we thank the anonymous referee who helps us significantly improve the quality of this paper. J.~Yu acknowledges the support of NYU-ECNU Institute of Mathematical Sciences at NYU Shanghai.
\end{acks}

\begin{funding}
R.~Wei is supported by XJTLU Research Development Fund (No.\ RDF-23-01-024) and NSFC (No.\ 12401170).
J.~Yu is supported by National Key R\&D Program of China (No.~2021YFA1002700) and NSFC (No.~12101238 and No.~12271010).
\end{funding}

\bibliographystyle{imsart-number} 
\bibliography{references}       





\end{document}